\documentclass[reqno]{amsart}
\usepackage{bm,amsmath,amsthm,amssymb,mathtools,verbatim,amsfonts,tikz-cd,mathrsfs,mathabx}
\usepackage{enumitem}
\usepackage{float}

\usepackage[all,cmtip]{xy}

\usepackage[hidelinks,colorlinks=true,linkcolor=blue, citecolor=black,linktocpage=true]{hyperref}
\usepackage{graphicx}
\usepackage[alphabetic]{amsrefs}
\usepackage{caption}
\usepackage{morefloats}
\usepackage[top=1.3in, bottom=1.1in, left=1.3in, right=1.3in,marginpar=1in]{geometry}
\usepackage{subfigure}

\counterwithin{figure}{section}
\newtheorem{thm}{Theorem}[section]
\newtheorem{prop}[thm]{Proposition}
\newtheorem{conj}[thm]{Conjecture}
\newtheorem{cor}[thm]{Corollary}
\newtheorem{lem}[thm]{Lemma}

\theoremstyle{definition}
\newtheorem{define}[thm]{Definition}

\theoremstyle{remark}
\newtheorem{rem}[thm]{Remark}

\newcommand{\ve}[1]{\boldsymbol{\mathbf{#1}}}
\newcommand{\R}{\mathbb{R}}

\newcommand{\Z}{\mathbb{Z}}
\newcommand{\N}{\mathbb{N}}

\newcommand{\C}{\mathbb{C}}

\renewcommand{\d}{\partial}
\renewcommand{\subset}{\subseteq}

\renewcommand{\tilde}{\widetilde}

\newcommand{\iso}{\cong}

\DeclareMathOperator{\ad}{{ad}}

\DeclareMathOperator{\codim}{{codim}}

\DeclareMathOperator{\gr}{{gr}}

\DeclareMathOperator{\Hom}{{Hom}}

\DeclareMathOperator{\id}{{id}}

\DeclareMathOperator{\Int}{{int}}

\DeclareMathOperator{\Spin}{{Spin}}

\DeclareMathOperator{\Span}{{Span}}

\DeclareMathOperator{\Sym}{{Sym}}

\DeclareMathOperator{\Tot}{{Tot}}

\renewcommand{\injlim}{\varinjlim}

\DeclareMathOperator{\Kom}{\mathbf{Kom}}

\newcommand{\bC}{\mathbb{C}}

\newcommand{\bE}{\mathbb{E}}
\newcommand{\bF}{\mathbb{F}}

\newcommand{\bR}{\mathbb{R}}

\newcommand{\bT}{\mathbb{T}}

\newcommand{\cA}{\mathcal{A}}
\newcommand{\cB}{\mathcal{B}}
\newcommand{\cC}{\mathcal{C}}

\newcommand{\cG}{\mathcal{G}}

\newcommand{\cR}{\mathcal{R}}

\newcommand{\cU}{\mathcal{U}}

\newcommand{\CF}{\mathit{CF}}

\newcommand{\SHI}{\mathit{SHI}}

\newcommand{\HFh}{\widehat{\mathit{HF}}}

\newcommand{\HFK}{\mathit{HFK}}

\newcommand{\SFH}{\mathit{SFH}}
\newcommand{\KHI}{\mathit{KHI}}
\newcommand{\KHM}{\mathit{KHM}}

\newcommand{\xs}{\ve{x}}
\newcommand{\ys}{\ve{y}}

\newcommand{\as}{\ve{\alpha}}
\newcommand{\bs}{\ve{\beta}}

\renewcommand{\a}{\alpha}
\renewcommand{\b}{\beta}
\newcommand{\g}{\gamma}
\newcommand{\dt}{\delta}
\newcommand{\veps}{\varepsilon}

\usepackage{leftidx}

\DeclareMathOperator{\Cone}{{Cone}}

\numberwithin{equation}{section}

\newcommand{\scF}{\mathscr{F}}
\newcommand{\scG}{\mathscr{G}}
\newcommand{\scH}{\mathscr{H}}
\newcommand{\scK}{\mathscr{K}}
\newcommand{\scM}{\mathscr{M}}
\newcommand{\scN}{\mathscr{N}}

\newcommand{\scS}{\mathscr{S}}

\newcommand{\G}{\Gamma}

\newcommand{\tildeotimes}{\mathrel{\tilde{\otimes}}}

\title{Connected sums and directed systems in  knot Floer homologies}

\author{Sudipta Ghosh}

\author{Ian Zemke}

\address{Department of Mathematics\\University of Notre Dame\\  South Bend, IN, USA}
\email{sghosh7@nd.edu}

\address{Department of Mathematics\\Princeton University\\  Princeton, NJ, USA}
\email{izemke@math.princeton.edu}

\begin{document}
\maketitle

\begin{abstract}
We prove a number of fundamental properties about instanton knot Floer homology. Our arguments rely on general properties of sutured Floer theories and apply also in the Heegaard Floer and monopole Floer settings, where many of our results were already known. Our main result is the connected sum formula for instanton knot Floer homology. An extension of this result proves the oriented skein exact triangle for the minus version of instanton knot Floer homology. Finally, we derive a new model of the minus version of instanton knot Floer homology, which takes the form of a free, finitely generated chain complex over a polynomial ring, as opposed to a direct limit. This construction is new to all of the Floer theories. We explore these results also in the context of Heegaard Floer theory as well.
\end{abstract}


\tableofcontents

\section{Introduction} \label{sec:intro}

In this paper, we study a diverse collection of knot invariants which arise from different Floer theories. Ozsv\'{a}th and Szab\'{o} \cite{OSKnots} defined an important package of knot invariants, usually referred to as \emph{knot Floer homology}. These invariants take the form of a vector space $\widehat{\HFK}(Y,K)$, and an $\bF[U]$-module $\HFK^-(Y,K)$. In this paper, we take $\bF$ to be $\Z/2 \Z$ or $\C$.
 Kronheimer and Mrowka \cite{KMSutures} define monopole and instanton versions of knot Floer homology, which we denote by
\[
\widehat{\KHM}(Y,K) \quad \text{and} \quad \widehat{\KHI}(Y,K).
\]
The above invariants are finitely dimensional vector spaces. More recently, work of Li \cite{LiLimits} has extended instanton and monopole theories to give a minus invariant
\[
\KHI^-(Y,K)\quad \text{and} \quad \KHM^-(Y,K).
\]
These take the form of a finitely generated module over $\C[U]$ or $\cR[U]$, where $\cR$ is the Novikov ring. 

The perspective of Li is modeled on earlier work of Etnyre, Vela-Vick, and Zarev \cite{EVZLimit} and Golla \cite{GollaLimit}, who described the $\bF[U]$-module $\HFK^-(Y,K)$ as a direct limit of the sutured invariants for $Y\setminus \nu(K)$ with various longitudinal slopes. These authors use the foundational work of Honda, Kazez and Mati\'{c} \cite{HKMTQFT} \cite{HKMSutured} on contact invariants in Juh\'{a}sz's Heegaard Floer theory for sutured manifolds \cite{JDisks}.

By mimicking these constructions in the setting of sutured instanton and monopole Floer homology, Li defines the instanton and monopole limit knot Floer homologies $\KHI^-(Y,K)$ and $\KHM^-(Y,K)$, which are modules over $\C[U]$ or $\cR[U]$, where $\cR$ is the Novikov ring.

There is a well-known conjecture that
\[
\widehat{\KHI}(Y,K)\iso \widehat{\HFK}(Y,K)\otimes \C\quad\text{and} \quad  \KHI^-(Y,K)\iso \HFK^-(Y,K)\otimes \C.
\]
 This is an analog of the more general conjecture of Kronheimer and Mrowka \cite{KMSutures}*{Conjecture~7.24} that 
\begin{equation} \label{Conj: KM}
 I^{\#}(Y) \cong \HFh(Y) \otimes \mathbb{C}.
 \end{equation}
 
The connection between monopole Floer homology and Heegaard Floer homology is comparably better understood. It was proved by Baldwin and Sivek \cite{BSequivalence} following works of  Taubes \cite{TaubesEquivalence}, Colin, Ghiggini and Honda \cite{CGHHF=ECH0}, \cite{CGHHF=ECH1}, \cite{CGHHF=ECH2}, \cite{CGHHF=ECH3} that there is an isomorphism of $\mathcal{R}$-modules
\begin{equation}
  \label{eqn:khm-hfk}
  \widehat{\KHM} (Y, K; \mathcal{R}) \cong \widehat{\HFK} (Y, K, \mathbb{Z}_2) \otimes \mathcal{R},
\end{equation}
 where $\mathcal{\cR}$ denotes the mod 2 Novikov ring.

In this paper, we prove a variety of novel properties about sutured and direct limit versions of knot Floer homology. Our focus is on proving the connected sum formula in instanton theory, but we will prove several additional results, some of which are novel even in the Heegaard Floer setting. Our goal is to give unified proofs which work in all three settings (instanton, monopole and Heegaard Floer).

\subsection{Results}

The central result of this paper is the following theorem:

\begin{thm}\label{thm: main} Suppose that $K_1,K_2$ are null-homologous knots in 3-manifolds $Y_1$ and $Y_2$, respectively. Then the instanton limit Floer homology $\KHI^-(Y_1\# Y_2, K_1\#K_2)$ is isomorphic to the derived tensor product
\[
\KHI^-(Y_1,K_1)\tildeotimes_{\C[U]} \KHI^-(Y_2,K_2)
\]
as an $\C[U]$-module.
\end{thm}

Note that original the proof of the connected sum formula in Heegaard Floer theory \cite{OSKnots} uses the existence of a special Heegaard diagram for $K_1\# K_2$. An alternate strategy is to consider the cobordism map for the pair-of-pants cobordism relating $(Y_1,K_1)\sqcup (Y_2,K_2)$ and $(Y_1\# Y_2,K_1\# K_2)$, as in \cite{ZemConnectedSums}. Neither of these strategies is available in sutured instanton theory.

Our proof of Theorem~\ref{thm: main} has two main ingredients. The first is a novel manipulation in terms of decomposing surfaces in sutured Floer theories. The second ingredient is a collection of very general results about generalized eigenspaces and filtered chain complexes which are necessary to perform the iterated mapping cone construction in sutured instanton theory. 

We recall that the iterated mapping cone construction gives a very general set of tools in Floer theory which relate the Floer homologies of different Dehn surgeries of a link $L\subset Y$. The origins of the construction date to Floer's original Dehn surgery exact triangle \cite{FloerExactTriangle} for instanton homology. In Heegaard Floer theory, Ozsv\'{a}th and Szab\'{o} \cite{OSDoubleBranchedCover} proved a link surgery spectral sequence, which is the natural generalization of the exact triangle to links. The construction has been extended to most other Floer theories by a number of authors. See \cite{KMUnkotDetector} \cite{BloomSpectralSequence} \cite{ScadutoOddKh}. 

The aforementioned works establish the iterated mapping cone construction in instanton Floer homology. However they are insufficient for sutured instanton theory, since sutured instanton theory is constructed from the ordinary instanton complex by taking certain generalized eigenspaces of several operators. In Section~\ref{sec:hypercubes-eigenspaces}, we describe a very general framework for eigenspaces and filtered complexes, which establish suitable versions of the iterated mapping cone complex in the context of sutured instanton homology. We expect this to be of independent interest.

We additionally illustrate our proofs in the context of Heegaard Floer theory, using Heegaard diagrams. In the process, we also prove a folklore result of J. Rasmussen relating the Honda-Kazez-Mati\'{c} maps for contact 2-handles with certain 4-dimensional cobordism maps. See Section~\ref{sec:folklore}. This conjecture was used by Golla \cite{GollaLimit} in his construction of the direct limit knot Floer homology.

Our techniques can also be used to study skein exact triangles in knot homology theories. The hat versions of Heegaard, instanton and monopole knot Floer homology are known to satisfy the oriented skein relation by the work of Ozsv\'ath and Szab\'o \cite{OSKnots}*{Theorem~10.2} and Kronheimer and Mrowka \cite{KMskein}*{Theorem~3.1}. In Heegaard Floer homology, Ozsv\'ath and Szab\'o \cite{OSskein}*{Theorem~1.1} showed that the minus version of knot Floer homology, $\HFK^-(Y, K)$ also satisfies the skein exact triangle.
In the instanton setting, it was heretofore an open question whether the minus version satisfied the Skein exact triangle \cite{GhoshLiWongTau}*{Question~1.20}. Using similar techniques to our proof of the connected sum formula, we prove the following:

\begin{thm}
Let $L_{+}$, $L_-$ and $L_0$ be oriented links in $Y$ as above. If the two strands meeting at the distinguished
crossing in $L_+$ belong to the same component, then there is an exact triangle
\[ 
\cdots \KHI^-(L_{+}) \xrightarrow{f} \KHI^-(L_-) \xrightarrow{g} \KHI^-(L_0)\tildeotimes_{\C[U_1,U_2]} \mathbb{C}[U_1,U_2]/(U_1-U_2) \xrightarrow{h} \KHI^-(L_{+})\cdots
.\]
If they belong to different components, then there is an exact triangle
\[
\cdots \KHI^-(L_{+}) \xrightarrow{f} \KHI^-(L_-) \xrightarrow{g} \KHI^-(L_0)\otimes_{\C} W_s \xrightarrow{h} \KHI^-(L_{+})\cdots
\]
where $W_s$ is the graded module given by 
\[
    W_s= 
\begin{cases}
    \mathbb{C},& \text{if s=0, 1 } \\
    0,              & \text{otherwise.}
\end{cases}
\]
Furthermore, $f$, $g$ and $h$ preserve the Alexander grading. Here, we equip $\bC[U_1,U_2]$ with the Alexander grading determined by setting $A(1)=0$ and $A(U_1)=A(U_2)=-1$. 
\end{thm}
Note that we may give $\KHI^-(L_0)\tildeotimes_{\C[U_1,U_2]} \C[U_1,U_2]/(U_1-U_2)$ a natural action of $\C[U]$ by having $U$ act by $U_1$ (or $U_2$) on $\C[U_1,U_2]/(U_1-U_2)$. At the moment, we can only prove Theorem~\ref{thm: skein} on the level of graded groups. We make the following conjecture:
\begin{conj}
\label{conj:intro}
There is a natural $\C[U_1,\dots, U_{\ell-1}]$-module structure on $\KHI^-(L_0)\tildeotimes_{\C[U_1,U_2]} \C[U_1,U_2]/(U_1-U_2)$ so that the map $f$, $g$ and $h$ are $\C[U_1,\dots, U_{\ell-1}]$-equivariant.
\end{conj}

Algebraically, one should expect that the module structure on the derived tensor product $\KHI^-(L_0)\tildeotimes_{\C[U_1,U_2]} \C[U_1,U_2]/(U_1-U_2)$ would require a choice of $A_\infty$-module structure on $\KHI^-(L_0)$ over $\C[U_1,U_2]$. We note that defining an $A_\infty$-module structure is essentially equivalent to defining a model of $\KHI(L_0)$ as a free, finitely generated chain complex over $\C[U_1,U_2]$. We propose such a free, finitely generated chain complex in Section~\ref{sec:limit-free}. We expect extensions of our techniques to prove Conjecture~\ref{conj:intro} for this model, however we do not approach it for the sake of brevity.

\subsection{A limit free definition of $\KHI^-(Y,K)$}

The  definition of limit knot Floer homology in terms of sutured manifolds, following \cite{GollaLimit} and \cite{EVZLimit}, is as a direct limit over the positive bypass maps of the knot complement. For some technical arguments, this is sometimes impractical since the resulting group is not the homology of a chain complex and is instead a direct limit of homologies. Our proof of the connected sum formula inspires a new model of $\KHI^-(K)$ which is the homology of a finitely generated, free chain complex over $\C[U]$, and does not involve taking the direct limit. Indeed, if $\Gamma_n$ is an $n$-framed longitude of $K$, we define the mapping cone complex
$\mathit{CKI}^-(K,n)$ to be the mapping cone complex
\[
\Cone\left(\begin{tikzcd}[column sep=2cm] \SFH(S^3\setminus \nu(K),\Gamma_n)\otimes_{\C} \C[U]\ar[r, "\phi^-|\id-\phi^+|U"] & \SFH(S^3\setminus \nu(K),\Gamma_{n+1})\otimes_{\C} \C[U]
\end{tikzcd}\right)
\]
which is a finitely generated free chain complex over $\C[U]$. 

\begin{thm}
\label{thm:free-model} There is an isomorphism of $\C[U]$-modules
\[
H_*\mathit{CKI}^-(K,n)\iso \KHI^-(K)
\]
for all $n\in \Z$. 
\end{thm}

We note that Theorem~\ref{thm:free-model} was previously unknown in the Heegaard Floer setting. Using bordered Floer theory \cite{LOTBordered}, we prove in Section~\ref{sec:free-models-Heegaard} that it also holds in Heegaard Floer theory, and compare the above to the description in \cite{EVZLimit}.

We note that the above construction extends in a rather straightforward manner to give a free, finitely generated chain complex for links in a 3-manifold, which we briefly describe in Section~\ref{sec:free-models-for-links}.

\subsection{Organization}
In Section~\ref{sec:sutured-background}, we review the contact gluing map and associate the contact 2-handle map with the 4-dimensional surgery maps. In Section \ref{sec:3}, we briefly review the definition of $\SFH (M, \g)$, bypass attachment and the definition $\HFK^-(Y, K)$ from the direct limit perspective. In Section~\ref{sec:4}, we discuss the algebra related to the direct limit of projectively transitive system and the naturality of direct limit knot Floer homologies. In Section~\ref{sec:5}, we discuss several tensor products which appear in our paper. In Section~\ref{sec:connected-sum-HF}, we describe our proof of the connected sum formula in the context Heegaard Floer theories. In Section~\ref{sec:7}, we prove the connected sum formula for instanton theory. In Section~\ref{sec:8}, consider gradings. In Section~\ref{sec:10}, we prove the skein exact triangle for minus instanton knot Floer homology. In Section~\ref{sec:U-action}, we prove that the connected sum formula respects the $\C[U]$-module structure.  In Section~\ref{sec:bypass-exact-triangle}, we apply our techniques to compute the maps in the bypass exact triangle.

\subsection{Acknowledgments}

The authors are indebted to John Baldwin, Zhenkun Li, Steven Sivek, and Fan Ye for pointing out a number of errors and gaps and for numerous helpful discussions and suggestions for the article. The first author thanks John Baldwin for posing the question that led to this work. The first author is grateful to Zhenkun Li for insightful discussion which helped improve some of the arguments in Section~\ref{sec:8}. The majority of this project was completed when the first author was a graduate student at Louisiana State University and the first author would like to thank his advisors David Shea Vela-Vick and Scott Baldridge for their support and encouragement. The first author was supported by his advisor David Shea Vela-Vick’s NSF Grant 1907654. The second author is supported by NSF grant 2204375. The first author is grateful to the Max Planck Institute for Mathematics for hosting the first author for the bulk of this work.

\section{Background on sutured Floer theories}
\label{sec:sutured-background}

In this section, we recall Juh\'{a}sz's sutured Floer homology \cite{JDisks} and its counterparts in monopole and instanton Floer homologies due to Kronheimer and Mrowka \cite{KMSutures}.

We begin by recalling Gabai's notion of a sutured manifold \cite{Gabai}, adapted to the setting of sutured Floer homology:
\begin{define}\label{define: sutured manifold} 
A \emph{sutured manifold}  $(M,\g)$ consists of a compact, oriented $3$-manifold $M$ with a closed, embedded 1-manifold $\g\subset\partial M$. Write $R(\g)$ for $\d M\setminus \g$. Additionally, we assume the following are satisfied:
\begin{enumerate}
    \item $M$ has no closed components.
    \item $\pi_{0}(\gamma)$ surjects onto $\pi_{0}(\partial M)$.
\item The components of $\d M\setminus \g$ are partitioned into two submanifolds, $R_+(\g)$ and $R_-(\g)$, such that
\[
\d R_+(\g)=\g\quad \text{and} \quad \d R_-(\g)=-\g.
\]
\end{enumerate}
We say that $(M,\g)$ is \emph{balanced} if $\chi(R_+(\gamma))=\chi(R_-(\gamma))$.
\end{define}

\subsection{Sutured Heegaard Floer homology}

We now briefly recall Juh\'{a}sz's construction of sutured Heegaard Floer homology \cite{JDisks}. If $(\Sigma,\as,\bs)$ is a Heegaard diagram for the balanced sutured manifold $(M,\g)$, then Juh\'{a}sz considers the two tori
\[
\bT_{\a}:=\a_1\times \cdots \times \a_n\quad \text{and }\quad \bT_{\b}:=\b_1\times \cdots \times \b_n
\]
inside of the symmetric product $\Sym^n(\Sigma)$. Juh\'{a}sz defines the sutured chain complex $\CF(\Sigma,\as,\bs)$ to be the free $\bF$-vector space generated by intersection points $\xs\in \bT_{\a}\cap \bT_{\b}$. The differential on $\CF(\Sigma,\as,\bs)$ counts pseudo-holomorphic disks of Maslov index 1. The group $\SFH(M,\g)$ is defined to be the homology of $\CF(\Sigma,\as,\bs)$.

\subsection{Sutured instanton Floer homology}\label{subsec:SHI}
In this subsection, we recall the definition of sutured instanton Floer homology, due to Kronheimer and Mrowka \cite{KMSutures}.

Let $(Y, \a)$ consist of a closed, oriented 3-manifold $Y$ and a closed,
oriented 1-manifold $\a \subset Y$ intersecting some embedded surface transversally in an odd number
of points. We associate the following data to this pair:

\begin{itemize}
\item A Hermitian line bundle $w \to Y$ with $c_1(w)$ Poincar\'{e} dual to $\a$;
\item A $U(2)$ bundle $E \to Y$ equipped with an isomorphism $\Theta \colon  \Lambda^2{E} \to w$.
\end{itemize}
The \emph{instanton Floer homology} $I_*(Y )_\a$ is the Morse homology of the Chern-Simons functional
on the space $\cB = \cC/\cG$ of $SO(3)$ connections on $\ad(E)$, modulo determinant 1 gauge transformations. It is a relatively $\mathbb{Z}/8$-graded $\bC$-module.

For each even-dimensional class $\Sigma \in H_d(Y )$, there is an operator
\[\mu(\Sigma) : I_*(Y )_{\a} \to I_{{* + d}- 4}(Y )_\a.\]
This operator is described in \cite{KMSutures}*{Section~7.2}. See \cite{DonaldsonKronheimerBook}*{Section~5.1.2} for a detailed background.

\begin{rem} \label{rem:localize-mu(R)}
The operator $\mu(\Sigma)$ may be localized in the following sense. We may delete a neighborhood of $\Sigma\times \{0\}$ inside of $Y\times \R$ to obtain a cobordism from $Y\sqcup (S^1\times \Sigma)$ to $Y$. This gives a map $F$ from $I_*(Y)_{\alpha}\otimes I_*(S^1\times \Sigma)_{\alpha}$ to $I_*(Y)_{\alpha}$. If $y$ is the image of the cobordism map applied to $D^2\times \Sigma$, then $\mu(\Sigma)\colon I_*(Y)_{\alpha}\to I_*(Y)_{\alpha}$ is equal to the cobordism map $F$ applied to $I_*(Y)_{\alpha}\otimes \mu(\Sigma)(y)$. See the proof of \cite{KMSutures}*{Corollary~7.2}.
\end{rem}

\begin{thm}[\cite{KMSutures}*{Corrollary~7.2}]\label{thm: eigen value}
Suppose $R$ is a closed surface in $Y$ of positive genus with $\#(\a \cap R)$ odd. Then the simultaneous eigenvalues of the operators $\mu(R)$ and $\mu(pt)$ on $I_*(Y )_\a$ belong to a subset of the pairs \[(i^r 2k, (-1)^r 2)\] for $0 \le r \le 3$ and $0 \le k \le g(R) -1$, where $i = \sqrt{-1}$.
\end{thm}

\begin{define}[\cite{KMskein}*{Definition~7.3}] 
Given $Y$, $\a$, and $R$ as in Theorem~\ref{thm: eigen value}, let \[I_*(Y|R)_\a \subset I_*(Y)_\a  \] be the generalized $(2g(R) -2, 2)$-eigenspace of $(\mu(R), \mu(pt))$ on $I_*(Y)_\a$.
\end{define}

\begin{define}[\cite{KMSutures}*{Section~7.4}]\label{define: closure}
 Suppose $(M, \gamma)$ is a balanced sutured manifold. Let $T$ be a connected, compact, and oriented surface of genus at least 2 such that $\d T$ and $\g$ have the same number of components.
  We say a manifold $\tilde{M}$ is a \emph{preclosure} of $(M, \g)$ if it is obtained by gluing $\d T\times [-1,1]\subset T\times [-1,1]$ to a regular neighborhood $A(\g)\subset \d M$ via an orientation reversing map
  \[
  h\colon A(\g)\to \d T\times [-1,1]
  \]
   satisfying $h(\d T\times \{\pm 1\})\subset R_{\pm}(\g)$. The boundary $\d\tilde{M}$ has two diffeomorphic components, which we denote by $\tilde{R}_+$ and $\tilde{R}_-$. Here, $\tilde{R}_+$ (resp. $\tilde{R}_-$) is the component which intersects $R_+(\g)$ (resp. $R_-(\g)$) non-trivially. We glue $\tilde{R}_+$ to $\tilde{R}_-$ via any orientation reversing diffeomorphism which fixes a point $q \in T$ to form a manifold $Y$. We let $R$ be the image of $\tilde{R}_+=\tilde{R}_-$ in $Y$ and $\a$ be the image of the arc $q \times [-1, 1]$. We say that the pair $(Y,R, \a)$ is a \emph{closure} of $(M, \gamma)$. The
genus of $R$ is called the \emph{genus} of the closure $(Y, R)$. Note that $\a$ is a simple closed curve which intersects $R$ transversely at one point and $\a \cap \Int(M) = \emptyset$

\end{define}

\begin{define}[\cite{KMSutures}*{Definition~7.10}]\label{define: Sutured instanton Floer homology}
For a closure $(Y, R, \a)$ with genus at least 2, define \[ \SHI (M, \g) \coloneqq  I_*(Y|R)_\a\]
\end{define}

\subsection{Sutured monopole Floer homology}\label{subsec:SHM}
In this subsection, we briefly recall the definition of sutured monopole homology, also due to Kronheimer and Mrowka \cite{KMSutures}.

\begin{define}[\cite{KMSutures}*{Definition~4.3}]\label{define: Sutured monopole Floer homology}
For a closure $(Y, R)$ with genus at least 2, define
\[\mathit{SHM} (M, \g) := \bigoplus_{\mathfrak{s}\in\mathfrak{S}(Y|R)}\widecheck{\mathit{HM}}(Y,\mathfrak{s}),\]
 where
  \[ \mathfrak{S}(Y|R)=\left\{\mathfrak{s}\in \Spin^c(Y)|~c_1(\mathfrak{s})[R]=2g(R)-2,~\widecheck{\mathit{HM}}(Y,\mathfrak{s})\neq 0\right\}.
   \]
\end{define}

\subsection{Contact handle maps in Floer theories}

Contact handles were introduced by Giroux \cite{Giroux}*{Definition~5.5}. Baldwin and Sivek \cite{BSContact} defined contact handle maps in instanton and monopole Floer theories, and Juh\'{a}sz and the second author defined parallel maps in Heegaard Floer theory \cite{JZContactHandles}. The construction of Juh\'{a}sz and the second author gives an alternate description of the contact gluing map of Honda, Kazez and Mati\'{c} \cite{HKMTQFT}.

\subsubsection{Contact 1-handles}
Suppose $D_+$ and $D_- $ are disjoint embedded disks in $\partial M$ which each intersect $\gamma$ in a single, properly embedded arc. We view $D^2\times [-1,1]$ as being a standard contact 3-ball.  We glue $D^2 \times [-1, 1]$ to $(M, \gamma)$ by diffeomorphisms 
\[
 D^2 \times \{-1\} \rightarrow D_- \quad\text{ and }\quad D^2 \times \{1\} \rightarrow D_+  
 \]
  which preserve and reverse orientations, respectively, and identify the dividing sets with the sutures. Then we round corners as shown in Figure~\ref{fig:58}. Let $(M', \gamma')$ be the resulting sutured manifold.

On the Heegaard Floer side, if $(\Sigma, \as, \bs)$ is a Heegaard diagram for $(M, \gamma)$, then we obtain a Heegaard diagram for $(M', \gamma')$ by gluing a band to $\partial \Sigma$ where the feet of the contact 1-handle are attached. We add no new $\as$ or $\bs$ curves; see Figure~ \ref{fig:60}. The map on chain complexes is the tautological map induced by the inclusion of Heegaard surfaces. We also call this map $C_{h_1}$.  See \cite{JZContactHandles}*{Section~3.3.2} for more details.

On the instanton and monopole Floer side, Baldwin and Sivek \cite{BSContact} constructed a closure $(Y, R)$ of $(M, \gamma)$ which is also a closure of $(M', \gamma')$. See \cite{BSContact}*{Section~4.2} for more details. Baldwin and Sivek define the contact 1-handle map in sutured monopole and instanton theories
\[
C_{h_1} : \SHI(M, \gamma) \rightarrow  \SHI(M', \gamma')
\] 
to be the identity map.

\begin{figure}[ht]
\centering
\begingroup%
  \makeatletter%
  \providecommand\color[2][]{%
    \errmessage{(Inkscape) Color is used for the text in Inkscape, but the package 'color.sty' is not loaded}%
    \renewcommand\color[2][]{}%
  }%
  \providecommand\transparent[1]{%
    \errmessage{(Inkscape) Transparency is used (non-zero) for the text in Inkscape, but the package 'transparent.sty' is not loaded}%
    \renewcommand\transparent[1]{}%
  }%
  \providecommand\rotatebox[2]{#2}%
  \newcommand*\fsize{\dimexpr\f@size pt\relax}%
  \newcommand*\lineheight[1]{\fontsize{\fsize}{#1\fsize}\selectfont}%
  \ifx\svgwidth\undefined%
    \setlength{\unitlength}{258.35635449bp}%
    \ifx\svgscale\undefined%
      \relax%
    \else%
      \setlength{\unitlength}{\unitlength * \real{\svgscale}}%
    \fi%
  \else%
    \setlength{\unitlength}{\svgwidth}%
  \fi%
  \global\let\svgwidth\undefined%
  \global\let\svgscale\undefined%
  \makeatother%
  \begin{picture}(1,0.35958559)%
    \lineheight{1}%
    \setlength\tabcolsep{0pt}%
    \put(0,0){\includegraphics[width=\unitlength,page=1]{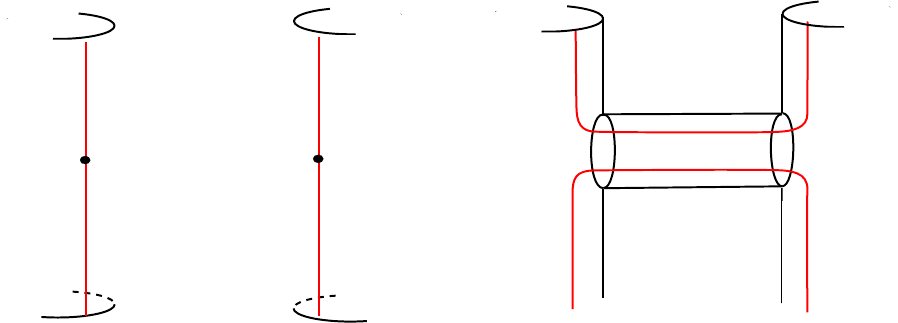}}%
    \put(0.02842279,0.16843161){\color[rgb]{0,0,0}\makebox(0,0)[lt]{\lineheight{1.25}\smash{\begin{tabular}[t]{l}$P$\end{tabular}}}}%
    \put(0.37726681,0.17073529){\color[rgb]{0,0,0}\makebox(0,0)[lt]{\lineheight{1.25}\smash{\begin{tabular}[t]{l}$Q$\end{tabular}}}}%
    \put(0,0){\includegraphics[width=\unitlength,page=2]{fig58.pdf}}%
  \end{picture}%
\endgroup%

\caption{Left: the sutured manifold $(M, \gamma)$ with two points $P$ and $Q$ on the suture. Right: the manifold $(M',\g')$ obtained by attaching a contact 1-handle with feet at $P$ and $Q$.}
\label{fig:58}
\end{figure}

\begin{figure}[ht]
\centering
\begingroup%
  \makeatletter%
  \providecommand\color[2][]{%
    \errmessage{(Inkscape) Color is used for the text in Inkscape, but the package 'color.sty' is not loaded}%
    \renewcommand\color[2][]{}%
  }%
  \providecommand\transparent[1]{%
    \errmessage{(Inkscape) Transparency is used (non-zero) for the text in Inkscape, but the package 'transparent.sty' is not loaded}%
    \renewcommand\transparent[1]{}%
  }%
  \providecommand\rotatebox[2]{#2}%
  \newcommand*\fsize{\dimexpr\f@size pt\relax}%
  \newcommand*\lineheight[1]{\fontsize{\fsize}{#1\fsize}\selectfont}%
  \ifx\svgwidth\undefined%
    \setlength{\unitlength}{260.56465333bp}%
    \ifx\svgscale\undefined%
      \relax%
    \else%
      \setlength{\unitlength}{\unitlength * \real{\svgscale}}%
    \fi%
  \else%
    \setlength{\unitlength}{\svgwidth}%
  \fi%
  \global\let\svgwidth\undefined%
  \global\let\svgscale\undefined%
  \makeatother%
  \begin{picture}(1,0.31023009)%
    \lineheight{1}%
    \setlength\tabcolsep{0pt}%
    \put(0,0){\includegraphics[width=\unitlength,page=1]{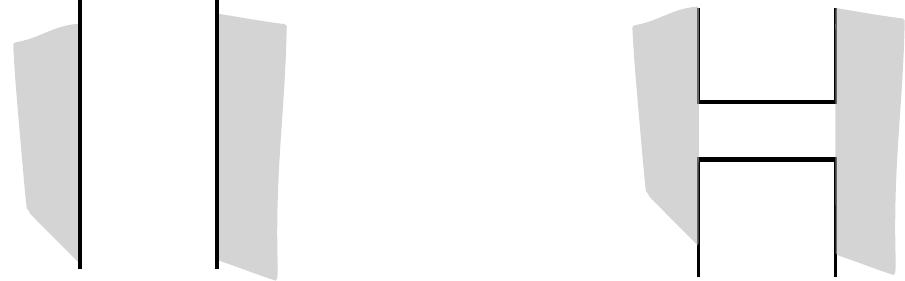}}%
    \put(0.07655228,0.13467898){\makebox(0,0)[rt]{\lineheight{1.25}\smash{\begin{tabular}[t]{r}$\Sigma_1$\end{tabular}}}}%
    \put(0.25042821,0.13467898){\makebox(0,0)[lt]{\lineheight{1.25}\smash{\begin{tabular}[t]{l}$\Sigma_2$\end{tabular}}}}%
    \put(0,0){\includegraphics[width=\unitlength,page=2]{fig60.pdf}}%
    \put(0.88222912,0.0965967){\makebox(0,0)[rt]{\lineheight{1.25}\smash{\begin{tabular}[t]{r}$\Sigma_0$\end{tabular}}}}%
  \end{picture}%
\endgroup%

\caption{The diagrams used in the definition of the contact handle maps}
\label{fig:60}
\end{figure}

\subsubsection{Contact 2-handles} 
Suppose $\mu$ is an embedded curve in $\partial M$ which intersects $\gamma$ transversely in two points. Let
$A(\mu)$ be an annular neighborhood of $\mu$ intersecting $\gamma$ in two arcs. We glue $D^2 \times [-1, 1]$ to $(M, \gamma)$ by an orientation-reversing diffeomorphism
\[
\partial D^2 \times [-1, 1] \to A(\gamma).
\]
We alter the sutures by adding a single arc to both $D^2\times \{-1\}$ and $D^2\times \{1\}$, rounding corners as in Figure~\ref{fig:59}. Let $(M', \gamma')$ be the resulting sutured manifold.

We first discuss the contact 2-handle on the Heegaard Floer side by Juh\'{a}sz and the second author \cite{JZContactHandles}*{Section~3.3.3}. Let $(\Sigma, \as, \bs)$ be an admissible diagram of $(M, \gamma)$. Suppose that $\mu\subset \d M$ is the attaching cycle of the 2-handle. We assume that $\mu$ intersects $\g$ in two points, $p_1$, $p_2$ $\in M$. Write $\lambda_\pm =\mu \cap R_{\pm}(\gamma)$. Since $R_+(\g)$ is obtained by surgering $\Sigma$ on $\bs$, we may view $\lambda_+$ as an embedded arc on $\Sigma$ which avoids $\bs$. Similarly $\lambda_-$ may be viewed as a curve on $\Sigma$ which avoids $\as$.
 We construct a Heegaard diagram $(\Sigma', \as \cup {\alpha_0}, \bs \cup {\beta_0})$ for $(M',\g')$ in Figure \ref{fig:61} . The surface
$\Sigma'$
is obtained by adding a band $B$ to the boundary of $\Sigma$ at $p_1$ and $p_2$. The curve $\alpha_0$ is obtained by
concatenating the curve $\lambda_-$ with an arc in the band. The curve $\beta_0$
is obtained by concatenating the
curve $\lambda_+$ with an arc in the band. We assume that $\alpha_0$ and $\beta_0$ intersect in a single point in the band
(and possibly other places outside the band). We define the contact 2-handle map
\[ C_{h^2}: \SFH (\Sigma, \as, \bs) \to \SFH (\Sigma'
, \as \cup {\alpha_0
}, \bs \cup {\beta_0})\]
 by the formula
\[C_{h_2} (\textbf{x}) := \textbf{x} \times c,\] where $c \in \alpha_0 \cap \beta_0$
is the intersection point in the band. See \cite{JZContactHandles}*{Section~3.3} for more details.

Now now turn our attention to the instanton and monopole Floer side. Baldwin and Sivek \cite{BSContact}*{Section~4.2} define the contact 2-handle map as follows.  Let $\mu'$ be the knot obtained by pushing $\mu$ into $M$ slightly. Suppose $(M_0, \gamma_0)$ be the manifold obtained from $(M, \gamma)$
by 0-surgery along $\mu'$ with respect to the boundary framing.  The sutured
manifold $(M_0, \gamma_0)$ can be obtained from $(M', \gamma')$ by attaching a contact 1-handle. Since $\mu'$ is contained in the interior of $M$, we can construct a cobordism
between the closures of $(M, \gamma)$ and $(M_0, \gamma_0)$ by attaching a 4-dimensional 2-handle along $\mu'$. Baldwin and Sivek define 
\[
C_{h_2} := {C_{h_1}}^{-1} \circ F_W: \SFH (M, \gamma) \to \SFH (M', \gamma'),
\]
 where $F_{W}$ is the cobordism map for attaching a 2-handle along $\mu'$.

\begin{figure}[ht]
\centering
\begingroup%
  \makeatletter%
  \providecommand\color[2][]{%
    \errmessage{(Inkscape) Color is used for the text in Inkscape, but the package 'color.sty' is not loaded}%
    \renewcommand\color[2][]{}%
  }%
  \providecommand\transparent[1]{%
    \errmessage{(Inkscape) Transparency is used (non-zero) for the text in Inkscape, but the package 'transparent.sty' is not loaded}%
    \renewcommand\transparent[1]{}%
  }%
  \providecommand\rotatebox[2]{#2}%
  \newcommand*\fsize{\dimexpr\f@size pt\relax}%
  \newcommand*\lineheight[1]{\fontsize{\fsize}{#1\fsize}\selectfont}%
  \ifx\svgwidth\undefined%
    \setlength{\unitlength}{332.79167055bp}%
    \ifx\svgscale\undefined%
      \relax%
    \else%
      \setlength{\unitlength}{\unitlength * \real{\svgscale}}%
    \fi%
  \else%
    \setlength{\unitlength}{\svgwidth}%
  \fi%
  \global\let\svgwidth\undefined%
  \global\let\svgscale\undefined%
  \makeatother%
  \begin{picture}(1,0.32362225)%
    \lineheight{1}%
    \setlength\tabcolsep{0pt}%
    \put(0,0){\includegraphics[width=\unitlength,page=1]{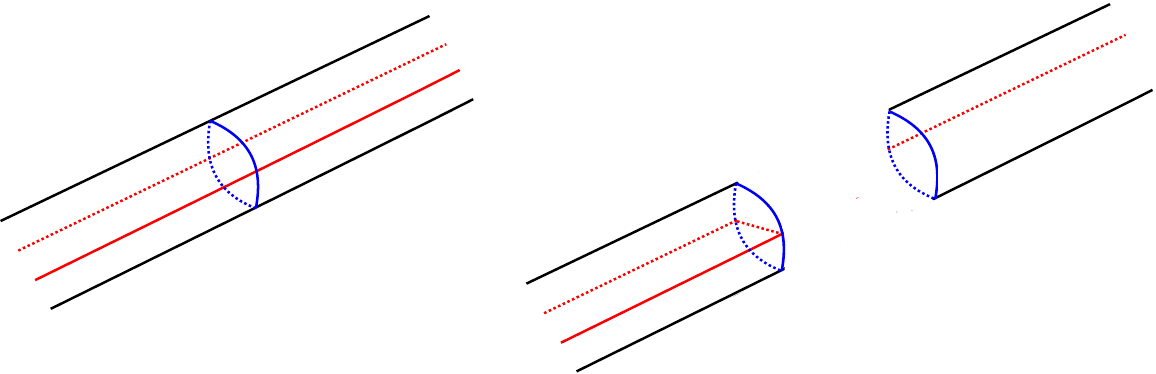}}%
    \put(0.1911689,0.10752936){\color[rgb]{0,0,0}\makebox(0,0)[lt]{\lineheight{1.25}\smash{\begin{tabular}[t]{l}$\mu$\end{tabular}}}}%
    \put(0.04278151,0.18782378){\color[rgb]{0,0,0}\makebox(0,0)[lt]{\lineheight{1.25}\smash{\begin{tabular}[t]{l}$M$\end{tabular}}}}%
    \put(0.69511087,0.14361247){\color[rgb]{0,0,0}\makebox(0,0)[lt]{\lineheight{1.25}\smash{\begin{tabular}[t]{l}$M'$\end{tabular}}}}%
    \put(0,0){\includegraphics[width=\unitlength,page=2]{fig59.pdf}}%
    \put(0.36964155,0.27147302){\color[rgb]{0,0,0}\makebox(0,0)[lt]{\lineheight{1.25}\smash{\begin{tabular}[t]{l}$\gamma$\end{tabular}}}}%
    \put(0.38243758,0.02482837){\color[rgb]{0,0,0}\makebox(0,0)[lt]{\lineheight{1.25}\smash{\begin{tabular}[t]{l}$\gamma'$\end{tabular}}}}%
    \put(0,0){\includegraphics[width=\unitlength,page=3]{fig59.pdf}}%
  \end{picture}%
\endgroup%

\caption{Left: the sutured manifold $(M, \gamma)$ and the curve $\mu \subset \partial M$  that intersects $\gamma$ at two points. Right: the contact 2-handle attachment along the curve $\mu$. In both pictures, we view $M$ and $M'$ as being the exterior of the surfaces shown.
 }
\label{fig:59}
\end{figure}

\begin{figure}[ht]
\centering
\begingroup%
  \makeatletter%
  \providecommand\color[2][]{%
    \errmessage{(Inkscape) Color is used for the text in Inkscape, but the package 'color.sty' is not loaded}%
    \renewcommand\color[2][]{}%
  }%
  \providecommand\transparent[1]{%
    \errmessage{(Inkscape) Transparency is used (non-zero) for the text in Inkscape, but the package 'transparent.sty' is not loaded}%
    \renewcommand\transparent[1]{}%
  }%
  \providecommand\rotatebox[2]{#2}%
  \newcommand*\fsize{\dimexpr\f@size pt\relax}%
  \newcommand*\lineheight[1]{\fontsize{\fsize}{#1\fsize}\selectfont}%
  \ifx\svgwidth\undefined%
    \setlength{\unitlength}{272.30253911bp}%
    \ifx\svgscale\undefined%
      \relax%
    \else%
      \setlength{\unitlength}{\unitlength * \real{\svgscale}}%
    \fi%
  \else%
    \setlength{\unitlength}{\svgwidth}%
  \fi%
  \global\let\svgwidth\undefined%
  \global\let\svgscale\undefined%
  \makeatother%
  \begin{picture}(1,0.3032114)%
    \lineheight{1}%
    \setlength\tabcolsep{0pt}%
    \put(0,0){\includegraphics[width=\unitlength,page=1]{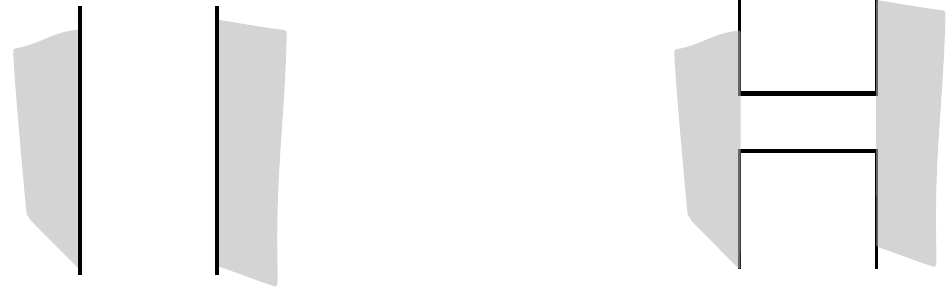}}%
    \put(0.07325242,0.12887348){\makebox(0,0)[rt]{\lineheight{1.25}\smash{\begin{tabular}[t]{r}$\Sigma_1$\end{tabular}}}}%
    \put(0.2396333,0.12887348){\makebox(0,0)[lt]{\lineheight{1.25}\smash{\begin{tabular}[t]{l}$\Sigma_2$\end{tabular}}}}%
    \put(0,0){\includegraphics[width=\unitlength,page=2]{fig61.pdf}}%
  \end{picture}%
\endgroup%

\caption{The Heegaard diagrams used in the definition of the contact 2-handle map. On the left is $(\Sigma,\as,\bs)$ and on the right is $(\Sigma',\as\cup \{\a_0\}, \bs\cup \{\b_0\})$.}
\label{fig:61}
\end{figure}

\subsection{On the Heegaard Floer contact 2-handle map}
\label{sec:folklore}

In this section, we show that the definition of the Heegaard Floer contact 2-handle map from \cite{JZContactHandles} is equivalent to the definition given by Baldwin and Sivek \cite{BSContact}*{Section~4.2} in sutured monopole and instanton Floer homology. 

This verifies a formula for the Honda--Kazez--Mati\'{c} gluing map for a contact bypass attachment which was conjectured by Rasmussen (cf. \cite{GollaLimit}*{Theorem~5.5}). To the best of our knowledge, a proof of Rasmussen's formula has not appeared in the literature.

\begin{lem}\label{lem:ctct-2-handle=triangles}
Let $\mu\subset \d M$ be the attaching cycle of a contact 2-handle and let $\mu'$ be the framed knot obtained by pushing $\mu$ into $M$ slightly, equipped with framing parallel to $\d M$. Write $(M_0,\g_0)$ for the surgered sutured manifold. There is a product disk $D\subset M_0$ which has boundary $\mu$, obtained by capping $\mu$ with the core of the 2-handle attached along $\mu'$. Then
\[
C_{h^2}=(C_{h^1})^{-1}\circ F_{W}
\]
where $F_{W}$ is the 4-dimensional 2-handle map for $\mu'$, as defined in \cite{JCob}, and $(C_{h^{1}})^{-1}$ is the inverse of the contact 1-handle map corresponding to the product disk $D$.
\end{lem}
\begin{proof}The 2-handle map $F_{W}$ may be computed by first stabilizing the Heegaard diagram $(\Sigma,\as,\bs)$, and then counting holomorphic triangles as shown in Figure~\ref{fig:6}. In the terminology of \cite{JZContactHandles}*{Section~2.2}, the stabilization in Figure~\ref{fig:6} is a \emph{compound stabilization}.  The natural map $\sigma(\xs)=\xs\times c_0$ coincides with the map from naturality by \cite{JZContactHandles}*{Proposition~2.2}. Here, $c_0$ is the canonical intersection point shown in Figure~\ref{fig:6}. 

 After attaching a 2-handle along $\mu'$, the Heegaard diagram for $(M_0,\g_0)$ coincides with the diagram of a contact 2-handle attachment followed by a contact 1-handle attachment. The inverse of the contact 1-handle map corresponds to deleting the extra band. The equality in the statement $C_{h^2}=(C_{h^1})^{-1}\circ F_{W}$  is proven by an easy model triangle count, shown in Figure~\ref{fig:6}.
\end{proof}

\begin{figure}[ht]
\begingroup%
  \makeatletter%
  \providecommand\color[2][]{%
    \errmessage{(Inkscape) Color is used for the text in Inkscape, but the package 'color.sty' is not loaded}%
    \renewcommand\color[2][]{}%
  }%
  \providecommand\transparent[1]{%
    \errmessage{(Inkscape) Transparency is used (non-zero) for the text in Inkscape, but the package 'transparent.sty' is not loaded}%
    \renewcommand\transparent[1]{}%
  }%
  \providecommand\rotatebox[2]{#2}%
  \newcommand*\fsize{\dimexpr\f@size pt\relax}%
  \newcommand*\lineheight[1]{\fontsize{\fsize}{#1\fsize}\selectfont}%
  \ifx\svgwidth\undefined%
    \setlength{\unitlength}{220.06841904bp}%
    \ifx\svgscale\undefined%
      \relax%
    \else%
      \setlength{\unitlength}{\unitlength * \real{\svgscale}}%
    \fi%
  \else%
    \setlength{\unitlength}{\svgwidth}%
  \fi%
  \global\let\svgwidth\undefined%
  \global\let\svgscale\undefined%
  \makeatother%
  \begin{picture}(1,1.42449909)%
    \lineheight{1}%
    \setlength\tabcolsep{0pt}%
    \put(0,0){\includegraphics[width=\unitlength,page=1]{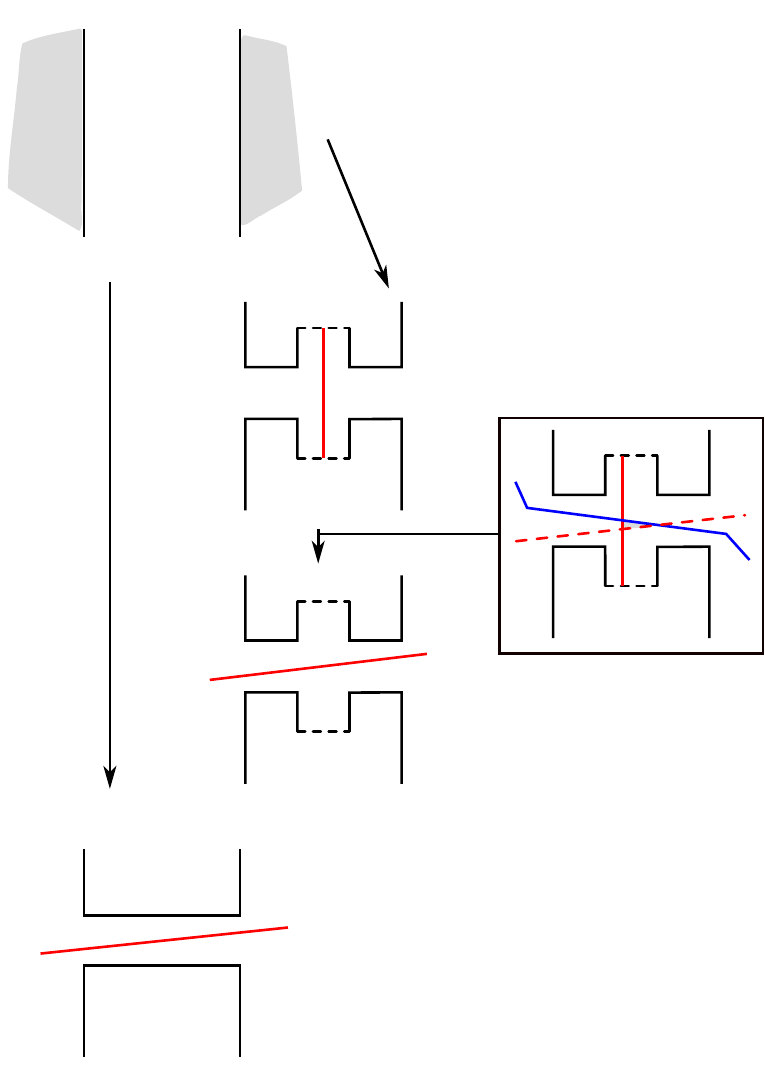}}%
    \put(0.09654184,1.18812589){\makebox(0,0)[rt]{\lineheight{1.25}\smash{\begin{tabular}[t]{r}$\Sigma$\end{tabular}}}}%
    \put(0.32742193,1.18812589){\makebox(0,0)[lt]{\lineheight{1.25}\smash{\begin{tabular}[t]{l}$\Sigma$\end{tabular}}}}%
    \put(0,0){\includegraphics[width=\unitlength,page=2]{fig6.pdf}}%
    \put(0.48830386,1.15134238){\makebox(0,0)[lt]{\lineheight{1.25}\smash{\begin{tabular}[t]{l}$\sigma$\end{tabular}}}}%
    \put(0.48830386,0.24027971){\makebox(0,0)[lt]{\lineheight{1.25}\smash{\begin{tabular}[t]{l}$(C_{h^{1}})^{-1}$\end{tabular}}}}%
    \put(0.8278978,0.58712728){\makebox(0,0)[t]{\lineheight{1.25}\smash{\begin{tabular}[t]{c}$F_{W(L)}$\end{tabular}}}}%
    \put(0.12968656,0.63439174){\makebox(0,0)[rt]{\lineheight{1.25}\smash{\begin{tabular}[t]{r}$C_{h^2}$\end{tabular}}}}%
  \end{picture}%
\endgroup%

\caption{Realizing the contact 2-handle map $C_{h^2}$ as a composition of a compound stabilization $\sigma$, followed by a 4-dimensional 2-handle map $F_{W}$, followed by the inverse of the contact 1-handle map $(C_{h^1})^{-1}$. A holomorphic triangle of the 2-handle map is indicated in the right-most box.}
\label{fig:6}
\end{figure}

\section{Background on directed systems in knot Floer homologies}\label{sec:3}

In this section, we recall the construct of knot Floer homology using direct limits. The construction was first given by Etnyre, Vela-Vick and Zarev \cite{EVZLimit}, and independently Golla \cite{GollaLimit}, using Heegaard Floer theory. Later, Li \cite{LiLimits} gave an analogous construction in Monopole and instanton theory, using contact gluing maps of Baldwin and Sivek \cite{BSContact}.

\subsection{Bypass attachment}
Bypass attachments were introduced by Honda \cite{HondaClassI}. Suppose $\alpha$ is an embedded arc in $\partial M$ which intersects $\g$ in three points. A bypass move along $\a$ replaces $\g$ with a new set of sutures $\g'$ which differ from $\g$ in a neighborhood of $\a$, as shown in Figure~\ref{fig:62}. Let $(M', \g')$ be the resulting sutured manifold. If $\gamma$ is the dividing set of a contact structure $\xi$ on $M$, then a bypass move is achieved by attaching an actual bypass along $\a$, as defined by Honda \cite{HondaClassI}.

\begin{figure}[ht]
\begingroup%
  \makeatletter%
  \providecommand\color[2][]{%
    \errmessage{(Inkscape) Color is used for the text in Inkscape, but the package 'color.sty' is not loaded}%
    \renewcommand\color[2][]{}%
  }%
  \providecommand\transparent[1]{%
    \errmessage{(Inkscape) Transparency is used (non-zero) for the text in Inkscape, but the package 'transparent.sty' is not loaded}%
    \renewcommand\transparent[1]{}%
  }%
  \providecommand\rotatebox[2]{#2}%
  \newcommand*\fsize{\dimexpr\f@size pt\relax}%
  \newcommand*\lineheight[1]{\fontsize{\fsize}{#1\fsize}\selectfont}%
  \ifx\svgwidth\undefined%
    \setlength{\unitlength}{220.01507486bp}%
    \ifx\svgscale\undefined%
      \relax%
    \else%
      \setlength{\unitlength}{\unitlength * \real{\svgscale}}%
    \fi%
  \else%
    \setlength{\unitlength}{\svgwidth}%
  \fi%
  \global\let\svgwidth\undefined%
  \global\let\svgscale\undefined%
  \makeatother%
  \begin{picture}(1,0.37486585)%
    \lineheight{1}%
    \setlength\tabcolsep{0pt}%
    \put(0,0){\includegraphics[width=\unitlength,page=1]{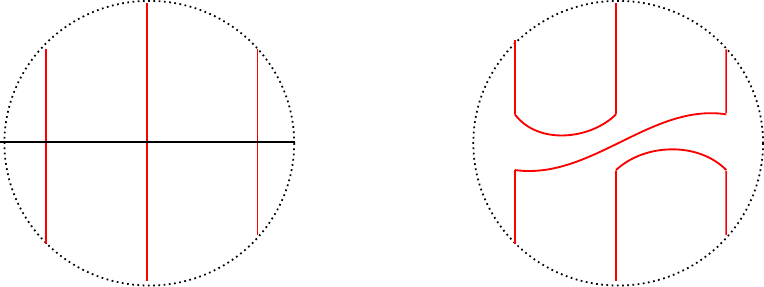}}%
    \put(0.10925886,0.15723093){\color[rgb]{0,0,0}\makebox(0,0)[lt]{\lineheight{1.25}\smash{\begin{tabular}[t]{l}$\alpha$\end{tabular}}}}%
    \put(0,0){\includegraphics[width=\unitlength,page=2]{fig62.pdf}}%
  \end{picture}%
\endgroup%

\caption{Attaching a bypass along an arc $\a \subset \partial M$ where the orientation of $\partial M$ is pointing out.}
\label{fig:62}
\end{figure}

 We note that bypass attachment is equivalent to attaching a contact 1-handle along disks
in $\partial M$ centered at the endpoints of $\a$ and then attaching a contact 2-handle along a circle that is the union of $\a$ and the core arc on the attached 1-handle. See Ozbagci \cite{Ozbagci}*{Section~3} for an exposition. Topologically, the 1-handle and the 2-handle form a canceling pair, so the diffeomorphism type of the 3-manifold does not change. However, the contact structure is changed, and the suture $\g$ is replaced by $\g'$.

 \subsection{Construction of direct limit knot Floer homology}

We now recall the description of knot Floer homology as a direct limit \cite{EVZLimit} \cite{GollaLimit} \cite{LiLimits}. Let $\SFH(S^3\setminus {\nu(K)},\Gamma_{n})$ denote the sutured knot exterior where the suture $\G_n$ denotes two oppositely oriented curves of slope $\lambda + n\mu$ on $\partial {Y(K)}$. Here $\lambda$ is the Seifert longitude and $\mu$ is the meridian of the knot $K$. The bypass attachments described above induce contact gluing maps 
\[
\phi^+_{n},\phi^-_{n}\colon \SFH(S^3\setminus {\nu(K)}, \G_n) \to \SFH(S^3\setminus {\nu(K)}, \G_{n+1}) .
\]
We sometimes abbreviate $\phi^{\pm}_{n}$ by $\phi^{\pm}$.
We view the complexes $\SFH(S^3\setminus \nu(K),\Gamma_n)$ as forming a transitive system over $n\in \N$, where the transition maps are the positive bypass maps $\phi^+$. The minus knot Floer complex $\HFK^-(K)$ is defined to be the direct limit over these maps.

Using functoriality of the contact gluing map, it is straightforward to see that $\phi^+\circ \phi^-\simeq \phi^-\circ \phi^+$. In particular, one may define an action $U$ on the direct limit as $\phi^-$.

Additionally, the group $\HFK^-(Y, K)$ possesses a $\Z$-valued Alexander grading, as described by \cite{LiLimits}. With respect to this grading, $U$ shifts grading by $-1$. See  Section~\ref{sub: grading} for more details.

\section{Background on transitive systems}\label{sec:4}

In this section, we provide some algebraic background  on transitive systems, focusing on constructions which appear in knot Floer homology. 

\subsection{Transitive systems}

Throughout, we let $\ve{k}$ denote a ring.

\begin{define}\label{def:transitive-system}
 Let $(I,\le )$ be a directed, partially ordered set. A \emph{transitive system} of $\ve{k}$-modules over $(I,\le )$ consists of a collection of $\ve{k}$-modules $G_i$, together with a distinguished homomorphism $\phi_{i\to j} \colon G_i\to G_j$ whenever $i\le j$, such that the following are satisfied:
 \begin{enumerate}
 \item $\phi_{i\to i}=\id_{G_i}$
 \item $\phi_{j\to k}\circ \phi_{i\to j}=\phi_{i\to k}$.
 \end{enumerate}
 
 A \emph{projectively transitive system} of $\ve{k}$-modules consists of a collection of $G_i$ and $\phi_{i\to j}$, as above, except we only require $\phi_{i\to i}\doteq \id_{G_i}$ and $\phi_{j\to k}\circ \phi_{i\to j}=\phi_{i\to k}$, where $\doteq$ denotes equality up to multiplication by a unit in $\ve{k}$.
\end{define}

If $(G_i,\phi_{i\to j})$ is a transitive system of $\ve{k}$-modules over $(I,\le)$, the direct limit is defined as the following $\ve{k}$-modules
\begin{equation}
\underrightarrow{\lim} (G_i,\phi_{i\to j}):=\frac{\bigoplus_{i\in I} G_i}{\Span (x_i-\phi_{i\to j}(x_i): x_i\in G_i, i\le j)}.\label{eq:direct-limit-def}
\end{equation}

We will also be interested in \emph{transitive systems of chain complexes}. Definition~\ref{def:transitive-system} adapts to this setting. We require the morphisms $\phi_{i\to j}$ to be chain maps (specified on the chain level, and not only up to chain-homotopy), and the relations $\phi_{i\to i}=\id_{G_i}$ and $\phi_{j\to k}\circ \phi_{i\to j}=\phi_{i\to k}$ must hold at the chain level (not just up to chain homotopy). The limit is itself a chain complex.

We recall the following notion of morphism between transitive systems:

\begin{define}
\label{def:morphism-transitive}
Suppose $(I,\le)$ and $(J,\le)$ are two directed systems, and $\scN=(N_i,\phi_{i\to i'})$ and $\scM=(M_j,\psi_{j\to j'})$ are two transitive systems over $I$ and $J$ (respectively) of $\ve{k}$-modules. A \emph{morphism} $\scF$ from $\scM$ to $\scN$ is the following data:
\begin{enumerate}
\item A monotonically non-decreasing function $f\colon I\to J$.
\item For each $i\in I$, a morphism of $\ve{k}$-modules $F_{i,f(i)}\colon M_i\to N_{f(i)}$ such that if $i\le i'$ then
\[
F_{i',f(i')}\circ \phi_{i\to i'}=\psi_{f(i)\to f(i')}\circ F_{i,f(i)}.
\]
\end{enumerate}
\end{define}

\subsection{Projectively transitive systems}

We need the following generalization of the notion of a transitive system (see \cite{BSNaturality}):

\begin{define} Let $(I,\le)$ be a directed, partially ordered set, and let $\ve{k}$ denote a ring. A \emph{projectively transitive system of $\ve{k}$-modules} over $(I,\le)$ consists of a collection of $\ve{k}$-modules $M_i$, together with distinguished $\ve{k}$-module homomorphisms $\phi_{i\to j}\colon M_i\to M_j$ whenever $i\le j$, satisfying the following:
\begin{enumerate}
\item $\phi_{i\to i}\doteq \id_{M_i}$
\item $\phi_{j\to k}\circ \phi_{i\to j}\doteq \phi_{i\to k}$.
\end{enumerate}
\end{define}
In the above definition, $\doteq$ denotes equality up to multiplication by a unit in $\ve{k}$.

Similarly, we define a \emph{morphism} of projectively transitive systems by replacing `$=$' in Definition~\ref{def:morphism-transitive} with `$\doteq$'.

\subsection{Limits and projectively transitive systems}
\label{sec:callibrated-limits}

Given an projectively transitive system $\scM$ over $(I,\le)$ of $\ve{k}$-modules which is suitably nicely behaved (in the sense of the subsequent definition), we  describe a way of constructing a projectively transitive system $\underline{\scM}$, also over $(I,\le)$ which has the property that each morphism $\phi_{i\to j}$ is an isomorphism. We call $\underline{\scM}$ the direct limit, and we call the procedure to produce $\underline{\scM}$ as \emph{callibrating} the limit.

We restrict to the case that $\ve{k}$ is a field and consider only projectively transitive systems which satisfy the following:

\begin{define} Let $(I,\le)$ be a directed, pre-ordered set, and let $\scM=(M_i,\phi_{i\to j})$ be a projectively transitive system of $\ve{k}$-modules. We say that $\scM$ \emph{weakly non-trivial} if each $M_i$ is not the zero module, and also each map $\phi_{i\to j}\colon M_i\to M_j$ is non-zero.
\end{define}

In practice, it is sufficient also to consider the case when there is a some non-empty subset $J\subset I$, which is closed from above, such that $\scM|_J$ is weakly non-trivial.

If $i_0\in I$ is fixed, let
\[
U(i_0)=\{i\in I: j\ge i_0\}.
\]
We now define a genuine transitive system $\scM_{(i_0)}=(M_i,\psi_{j\to k})$ over the directed set $(U(i_0),\le)$, as follows. The groups are $M_i$ are the same as for $\scM$ (except restricted to for $j\in U(i_0)$). We define $\psi_{i_0\to j}=\phi_{i_0\to j}$ for all $j\in U(i_0)$. For arbitrary $j,k\in U(i_0)$ such that $j\le k$, we define $\psi_{j\to k}$ to be $\alpha\cdot \phi_{j\to k}$ where $\alpha\in \ve{k}^\times$ is the unique element which makes the following diagram commute:
\[
\begin{tikzcd}[labels=description]
&M_{i_0} \ar[dl, "\phi_{i_0\to j}"] \ar[dr, "\phi_{i_0\to k}"]&\\
M_{j}\ar[rr, "\alpha\cdot \phi_{j\to k}"]&& M_k
\end{tikzcd}
\]
Clearly this construction gives a genuine transitive system over $\ve{k}$.  We think of the morphisms as being ``callibrated'' versions of the morphisms from $\scM$.

 We write $\hat{M}_{i_0}$ for the limit (in the ordinary sense). 

If $i_0,j_0\in I$ and $i_0\le j_0$, we now define a map 
\[
\hat{\theta}_{i_0\to j_0}\colon \hat{M}_{i_0}\to \hat{M}_{j_0},
\]
as follows. First,  if $m\in U(i_0)$ and $n\in U(j_0)$, we define a map $\theta_{m\to n}^{i_0,j_0}\colon M_m\to M_n$ by setting $\theta_{m\to n}^{i_0,j_0}=\alpha\cdot \phi_{m\to n}$ where $\alpha\in \ve{k}^\times$ is the unique unit which makes the following diagram commute:
\[
\begin{tikzcd}[labels=description, column sep=1.5cm, row sep=1.5 cm]
 M_{i_0}\ar[r, "\phi_{i_0\to m}"]\ar[d, "\phi_{i_0\to j_0}"]& M_{m} \ar[d,"\alpha\cdot  \phi_{m\to n}"]\\
M_{j_0}\ar[r, "\phi_{j_0\to n}"] &M_n
\end{tikzcd}
\]
The maps $\theta^{i_0,j_0}_{m\to n}$ induce a map on direct limits, for which we write $\hat{\theta}_{i_0\to j_0}$.

\begin{lem} Suppose that $\scM$ is a projectively transitive system which is weakly non-trivial.
\begin{enumerate}
\item The maps $\hat{\theta}_{i_0\to j_0}$ make $(\hat{M}_{i_0})_{i_0\in I}$ into a projectively transitive system.
\item Each  $\hat{\theta}_{i_0\to j_0}$ is an isomorphism of $\ve{k}$-modules.
\end{enumerate}
\end{lem}
\begin{proof}The first claim is obtained by performing a diagram chase to the following diagram:
\[
\begin{tikzcd}[labels=description]
M_{i_0} \ar[dr, "\phi_{i_0\to j_0}"]\ar[rr, "\phi_{i_0\to \ell}"] \ar[dd, "\delta\cdot \phi_{i_0\to k_0}"]  &&M_\ell \ar[dd, "\alpha\cdot \phi_{\ell\to n}",pos=.65]\ar[dr, "\beta \cdot \phi_{\ell\to m}"]\\
&M_{j_0}\ar[dl, "\phi_{j_0\to k_0}"] \ar[rr,crossing over, "\phi_{j_0\to m}"]&&M_m \ar[dl, "\gamma \cdot \phi_{m\to n}"]\\
M_{k_0} \ar[rr, "\phi_{k_0\to n}"]&&M_{n}
\end{tikzcd}
\]
where $\a,\b,\g,\dt\in \ve{k}^\times$.
We assume commutativity on all but the back face. An easy diagram chase shows that the back face commutes, which gives $\hat{\theta}_{j_0\to k_0}\circ \hat{\theta}_{i_0\to j_0}=\delta \cdot \hat{\theta}_{i_0\to k_0}$.

 We consider the second claim. We construct a map $\hat{\tau}_{j_0\to i_0}$ in the opposite direction, as follows. The map $\hat{\tau}_{j_0\to i_0}$ is defined on $M_n$ for $n\ge j_0$ to be $\alpha\cdot \id_{M_n}$ for the unique $\alpha\in \ve{k}^\times$ which makes the following diagram commute:
 \[
 \begin{tikzcd}[labels=description, column sep=1.5cm, row sep=1.5 cm]
 M_{i_0}\ar[r, "\phi_{i_0\to n}"]\ar[d, "\phi_{i_0\to j_0}"]& M_{n} \\
M_{j_0}\ar[r, "\phi_{j_0\to n}"] &M_n \ar[u,"\alpha \cdot \id"]
\end{tikzcd}.
\] 

An easy diagram chase shows that $\hat{\tau}_{j_0\to i_0}\circ \hat{\theta}_{i_0\to j_0}$ is the map on direct limits which sends an $x\in M_n$ to $\phi_{n\to m}(x)$ for some $m\ge n$ and $m\ge j_0$ (the choice of $m$ may depend on $n$; but the induced element $\phi_{n\to m}(x)$ is independent of $m$). In particular $\hat{\tau}_{j_0\to i_0}\circ \hat{\theta}_{i_0\to j_0}$ is the identity map on direct limits.
\end{proof}

\begin{rem}
 The above construction also works in the case that $\scM$ is a projectively transitive system of chain complexes. Forming the limit, in the same manner as for $\ve{k}$-modules, yields a direct limit which is a projectively transitive system of chain complexes where each morphism is a chain isomorphism. 
\end{rem}

Note that the above callibration operation works on the level of morphisms, as well, as long as we restrict to morphisms such that each $F_{i,f(i)}$ is non-zero. Indeed if $\scF=(f,F_{i,f(i)})$ is a morphism of projectively transitive systems, we may define a genuine morphism of transitive systems 
\[
\scF_{(i_0)}=(f, F_{i,f(i)}^{(i_0)})\colon \scM_{(i_0)}\to \scN_{(f(i_0))}
\]
 as follows.  We define $F_{i_0,f(i_0)}^{(i_0)}=F_{i_0,f(i_0)}$. For $j\ge i_0$, we set $F_{j,f(j)}^{(i_0)}=\alpha\cdot F_{j,f(j)}$ where $\alpha\in \ve{k}^\times$ is the unique unit such that the following diagram commutes
 \[
 \begin{tikzcd}[column sep=1.5cm, row sep=1.5cm, labels=description]
 M_{i_0}
 	\ar[r, "F_{i_0,f(i_0)}"]
 	\ar[d, "\psi_{i_0\to j}"]
 &
 N_{f(i_0)}
 	\ar[d, "\psi_{f(i_0)\to f(j)}"]
 \\
 M_{j}
 	\ar[r, "\alpha \cdot F_{j,f(j)}"] & N_{f(j)}
 \end{tikzcd}
 \]

\subsection{Naturality and direct limit knot Floer homologies}

In this section, we apply the formalism from the previous section to sutured and monopole limit homologies. We recall that these theories are defined in terms of projectively transitive systems. The following is based on the naturality result of Baldwin and Sivek \cite{BSNaturality} for sutured monopole and instanton homology:

\begin{prop} 
By taking limits in the sense of Section~\ref{sec:callibrated-limits}, we may view $\KHI^-(K)$ as the limit of a genuine  transitive system over $\N$ of $\C$-vector spaces. On $\KHI^-(K)$, there is an action of $U$ which is well-defined up to multiplication by a unit. The transition maps commute with the action of $U$ up to overall multiplication by a unit.
\end{prop}

We observe the following basic lemma:

\begin{lem}\label{lem:rescale-module} Suppose that $M$ is a module over $\ve{k}[U_1,\dots, U_r]$ for $\ve{k}$ a field, which admits a relative $r$-component Alexander grading, such that $A_i(U_j)=-\delta_{i,j}$, where $\delta_{i,j}$ is the Kronecker delta function. Let, $\a_1,\dots, \a_r\in \ve{k}^{\times}$ and let $\tilde{M}$ be the $\ve{k}[U_1,\dots, U_r]$ module obtained by having $U_i$ act by $\a_i U_i$. Then $M$ and $\tilde{M}$ are isomorphic as $\ve{k}[U_1,\dots, U_r]$-modules.
\end{lem}
\begin{proof} We define a map $\phi$ from $M$ to $\tilde{M}$ via the formula
\[
\phi(\xs)=\a_1^{-A_1(\xs)}\cdots \a_r^{-A_r(\xs)} \cdot\xs.
\]
Since $U_i$ has Alexander grading $-\delta_{i,j}$ it is straightforward to verify that $\phi$ commutes with the action of $\ve{k}[U_1,\dots, U_r]$.
\end{proof}

\begin{rem} The above lemma is false if $M$ is not graded. As an example, consider $M= \C[U]/(U^n-1)$. Then $U^n$ acts by $1$, whereas $(\a U)^n$ acts by $\a^n$, which may be chosen to be not 1. In particular $U'=(\a U)$ determines a different module structure. 
\end{rem}

Since $U$ is given grading $-1$ is sutured instanton and monopole homology, we obtain the following corollary:

\begin{cor}
 The graded $\C[U]$-isomorphism class of $\KHI^-(K)$ is well-defined.
\end{cor}

\section{Several tensor products}\label{sec:5}

In this section, we recall the derived tensor product of $\bF[U]$-modules, and also a tensor product of direct limits which arises in our proof of the connected sum formula.

\subsection{Derived tensor products over \texorpdfstring{$\bF[U]$}{F[U]}}

Suppose $\bF$ is a field and $M_1$ and $M_2$ are two finitely generated modules over the polynomial ring $\bF[U]$. We now discuss the derived tensor product of $M_1$ and $M_2$. To define the derived tensor product, it is helpful to introduce two $U$ variables, and view $M_1$ as being a module over $\bF[U_1]$ and $M_2$ as being a module over $\bF[U_2]$. The \emph{derived tensor product} of $M_1$ and $M_2$ is the homology of the chain complex
\[
M_1\tildeotimes_{\bF[U]} M_2:=H_* \left(\begin{tikzcd}[column sep=2cm] M_1\otimes_{\bF} M_2 [1]\ar[r, "U_1|1-1|U_2"] & M_1\otimes_{\bF} M_2 \end{tikzcd}\right).
\]
Note that $M_1\tildeotimes_{\bF[U]} M_2$ is a module over $\bF[U_1,U_2]$, however it is straightforward to see that both $U_1$ and $U_2$ have the same action on homology, so we view the homology as a $\bF[U]$-module, where $U$ acts by either $U_1$ or $U_2$. Here $[1]$ denotes a shift in the homological grading, if $M_1$ and $M_2$ have a grading (if they are not graded, we ignore the grading shift). Also, we are following the convention that $M[1]$ denotes $M\otimes_{\bF} \bF_1$, where $\bF_1$ is a copy of $\bF$ supported in homological grading $1$.

The following lemma is elementary and well known, though we record an elementary proof for completeness:

\begin{lem}
 Suppose that $C_1$ and $C_2$ are two finitely generated, free chain complexes over $\bF[U]$. Then all of the following modules are isomorphic:
 \begin{enumerate}
\item\label{group:1} $H_*(C_1\otimes_{\bF[U]} C_2)$.
\item\label{group:2} $H_*(C_1\tildeotimes_{\bF[U]} C_2)$.
\item\label{group:3} $H_*(H_*(C_1)\tildeotimes_{\bF[U]} H_*(C_2))$. 
\end{enumerate}
\end{lem}
\begin{proof}
 We first show that \eqref{group:1} and~\eqref{group:2} are isomorphic. To see this, we note that there is a short exact sequence
 \[
 \begin{tikzcd}[column sep=1.4cm]
 0 
 \ar[r]
 &
  C_1\otimes_{\bF} C_2 
 \ar[r, "U_1|1-1|U_2"]
 &
  C_1\otimes_{\bF} C_2
 \ar[r, "\Pi"]
 &
 C_1\otimes_{\bF[U]} C_2
\ar[r]
&
0
 \end{tikzcd}
\]
Here $\Pi$ is the canonical projection map. One can (non-canonically) define splittings of the above exact sequence, as follows. There is a splitting 
\[
s\colon C_1\otimes_{\bF[U]} C_2\to C_1\otimes_{\bF} C_2
\]
 given by $s(U^i \xs|\ys)=U_1^i \xs|\ys$. There is a splitting $p$ of the map $U_1|1-1|U_2$ given by 
 \[
 p(U_1^i \xs|U_2^j \ys)=U_1^i(U_1^{j-1}+U_1^{j-1} U_2+\cdots +U_1 U_2^{j-2}+U_2^{j-1})(\xs|\ys).
 \]
 It is easy to check that $p\circ s=0$. We can explicitly write down a homotopy equivalence between $C_1\otimes_{\bF[U]} C_2$ and $C_1\tildeotimes_{\bF[U]} C_2$. A map $\Phi\colon C_1\tildeotimes_{\bF[U]} C_2\to  C_1\otimes_{\bF[U]} C_2$ is given by $\Pi$ (applied to the codomain of the mapping cone). A map $\Psi$ in the opposite direction is given by $\Psi=(p \d s, s)$. It is easy to check that $\Phi\circ \Psi=\id$ and $\Psi\circ \Phi=\id+[\d,H]$, where $H$ is the map $p$ (viewed as an endomorphism of the mapping cone). We note that $\Phi$ and $\Psi$ are also maps of $\bF[U_1]$-modules.
 
 We now show that~\eqref{group:2} and~\eqref{group:3} are isomorphic. To see this, we note first that by the classification theorem for finitely generated, free chain complexes over a PID, the complex $C_1$ decomposes as a sum of 1-step and 2-step complexes. In particular $C_1$ decomposes as $C_1\iso (X\xrightarrow{F} X')$, where $X$ and $X'$ are free, finitely generated modules over $\bF[U_1]$ (with vanishing internal differential), and $F$ is some $\bF[U_1]$-equivariant map. In particular, the complex $C_1$ admits a quasi-isomorphism $\phi$ onto its homology (corresponding to projection onto $X'/F(X)$). Entirely analogously, there is a quasi-isomorphism $\psi\colon C_2\to H_*(C_2).$ We thus obtain an $\bF[U_1,U_2]$-equivariant chain map
 \[
 \phi|\psi\colon C_1\tildeotimes_{\bF[U]} C_2\to H_*(C_1)\tildeotimes_{\bF[U]} H_*(C_2),
 \]
 which we claim is a quasi-isomorphism. To establish this, we note that the derived tensor product has a 2-step filtration, corresponding to the $U_1|1-1|U_2$ direction. The associated graded complex to this filtration is two copies of $C_1\otimes_{\bF} C_2$ or two copies of $H_*(C_1)\otimes_{\bF} H_*(C_2)$. The map $\phi|\psi$ clearly induces an isomorphism on the homology of the associated graded complex, and hence an isomorphism on the homology of the original complex, by a spectral sequence argument.
\end{proof}

There is another perspective on the above model for the derived tensor product using the language of type-$D$ and $A$ modules of Lipshitz, Ozsv\'{a}th and Thurston \cite{LOTBordered} \cite{LOTBimodules}. Namely, we may define a type-$DD$ bimodule ${}^{\bF[U]}\Lambda^{\bF[U]}$, where $\Lambda=\Span_{\bF}(1,\theta)$ and
\[
\delta^{1,1}(1)=U\otimes \theta\otimes 1+1\otimes \theta\otimes U \quad \text{and} \quad \delta^{1,1}(\theta)=0.
\]
Then
\begin{equation}
M_{F[U]}\tildeotimes {}_{\bF[U]} N= M_{\bF[U]}\boxtimes {}^{\bF[U]}\Lambda^{\bF[U]}\boxtimes{}_{\bF[U]} N.\label{eq:derived-tensor-bimodules}
\end{equation}

\subsection{Tensor products and transitive systems}
\label{sec:tensor-products-transitive-systems}
In this section, we prove a result about tensor products and transitive systems which is used in our proof of the main theorem. Firstly, suppose that $\scG=(G_n,\phi_n)$ and $\scH=(H_n,\psi_n)$ are transitive systems of groups. Given $n$ and $m$, we will consider the chain complex given by the diagram
\[
\scS_{n,m}:=
\left(\begin{tikzcd}[labels=description, column sep=1cm]
G_{n-1}\otimes H_m\ar[dr, "\phi_{n-1}|\id"]&& G_n\otimes H_{m-1}\ar[dl, "\id|\psi_{m-1}"]\\
& G_n\otimes H_m
\end{tikzcd}\right)
\]

If $G$ is a $\Z$-graded group and $\delta\in \Z$, write $G^{>\delta}$ for the subgroup in gradings greater than $\delta$.

\begin{lem}\label{lem:staircase-lemma} Suppose that $(G_n,\phi_n)$ and $(H_n,\psi_n)$ are transitive systems such that $G_n$ and $H_n$ are equipped with $\Z$-gradings with respect to which $\phi_n$ and $\psi_n$ are grading preserving, and such that $G_n$ and $H_n$ are bounded above for each $n$. Furthermore, suppose that for each $\delta\in \Z$, there is an $N$ such that if $n,m>N$, maps $\psi_m\colon H_m^{>\delta}\to H_{m+1}^{>\delta}$ and $\phi_n\colon G_n^{>\delta}\to G_{n+1}^{>\delta}$ are isomorphisms, Then 
\[
H_*(\scS_{n,m})^{>\delta}\iso (G_n\otimes H_m)^{>\delta}
\]
for all sufficiently large $n$ and $m$.
\end{lem}
\begin{proof} The assumptions that $\phi_n$ and $\psi_n$ are isomorphisms in higher gradings, and that $G_n$ and $H_m$ are bounded above implies that the map $\phi|\psi\colon \scS_{n,m}^{>\delta}\to \scS_{n+1,m+1}^{>\delta}$ is also an isomorphisms, provided $n$ and $m$ are sufficiently large.

In this case, we can write down a chain homotopy equivalence between $\scS_{n,m}^{>\delta}$ and $(G_m\otimes H_n)^{>\delta}$ (where we give the latter complex vanishing differential). To this end, we define the following maps:
\[
  \begin{split}
 \Psi_{n, m}&\colon G_{n}\otimes H_{m}\to \scS_{n,m}\\
  \Phi_{n,m}&\colon \scS_{n,m}\to G_{n}\otimes H_m\\
  J_{n,m}&\colon \scS_{n,m}\to \scS_{n+1,m+1}
 \end{split}
 \]
 such that
 \begin{equation}
 \begin{split}
 \Phi_{n+1,m+1}\circ \Psi_{n,m}&=\phi_n|\psi_m\\
 \Psi_{n+1,m+1}\circ \Phi_{n,m}&=\phi|\psi+[\d, J_{n,m}].
 \end{split}
 \label{eq:relations-Phi-Psi-J}
 \end{equation}
 The maps $\Phi_{n,m}$, $\Psi_{n,m}$ and $J_{n, m}$ are shown in Figure~\ref{fig:transitive-systems}. The relations in~\eqref{eq:relations-Phi-Psi-J} are easily verified.
as indicated in Figure~\ref{fig:transitive-systems}.

 Since $\phi_m|\psi_n$ and $\phi|\psi$ are isomorphisms of vector spaces on the $\delta$-truncations of the complexes,  the proof is complete.
\end{proof}

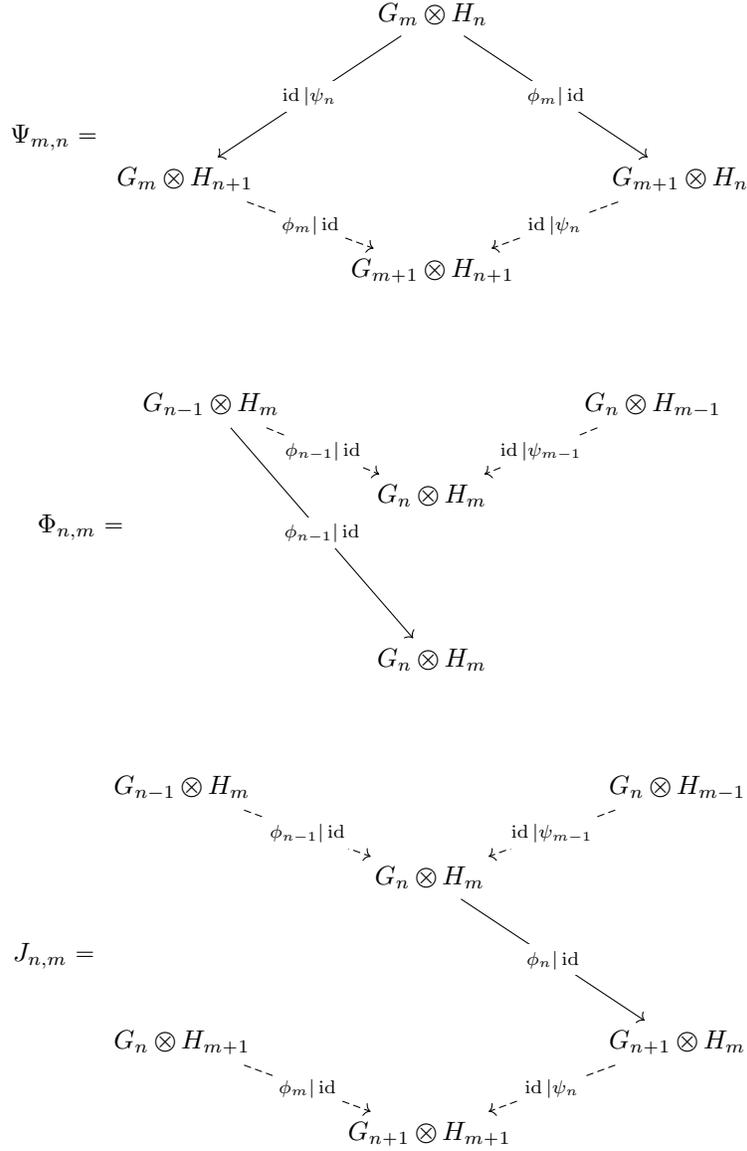
\begin{figure}[ht!]
\[
\Psi_{m,n}=\begin{tikzcd}[labels=description, column sep=1cm]
&G_{m}\otimes H_{n} 
	\ar[dl,"\id|\psi_n"]
	\ar[dr, "\phi_m|\id"]\\[1cm]
G_{m}\otimes H_{n+1}
	\ar[dr, dashed,"\phi_m|\id"]
	&&
G_{m+1}\otimes H_{n}
	\ar[dl,dashed, "\id|\psi_n"]\\
& G_{m+1}\otimes H_{n+1}
\end{tikzcd}
\]
\vspace{1cm}
\[
\Phi_{n,m}=\begin{tikzcd}[labels=description, column sep=1cm]
G_{n-1}\otimes H_m
	\ar[ddr,"\phi_{n-1}|\id"]
	\ar[dr, dashed,"\phi_{n-1}|\id"]
&&
G_n\otimes H_{m-1}
	\ar[dl,dashed, "\id|\psi_{m-1}"]
\\
& G_n\otimes H_m\\[1cm]
&G_n\otimes H_m
\end{tikzcd}
\]
\vspace{1cm}
\[
J_{n,m}=\begin{tikzcd}[labels=description, column sep=1cm]
G_{n-1}\otimes H_m
	\ar[dr, dashed,"\phi_{n-1}|\id"]
	&&
G_n\otimes H_{m-1}
	\ar[dl,dashed, "\id|\psi_{m-1}"]\\
& G_n\otimes H_m \ar[dr, "\phi_{n}|\id"]
\\[1cm]
G_{n}\otimes H_{m+1}
	\ar[dr, dashed,"\phi_{m}|\id"]
	&&
G_{n+1}\otimes H_{m}
	\ar[dl,dashed, "\id|\psi_{n}"]\\
& 
G_{n+1}\otimes H_{m+1}
\end{tikzcd}
\]
\caption{The morphisms $\Phi_{n,m}$, $\Psi_{n,m}$ and $J_{n,m}$ (denoted with solid arrows). Dashed arrows denote internal differentials of $\scK$.}
\label{fig:transitive-systems}
\end{figure}

\section{Hypercubes and generalized eigenspaces}
\label{sec:hypercubes-eigenspaces}

The sutured instanton theory is defined as a generalized eigenspace of the homology group $I_*(Y)$ of an operator $\mu(R)$. Since $\mu(R)$ commutes with the differential, the chain groups $\mathit{CI}_*(Y)$ admit a generalized eigenspace decomposition which is preserved by the differential. In this section, we prove several algebraic results about generalized eigenspaces and hypercubes. The work in this section lays the algebraic foundation of iterating mapping cones in the sutured instanton setting (since these are defined using eigenspaces).

\subsection{Hypercubes}

We will phrase our results in terms of Manolescu and Ozsv\'{a}th’s algebraic formalism of hypercubes of chain complexes \cite{MOIntegerSurgery}*{Section~5}. Define 
\[
\bE_n \coloneqq \{0, 1\}^n \subset \bR^n
\]
 with the convention that $\bE_0 = \{0\}$. 

If $\epsilon, \epsilon' \in \bE_n$, we write $\epsilon \leq \epsilon'$ if the inequality holds for each coordinate of $\epsilon$ and $\epsilon'$. We write $\epsilon < \epsilon'$ if $\epsilon \leq \epsilon'$ and strict inequality holds at one or more coordinates.

\begin{define}\label{def: hypercube}
An \emph{n-dimensional hypercube of chain complexes} $({C}_\epsilon, {D}_{\epsilon, \epsilon'})_{\epsilon \in \bE_n}$ consists of the following:
\begin{enumerate}
    \item A group ${C}_\epsilon$ for each $\epsilon \in \bE_n$.
    \item For each pair of indices $\epsilon, \epsilon' \in \bE_n$ such that $\epsilon \leq \epsilon'$, a linear map \[{D}_{\epsilon, \epsilon'} \colon C_\epsilon \to C_\epsilon'.\]
\end{enumerate}
Furthermore, we assume that the following compatibility condition holds whenever $\epsilon \leq \epsilon'$:

\begin{equation}\label{Eq : hypercube differential}
  \sum_{\substack{\veps' \in \bE_n \\ \veps \leq \veps' \leq \veps''}} {D_{\veps', \veps''} \circ {D}_{\veps, \veps'}} = 0
  \end{equation}
\end{define}

Of course, if we equip $\bigoplus_{\veps\in \bE_n} C_\veps$ with the differential $D:=\sum_{\veps\le \veps'} D_{\veps,\veps'}$, then Equation~\eqref{Eq : hypercube differential} is equivalent to $D^2=0$.

If $\cC$ and $\cC'$ are $n$-dimensional hypercubes, a morphism of hypercubes $F\colon \cC\to \cC'$ consists of a collection of maps $F_{\veps,\veps'}\colon C_\veps\to C_{\veps'}$ ranging over $\veps\le \veps'$. The space $\Hom(\cC,\cC')$ is a chain complex, with differential
\[
\d(F)=F\circ D+(-1)^{|F|+1}D'\circ F,
:\]
where we are conflating $F$ with its total map $\sum_{\veps\le \veps'} F_{\veps,\veps'}$. Also $|F|$ denotes the grading of $F$ (note that only a $\Z_2$ grading is necessary to define the above formula). 

 We may form the mapping cone diagram $\Cone(F\colon \cC\to \cC')$, viewed as a cube of dimension $n+1$ whose total space is $\cC\oplus T^{|F|+1}\cC'$, where $T$ denotes translation functor for chain complexes, which shifts grading by 1 and changes the differential $D'$ to $-D'$, and $T^{|F|+1}$ denotes its $(|F|+1)$-fold iterate. We equip this complex with the differential  
\[
\begin{pmatrix}D&0\\
F& (-1)^{|F|+1}D'
\end{pmatrix}.
\]

Here is an important and well-known lemma about homotopy equivalences of hypercubes:

\begin{lem}
\label{lem:hypercubes-quasi-iso} Suppose that $\cA=(A_{\veps}, d_{\veps,\veps'})_{\veps\in \bE_n}$ and $\cB=(B_{\veps},\delta_{\veps,\veps'})_{\veps\in \bE_n}$ are hypercubes of chain complexes such that each $A_\veps$ and $B_\veps$ are each finite dimensional $\C$ vector spaces. If $F\colon \cA\to \cB$ is a morphism of hypercubes which satisfies $\d(F)=0$ and such that each $F_{\veps,\veps}$ is a quasi-isomorphism for each $\veps$, then $F$ is a homotopy equivalence of hypercubes.
\end{lem}
\begin{proof} Since each $A_\veps$ and $B_\veps$ is a finite dimensional chain complex over $\C$, by using the homological perturbation lemma for hypercubes (cf. \cite{HHSZDual}*{Section~2.7}),  we may assume that each $A_{\veps}$ and $B_{\veps}$ has vanishing internal differential $d_{\veps,\veps}$ and $\delta_{\veps,\veps}$. In particular, we may assume that the maps $F_{\veps,\veps}$ are isomorphisms. This implies that the total map $F=\sum_{\veps\le \veps'} F_{\veps,\veps'}$ is also a chain isomorphism.  The inverse map $F^{-1}$ is $\bE_n$-filtered, and hence a morphism of hypercubes. The proof is complete. 
\end{proof}

\subsection{Generalized eigenspaces and hypercubes}

In this section, we describe a very important result about eigenspaces of hypercube operators. This result concerns a hypercube $\cC=(C_{\veps}, d_{\veps,\veps'})$ equipped with an operator $\mu\colon C\to C$ (i.e. a morphism of hypercubes). We will write $e^\lambda(\cC, \mu)$ for the generalized $\lambda$-eigenspace of the operator $\mu$ on the total space $\bigoplus_{\veps\in \bE_n} C_\veps$ of $\cC$.

In general, the space $e^{\lambda}(\cC,\mu)$ will not be a hypercube of chain complexes, and will instead only be a chain complex which is filtered by the cube $\bE_n$ (where we write $\veps\le \veps'$ for $\veps,\veps'\in \bE_n$ if inequality holds for all coordinates). In this section we will show that $e^{\lambda}(\cC,\mu)$ is always chain isomorphic to a hypercube of chain complexes, which we denote by 
\[
E^{\lambda}(\cC,\mu)=(E_{\veps}^\lambda, \delta_{\veps,\veps'})_{\veps\in \bE_n}, \quad \text{where} \quad 
E_{\veps}^{\lambda}=e^{\lambda}(C_\veps, \mu_{\veps,\veps}). 
\]
The complex $E^{\lambda}(\cC,\mu)$ will be well-defined up to chain isomorphism. Additionally, we analyze some further properties of the construction.

 First, recall that a \emph{filtered vector space} over $\bE_n$ consists of a vector space $Y$ equipped with subspaces $Y_{\ge \veps}\subset Y$ such that if $\veps'\ge \veps$ then 
\[
Y_{\ge \veps'}\subset Y_{\ge \veps}.
\]
We will write $Y_{>\veps}$ for $\sum_{\veps'>\veps} Y_{\veps'}$. Given a vector space $Y$ which is filtered by $\bE_n$, recall that the \emph{associated graded} space is defined to be
\[
\gr(Y):=\bigoplus_{\veps\in \bE_n} Y_{\ge \veps}/ Y_{>\veps}.
\]

A hypercube of chain complexes $\cC=(C_\veps,D_{\veps,\veps'})$ naturally determines a filtered chain complex $\cC$ over $\bE_n$, where we set \[C_{\ge \veps}:=
\bigoplus_{\veps'\ge \veps} C_{\veps'}.
\]
\begin{lem}\label{lem:eigenvales-cube} Suppose that $\cC=(C_{\veps})_{\veps\in \bE_n}$ is a cube of finite dimensional $\bC$-vector spaces, and $\mu\colon C\to C$ is a map which is filtered with respect to the cube filtration. Write  $e^{\lambda}(\cC,\mu)$ for the generalized $\lambda$-eigenspace of the total space of $\cC$. Then there is a canonical isomorphism
\[
\gr(e^{\lambda}(\cC,\mu))\iso \bigoplus_{\veps\in \bE_n} E^{\lambda}_\veps.
\]
Furthermore, there is an isomorphism
\[
\Phi\colon \bigoplus_{\veps\in \bE_n} E^{\lambda}_{\veps}\to e^{\lambda}(\cC,\mu)
\]
which is filtered and admits a filtered inverse.
\end{lem}
\begin{proof} Let us abbreviate $e^{\lambda}(\cC,\mu)$ by $e^{\lambda}$. We first claim that for each $\veps$, the canonical projection map
\[
e_{\ge \veps}^{\lambda}/e_{>\veps}^{\lambda}\to C_{\veps}
\]
is an isomorphism onto its image, which is $E^{\lambda}_{\veps}$. This may be proven by induction on the dimension of the cube. In general, we may view an $n$-dimensional cube $X$ as a pair of $n-1$ dimensional cubes $X_0$ and $X_1$. We have an exact sequence
\[
0\to X_0\to X\to X_1\to 0.
\]
If $X$ is equipped with a filtered endomorphism $\mu$, the map $\mu$ restricts to $X_0$ and $X_1$ and the maps in the above exact sequence commute with $\mu$. We obtain a sequence
\begin{equation}
0\to X_1^{\lambda}\to X^\lambda\to X_0^\lambda\to 0, \label{eq:short-exact-eigenvalues}
\end{equation}
which is trivially exact except at $X_0^{\lambda}.$ Exactness at $X_0^{\lambda}$ may be proven by observing that $X$ and $X_\veps$ decompose as direct sums of generalized eigenspaces, and then using dimension counts to establish surjectivity of the final map. By induction, $X_1^{\lambda}$ and $X_0^{\lambda}$ have associated graded spaces as in the statement. Together with the above exact sequence in Equation~\eqref{eq:short-exact-eigenvalues}, the claim about $\gr(e^\lambda)$ follows.

We now consider the claim that $e^{\lambda}\iso \bigoplus_{\veps\in \bE_n} E_\veps^{\lambda}$ as a filtered vector space. We construct a filtered map
\[
\Phi\colon \bigoplus_{\veps\in \bE_n} E_{\veps}^{\lambda}\to e^{\lambda}
\label{eq:splitting-maps}
\]
as follows. On $E_\veps^{\lambda}$, we define $\Phi$ to be any splitting of the map
\[
e^{\lambda}_{\ge \veps}\twoheadrightarrow e^{\lambda}_{\ge \veps}/e^{\lambda}_{>\veps}\iso E_{\veps}^{\lambda}.
\]
The map $\Phi$ is clearly filtered. The map $\Phi$ is an isomorphism. The inverse of the map $\Phi$ may be seen to be filtered by considering the following general and easily verified fact: a filtered map $F$ between two $\bE_n$-filtered spaces is an isomorphism with filtered inverse if and only if $\gr(F)$ is an isomorphism. Since $\Phi$ induces an identity on the associated graded complexes by construction, the proof is complete.
\end{proof}

 Write $C_{\Tot}$ for the chain complex $(\bigoplus_{\veps\in \bE_n} C_{\veps}, \sum_{\veps\le \veps'} D_{\veps,\veps'})$. Write $D_{\Tot}$ for the differential $D_{\Tot}=\sum_{\veps\le \veps'} D_{\veps,\veps'}$.  We therefore construct the hypercube $E^{\lambda}(\cC,\mu)$ as follows. Given a choice of filtered map
\[
\Phi\colon \bigoplus_{\veps\in \bE_n} E_\veps^\lambda\to e^{\lambda}(\cC,\mu)
\]
which is an isomorphism whose inverse is also filtered, we define  $E^{\lambda}(\cC,\mu)$ to be the hypercube whose total space is $\bigoplus_{\veps\in \bE_n}$ and whose total structure map is $\Phi^{-1}\circ D_{\Tot}\circ \Phi$. 

Note that the hypercube $E^{\lambda}(\cC,\mu)$ depends on the choice of $\Phi$, as in Equation~\eqref{eq:splitting-maps}. We refer to the map $\Phi$ as the \emph{splitting map}. If we wish to emphasize the choice of $\Phi$, we write
\[
E^{\lambda}(\cC,\mu;\Phi).
\]

\begin{lem}
\label{lem:hypercube-eigenspaces}
Let $\mu$ be an operator on the hypercube $\cC$, and write $E^\lambda(\cC,\mu;\Phi)=(E_{\veps}^{\lambda}, \delta_{\veps,\veps'})$ for the eigenspace hypercube. Then:
\begin{enumerate}
\item The internal differentials $\delta_{\veps,\veps}$ are the restrictions of the maps $D_{\veps,\veps}$ to the generalized $\lambda$-eigenspaces $C^{\lambda}_{\veps}$.
\item If $|\veps'-\veps|_{L^1}=1$, then
 \[
\delta_{\veps,\veps'}=D_{\veps,\veps'}+\delta_{\veps,'\veps'}\circ \Phi_{\veps,\veps'}-\Phi_{\veps,\veps'}\circ \delta_{\veps,\veps},
\]
where $\Phi_{\veps,\veps'}\colon E_{\veps}^{\lambda}\to C_{\veps'}$ is the component of $\Phi$ going from $E_{\veps}^{\lambda}$ to $C_{\veps'}$ 
\end{enumerate}
\end{lem}
\begin{proof}
 The claim about length 0 components is immediate and follows from the fact that the component of $\Phi$ from $E_{\veps}^{\lambda}$ to $C_{\veps}$ is the canonical inclusion.

 Suppose now that $|\veps'-\veps|_{L^1}=1$. It is helpful to view $\Phi$ as decomposing as a sum of maps $\Phi_{\veps,\veps'}$ ranging over $\veps'\ge \veps$. Write $(x,y)$ for elements of $ E_{\veps}^{\lambda}\oplus E_{\veps'}^{\lambda}$. By construction, $\Phi$ maps $(x,0)$ to $(x, \Phi_{\veps,\veps'}(x))$. The map $D^{\Tot}$ maps this to
\[
(D_{\veps,\veps}(x),D_{\veps,\veps'}(x)+D_{\veps',\veps'}( \Phi_{\veps,\veps'}(x))).
\]
The map $\Phi^{-1}$ sends the above to
\[
\left(D_{\veps}(x),D_{\veps,\veps'}(x) +D_{\veps'}(\Phi_{\veps,\veps'}(x))-\Phi_{\veps,\veps'}( D_{\veps}(x))\right).
\]
The component mapping from $E_{\veps}^{\lambda}$ to $E_{\veps'}^{\lambda}$ sends $x$ to
\[
D_{\veps,\veps'}(x) +D_{\veps'}(\Phi_{\veps,\veps'}(x))-\Phi_{\veps,\veps'}( D_{\veps}(x)).
\]
as claimed.
\end{proof}

We now consider the splitting map $\Phi$. Clearly the choice of $\Phi$ does not affect the homotopy type of the hypercube $E^{\lambda}(\cC,\mu)$. We now consider the problem of extending a given splitting map defined on a subcube.

 Note that if $\bE\subset \bE_n$ is a subcube  (i.e. $\bE\subset \bE_n$ is a cartesian product of copies of $\{0\}$, $\{1\}$ or $\{0,1\}$) and $\cC$ is an $n$-dimensional hypercube of chain complexes, then there is an induced hypercube $\cC|_{\bE}$ of dimension $\dim (\bE)$. This hypercube map be defined as a quotient of a subcomplex as follows. Consider the subcomplexes
\[
\cC_{< \bE}\subset \cC_{\le \bE} \subset \cC.
\]
The subcomplex $\cC_{\ge \bE}$ is the filtered subspace generated by $C_{\veps}$ where $\veps\ge \veps_0$ for some $\veps_0\in \bE$. The space $\cC_{>\bE}$ is spanned by $C_{\veps}$ such that $\veps\ge \veps_0$ for some $\veps_0\in \bE$, but $\veps\not\in \bE$. Then
\[
\cC_{\bE}:= \cC_{\ge \bE}/ \cC_{>\bE}. 
\]
An operator $\mu$ on $\cC$ induces an operator $\mu_{\bE}$ on $\cC_{\bE}$.

Note that a splitting map $\Phi$ for $(\cC,\mu)$ induces a splitting map $\Phi|_{\bE}$ on $(\cC_{\bE},\mu_{\bE})$. The following will be a helpful technical lemma later on, and shows that we can construct the splitting map $\Phi$ inductively, starting with the $0$-cells of the cube, then moving to the 1-cells, and so forth.

\begin{lem}
\label{lem:extend-splittings}
 Consider an $n$-dimensional hypercube $\cC$ equipped with an operator $\mu$. Suppose that $\phi_{\bE}$ is a collection of splitting maps  for each codimension 1 subcube $\bE\subset \bE_n$, which are compatible in the sense that if $\bE_0$ is a subcube of $\bE_n$ which is contained in two codimension 1 cubes $\bE'$ and $\bE''$, then
\[
\phi_{\bE'}|_{\bE_0}=\phi_{\bE''}|_{\bE_0}.
\]
Then there is a splitting map $\Phi$ for $(\cC,\mu)$ such that
\[
\Phi|_{\bE}=\phi_{\bE}
\]
for each codimension 1 subcube $\bE\subset \bE_n$. 
\end{lem}
\begin{proof}We can view a choice of splitting $\Phi$ as a collection of splittings $\Phi_{\veps}\colon E^{\lambda}(C_{\veps},\mu)\to e^{\lambda}_{\ge \veps}(\cC,\mu)$ and we note that the definition requires no relation between the different $\Phi_{\veps}$ maps. The maps $\Phi_{\veps}$ are uniquely determined by the $\phi_{\bE}$ maps for $\veps\neq (0,\dots,0)$, so it suffices to define $\Phi_{(0,\dots, 0)}$.

Observe that $\cC_{\ge \bE} =\cC$ if $(0,\dots, 0)\in \bE$.  Therefore
\[
e^{\lambda}(\cC_{\bE})=e^{\lambda}(\cC/\cC_{>\bE})\iso e^{\lambda}(\cC)/e^{\lambda}_{>\bE}(\cC).
\]
The first equality is the definition, and the second isomorphism is by exactness of $e^{\lambda}$, proven in Lemma~\ref{lem:eigenvales-cube} (see Equation~\eqref{eq:short-exact-eigenvalues}).

For each pair of codimension 1 subcubes $\bE,\bE'\subset \bE_n$, we have a commutative diagram of epimorphisms
\[
\begin{tikzcd} & e^{\lambda}(\cC)\ar[dl, twoheadrightarrow] \ar[dr,twoheadrightarrow] \\
e^{\lambda}(\cC)/e^{\lambda}(\cC_{>\bE})
	\ar[dr, twoheadrightarrow]
&& e^{\lambda}(\cC)/e^{\lambda}(\cC_{>\bE'})
	\ar[dl, twoheadrightarrow]
\\
& E_{(0,\dots,0)}^{\lambda}.
\end{tikzcd}
\]

The splittings in the statement fit into into commutative diagrams:
\[
\begin{tikzcd}[labels=description]
e^{\lambda}(\cC)/e^{\lambda}(\cC_{>\bE})
	\ar[r, twoheadrightarrow]
 & e^{\lambda}(\cC)/(e^{\lambda}(\cC_{>\bE})+e^{\lambda}(\cC_{>\bE'})) 
& e^{\lambda}(\cC)/e^{\lambda}(\cC_{>\bE'})
	 \ar[l, twoheadrightarrow] 
\\
& E^{\lambda}_{(0,\dots,0)}
	\ar[ul, "\phi_{\bE}"] 
	\ar[ur, "\phi_{\bE'}"]
\end{tikzcd}
\]
Such a diagram uniquely determines a splitting
\[
\phi_{\bE,\bE'}\colon E^{\lambda}_{(0,\dots, 0)}\to e^{\lambda}(\cC)/(e^{\lambda}(\cC_{>\bE})\cap e^{\lambda}(\cC_{>\bE'})).
\]
Together, the collection of all $\phi_{\bE}$ (where $\codim(\bE)=1$) uniquely specifies a splitting
\[
\Psi\colon E_{(0,\dots, 0)}^{\lambda}\to  e^{\lambda}(\cC)/\bigcap_{\substack{\codim(\bE)=1\\(0,\dots, 0)\in \bE}} e^{\lambda}(\cC_{>\bE})=e^{\lambda}(\cC)/e^{\lambda}(\cC_{(1,\dots, 1)}).
\]
Composing this with an arbitrary splitting of the map
\[
e^{\lambda}(\cC)\twoheadrightarrow e^{\lambda}(\cC_{(1,\dots, 1)})
\]
gives a choice of $\Phi_{(0,\dots, 0)}$ in the statement.
\end{proof}

\subsection{The eigenspace functor}

We now consider the operation of taking eigenspaces on chain complexes, viewed functoriality. Firstly, there are several categories relevant for us. The first is the category of chain complexes with operators  $\mu$-$\Kom$. Objects are pairs $(C,\mu)$ where $C$ is a finitely generated chain complex over $\C$ and $\mu\colon C\to C$ is a chain map. Morphisms consist of pairs $(f,h)\colon (C_1,\mu_1)\to (C_2,\mu_2)$ where $f\colon C_1\to C_2$ is a chain map and $h\colon C_1\to C_2$ is a chain homotopy
\[
f\circ \mu_1-\mu_2\circ f=\d(h). 
\]
Equivalently, $f$ and $h$ fit into a hypercube of chain complexes
\[
\begin{tikzcd}C_0
	\ar[r, "f"]
	\ar[d, "\mu_0"]
	\ar[dr,"h"]
& TC_1
	\ar[d, "\mu_1"]
\\
TC_0 \ar[r, "-f"] & T^2C_1
\end{tikzcd}
\]
Two morphisms $(f,h)$ and $(g,j)$ are \emph{homotopic} if they are homotopic as morphisms of hypercubes from $\Cone(\mu_1)\to \Cone(\mu_2)$. (Here $T$ denotes the translation functor on chain complexes, which shifts gradings by 1 and changes the differential $\d$ to $-\d$). 

 There is a relatived category $\mu$-$\Kom_{\bE}$ consisting of hypercubes $\cC$ equipped with operators $\mu$ (filtered chain maps satisfying $\d(\mu)=0$). Morphisms are defined similarly.

There are also homotopy versions of the above categories, denoted $\mu$-$\ve{K}$ and $\mu$-$\ve{K}_{\bE}$. The objects consist of pairs $(\cC,\mu)$ where $\cC$ is a chain complex (resp. hypercube) and $\mu$ is an operator (\emph{not} a homotopy class of operators). Morphisms in these categories consist of morphisms $(f,h)$ in the chain complex category which are cycles, modulo homotopy equivalence.

In this section, we prove the following:

\begin{prop} The generalized eigenspace construction gives a functor from $\mu$-$\ve{K}$ to $\ve{K}$. It also gives a functor from $\mu$-$\ve{K}_{\bE}$ to $\ve{K}_{\bE}$. 
\end{prop}


To an object in $\Kom_{\bE}$-$\mu$, (i.e. a hypercube of chain complexes equipped with an operator, $(\cC,\mu)$) we use Lemma~\ref{lem:hypercube-eigenspaces} to construct a model of $E^{\lambda}(\cC,\mu)$. (To define the functors values concretely, we use the axiom of choice to pick arbitrarily one model for each hypercube).  To a morphism, $(f,h)\colon (\cC_0,\mu_0)\to (\cC_1,\mu_1)$ in $\mu$-$\Kom_{\bE}$, we apply the eigenspace hypercube construction to $\Cone(f,h)$. By Lemma~\ref{lem:extend-splittings}, the splitting map of the mapping cone can be chosen compatibly with arbitrary splitting maps used to construct $E^{\lambda}(\cC_0,\mu_0)$ and $E^{\lambda}(\cC_1,\mu_1)$.

We illustrate the construction in the case of a morphism $(f,h)\colon (C_0,\mu_0)\to (C_1,\mu_1)$ in $\mu$-$\Kom$. Examining the proof of Lemma~\ref{lem:hypercube-eigenspaces}, the map
\[
E^{\lambda}(f,h)\colon E^{\lambda}(C_0,\mu_0)\to E^{\lambda}(C_1,\mu_1)
\]
is equal to the restriction of $f+\d(\Phi_{0,1})$, where  $\Phi_{0,1}\colon C_0^{\lambda}\to C_1$ is any linear map (of grading $+1$) so that $(i, \Phi_{0,1})$ maps $E^{\lambda}(C_0)$ into $E^{\lambda}( C_0\xrightarrow{f} C_1)$. Here $i$ denotes inclusion.  

We now show that the construction is well-defined on level of chain homotopy classes of maps:

\begin{lem}
\label{lem:eigenspace-functor}
On $\mu$-$\Kom$, the map $E^{\lambda}$ satisfies the following:
\begin{enumerate}
\item For a fixed morphism $(f,h)\colon (C_0,\mu_0)\to (C_1,\mu_1)$, the morphism $E^{\lambda}(f,h)$ is a chain map and is well-defined up to chain homotopy.
\item If $(f,h)\simeq (g,j)$ then $E^{\lambda}(f,h)\simeq E^{\lambda}(g,j)$.
\item Given a sequence of composable morphisms
\[
\begin{tikzcd} (C_0,\mu_0)
	\ar[r, "{(f,h)}"] & 
(C_1, \mu_1)
	\ar[r, "{(g,j)}"] & 
(C_2,\mu_2)
\end{tikzcd}
\]
we have
\[
E^\lambda(g,j)\circ E^{\lambda}(f,h)\simeq E^{\lambda}( (g,j)\circ  (f,h)). 
\]
\end{enumerate} 
The same claims hold when $E^{\lambda}$ is applied to hypercubes.
\end{lem}
\begin{proof} 
We focus on the case of chain complexes with operators,  since the claim for hypercubes is not appreciably different. We consider the first claim. The only ambiguity is the choice of the splitting map $\Phi_{0}\colon E^{\lambda}(C_0)\to e^{\lambda}(\Cone(f))$. Let $\Phi$ and $\Phi'$ be two choices of linear maps. We view $e^{\lambda}(\Cone(f))$ as being in $C_0\oplus C_1$. 

Note that $\Phi-\Phi'$ maps $E^\lambda(C_0)$ into $E^{\lambda}(C_1)$ since if $x\in E^{\lambda}(C_0)$, we have
\[
\Phi_{0}(x)=(x,\Phi_{0,1}(x)) \quad \text{and} \quad \Phi_0'(x)=(x,\Phi_{0,1}'(x))
\] 
for some linear maps $\Phi_{0,1},\Phi_{0,1}'\colon E^{\lambda}(C_0)\to C_1$. Since
$e^{\lambda}(\Cone(f)) \cap (0\oplus C_1)=0\oplus E^{\lambda}(C_1)$, we conclude as in the second part of Lemma~\ref{lem:hypercube-eigenspaces} that
\[
E^\lambda(f,h;\Phi)-E^\lambda(f,h;\Phi')=\d(\Phi_{0,1}+\Phi_{0,1}'). 
\]

The second claim again follows from the hypercube eigenspace theorem, applied to the 3-dimensional hypercube corresponding to the diagram
\[
\begin{tikzcd} (C_0,\mu_0)\ar[d, "{\id }"]
	\ar[r, "{(f,h)}"] 
	\ar[dr, dashed]
&(TC_1,\mu_1)
	\ar[d, "{\id }"]
\\
(TC_0,\mu_0) \ar[r, "{-(g,h)}"] & (T^2C_1,\mu_1)
\end{tikzcd}
\]
(We are omitting one direction of the cube, which goes ``into'' the page. This is the direction of the $\mu_i$ operators). The hypercube eigenspace theorem gives a hypercube of the form
\[
\begin{tikzcd}[column sep=1.5cm]
 E^{\lambda}(C_0,\mu_0)
	\ar[d, "\id"]
	\ar[r, "{E^{\lambda}(f,h)}"] 
	\ar[dr, dashed]
&  E^{\lambda}(TC_1,\mu_1)
	\ar[d, "\id"]
\\
E^{\lambda}(TC_0,\mu_0) \ar[r, "{-E^{\lambda}(g,h)}"] &  E^{\lambda}(T^2 C_1,\mu_1)
\end{tikzcd}
\]
which proves the claim.

The third claim is similar: we apply the hypercube eigenspace theorem to the following diagram:
\[
\begin{tikzcd} (C_0,\mu_0)
	\ar[d,swap, "{(g,j)\circ (f,h)}"]
	\ar[r, "{(f,h)}"] 
	\ar[dr, dashed]
&(TC_1,\mu_1)
	\ar[d, "{(g,j)}"]
\\
(TC_2,\mu_2)
 	\ar[r, "{-\id}"] 
&(T^2 C_2,\mu_2)
\end{tikzcd}
\]
\end{proof}

\subsection{Simultaneous eigenspaces}
\label{sec:simultaneous-eigenspaces}

We now consider the situation of simultaneous eigenspaces. We focus on the case of two operators $\mu$ and $\mu'$ on a chain complex $C$ such that $\mu\circ \mu'\simeq \mu'\circ \mu$.

We define the simultaneous generalized eigenspace as
\[
E^{\lambda'}( E^{\lambda}(C,\mu),E(\mu')).
\]
This eigenspace depends on the choice of homotopy between $\mu\circ \mu'$ and $\mu'\circ \mu$. The construction also works if $\cC$ is a hypercube and $\mu$ and $\mu'$ are homotopy commuting hypercube operators.

\begin{lem}\label{lem:simultaneous-eigenvalues}
 Suppose that $\mu$ and $\mu'$ are two homotopy commuting operators on a hypercube $\cC$ (and a null-homotopy of $[\mu,\mu']$ is chosen). Then there is a homotopy equivalence of chain complexes
\[
E_{\lambda'}(E^{\lambda}(\cC,\mu), E(\mu'))\simeq E^{\lambda}(E_{\lambda'}(\cC,\mu'),E(\mu)). 
\]
\end{lem}
\begin{proof}
We can define a map from one complex to the other. We note that given an operator on $\mu$, there are inclusion and projection maps
\[
i_\lambda\colon e^{\lambda}(\cC,\mu)\hookrightarrow \cC\quad \text{and} \quad \pi_\lambda\colon \cC\twoheadrightarrow e^{\lambda}(\cC,\mu).
\]
When $\mu$ is a hypercube operator, it is straightforward to check that $i_\lambda$ and $\pi_\lambda$ are both $\bE_n$-filtered maps. Composing these maps with the filtered isomorphism $e^{\lambda}(\cC,\mu)\iso E^{\lambda}(\cC,\mu)$, we obtain inclusion and projection maps
\[
E^{\lambda}(\cC,\mu)\to \cC\quad \text{and} \quad  \cC\to E^{\lambda}(\cC,\mu)
\]

%


 Therefore we have chain maps
\[
E_{\lambda'}(E^{\lambda}(\cC,\mu), E(\mu'))\to E^{\lambda}(\cC,\mu)\to \cC\to E_{\lambda'}(\cC,\mu')\to E^{\lambda}(E_{\lambda'}(\cC,\mu'), E(\mu)). 
\]
We claim that the composition of the above sequence of hypercube maps induces a homotopy equivalence at the complexes at each point $\veps\in \bE_n$. If we write $\cA_{\veps}$ and $\cB_{\veps}$ for the complexes at $\veps\in \bE_n$ of $E_{\lambda'}(E^{\lambda}(\cC))$ and $E^{\lambda}(E_{\lambda'}(\cC))$, respectively, we observe that $H_*(\cA_{\veps})$ and $H_*(\cB_{\veps})$ both coincide with the ordinary simultaneous eigenspace of $H_*(\cC_{\veps})$ with respect to the commuting operators $\mu_*$ and $\mu_*'$ on homology. Furthermore, the induced map from $H_*(\cA_\veps)$ to $H_*(\cB_{\veps})$ by the above sequence of maps is an isomorphism.  Therefore Lemma~\ref{lem:hypercubes-quasi-iso} implies the claim.
%
\end{proof}

If instead we have a sequence of operators $\mu_1,\dots, \mu_n$, our analysis will hold as hold as we have an $n$-dimensional hypercube where where the length 1 maps are the $\mu_i$ maps. Lemma~\ref{lem:simultaneous-eigenvalues} has an obvious generalization.

\subsection{Comparing eigenspaces}

We now investigate a technical result which will be useful when considering grading shifts on the instanton groups. 

\begin{prop}
\label{prop:grading-shift} Suppose that $\cC=(C_{\veps}, D_{\veps,\veps'})_{\veps\in \bE_n}$ is a hypercube and $\mu$ and $\mu'$ are two homotopy commuting operators on $\cC$. Suppose that additionally that on each $H_*(C_\veps)$, we have that $\mu'=\mu+\a \cdot \id$ (equivalently, $H_*(E^{\lambda}(C_\veps,\mu))=H_*(E^{\lambda+\a}(C_\veps,\mu'))$ for all $\lambda$). Then
\[
E^{\lambda}(\cC,\mu)\simeq E^{\lambda+r}(\cC,\mu'). 
\] 
\end{prop}

We begin with a straightforward lemma:
\begin{lem}\label{lem:shift-lemma} Suppose that $\mu$ is an operator on a hypercube $\cC=(C_{\veps}, D_{\veps,\veps'})_{\veps\in \bE_n}$ of finite dimensional vector spaces and there is a $\lambda$ such that $E^{\lambda}(H_*(C_{\veps}), \mu)=H_*(C_{\veps})$ for all $\veps$. Then
\[
E^{r}(\cC,\mu)\simeq \begin{cases}\cC& \text{ if } r=\lambda\\
0& \text{ otherwise}.
\end{cases}
\]
\end{lem}
\begin{proof} There is an inclusion map $E^r(\cC,\mu)\to \cC$, similar to the proof of Lemma~\ref{lem:simultaneous-eigenvalues}. This map restricts to a quasi-isomorphism at each vertex, so Lemma~\ref{lem:hypercubes-quasi-iso} implies that it is a homotopy equivalence if $r=\lambda$. If $r\neq \lambda$ then each $E^r(C_\veps,\mu)$ has trivial homology, so a similar argument shows that $E^r(\cC,\mu)$ is homotopy equivalent to the 0 complex.
\end{proof}

\begin{proof}[Proof of Proposition~\ref{prop:grading-shift}] By Lemma~\ref{lem:shift-lemma}, we see that
\begin{equation}
E^{\lambda+\a}(E^{\lambda}(\cC,\mu),E(\mu'))\simeq 
E^{\lambda}(\cC,\mu)
\quad \text{and} \quad 
E^{\lambda}(E^{\lambda+\a}(\cC,\mu'),E(\mu))\simeq E^{\lambda+\a}(\cC,\mu').
\label{eq:repeated-eigenspace}
\end{equation}
 Combining Equation~\eqref{eq:repeated-eigenspace} with Lemma~\ref{lem:simultaneous-eigenvalues}, we obtain
\[
E^{\lambda}(\cC,\mu)\simeq E^{\lambda+r}(\cC,\mu')
\]
completing the proof. 
\end{proof}

\section{Connected sums in direct limit Heegaard knot Floer homology}
\label{sec:connected-sum-HF}

In this section, we describe our proof of the connected sum formula in the context of the direct limit  Heegaard knot Floer homology. Although the connected sum formula has long been established for Heegaard knot homology \cite{OSKnots}, the argument we present here will be a useful model for the instanton and monopole settings. In this section, we work over characteristic 2.

\subsection{Overview}

We let $K_1\subset Y_1$ and $K_2\subset Y_2$ be null-homologous knots. Let $M_1$ and $M_2$ denote $Y_1\setminus \nu(K_1)$ and $Y_2\setminus \nu(K_2)$ respectively. Let $\Gamma_m$ and $\Gamma_n$ denote longitudinal sutures  on $M_1$ and $M_2$ with slope $m$ and $n$, respectively. We begin by attaching a 3-dimensional 1-handle to $\d M_1$ and $\d M_2$ which connects $R_+$ of the two sutured manifolds. We also attach a 1-handle which connects $R_-$. We assume these 1-handles are disjoint from the sutures. See Figure~\ref{fig:11}. We call the resulting manifold $Y$, and write $\Gamma_{m,n}$ for the sutures.

The manifold $Y$ is not taut since it may be compressed along the cocores of the attached 1-handles. In particular,
\begin{equation}
\SFH(Y,\Gamma_{n, m})=0.\label{eq:not-taut}
\end{equation}

Consider the closed curves $L_1$ and $L_2$ in $\d Y$ which are shown in Figure~\ref{fig:11}. We assume that the portions of $L_1$ and $L_2$ which run along $\d M_i$ run parallel to a meridian of $K_1$ and $K_2$.

\begin{figure}[ht]
\begingroup%
  \makeatletter%
  \providecommand\color[2][]{%
    \errmessage{(Inkscape) Color is used for the text in Inkscape, but the package 'color.sty' is not loaded}%
    \renewcommand\color[2][]{}%
  }%
  \providecommand\transparent[1]{%
    \errmessage{(Inkscape) Transparency is used (non-zero) for the text in Inkscape, but the package 'transparent.sty' is not loaded}%
    \renewcommand\transparent[1]{}%
  }%
  \providecommand\rotatebox[2]{#2}%
  \newcommand*\fsize{\dimexpr\f@size pt\relax}%
  \newcommand*\lineheight[1]{\fontsize{\fsize}{#1\fsize}\selectfont}%
  \ifx\svgwidth\undefined%
    \setlength{\unitlength}{237.15934312bp}%
    \ifx\svgscale\undefined%
      \relax%
    \else%
      \setlength{\unitlength}{\unitlength * \real{\svgscale}}%
    \fi%
  \else%
    \setlength{\unitlength}{\svgwidth}%
  \fi%
  \global\let\svgwidth\undefined%
  \global\let\svgscale\undefined%
  \makeatother%
  \begin{picture}(1,0.79551907)%
    \lineheight{1}%
    \setlength\tabcolsep{0pt}%
    \put(0,0){\includegraphics[width=\unitlength,page=1]{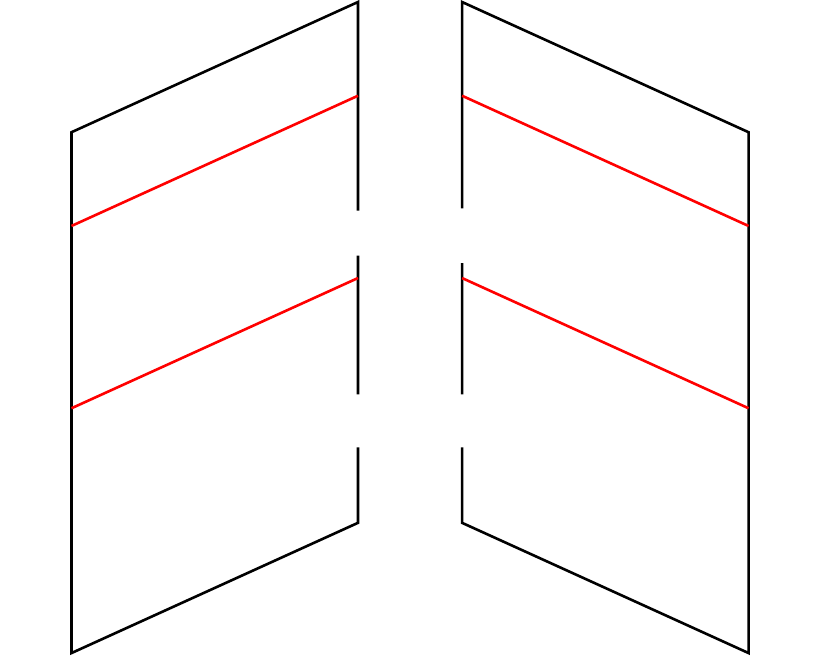}}%
    \put(-0.0012765,0.35935851){\makebox(0,0)[lt]{\lineheight{1.25}\smash{\begin{tabular}[t]{l}$M_1$\end{tabular}}}}%
    \put(0.94406455,0.35935851){\makebox(0,0)[lt]{\lineheight{1.25}\smash{\begin{tabular}[t]{l}$M_2$\end{tabular}}}}%
    \put(0,0){\includegraphics[width=\unitlength,page=2]{fig11.pdf}}%
    \put(0.75318913,0.41808953){\makebox(0,0)[lt]{\lineheight{1.25}\smash{\begin{tabular}[t]{l}$L_1$\end{tabular}}}}%
    \put(0.75318913,0.15380063){\makebox(0,0)[lt]{\lineheight{1.25}\smash{\begin{tabular}[t]{l}$L_2$\end{tabular}}}}%
    \put(0.75318913,0.64849512){\makebox(0,0)[lt]{\lineheight{1.25}\smash{\begin{tabular}[t]{l}$L_2$\end{tabular}}}}%
    \put(0,0){\includegraphics[width=\unitlength,page=3]{fig11.pdf}}%
  \end{picture}%
\endgroup%

\caption{The manifold $Y=Y_{\infty,\infty}(L)$ and the components $L_1$ and $L_2$ of $L$.}
\label{fig:11}
\end{figure}

Note that if we attach contact 2-handles along $L_1$ and $L_2$, then we obtain the sutured manifold complementary to $K_1\# K_2$, with longitudinal sutures of slope $n+m$. Let us write $(M^\#, \Gamma_{m+n}^{\#})$. Using the topological description from Lemma~\ref{lem:ctct-2-handle=triangles}, we may instead view a contact handle attachment as a 4-dimensional 2-handle attachment, followed by cutting along a product disk. Abusing notation, we write $L_1$ and $L_2$ also for the framed knots obtained by pushing $L_1$ and $L_2$ into the interior of $Y$, and using the framing which is parallel to the boundary. Writing $Y_{0,0}(L)$ for the surgered manifold, we conclude that
\[
\SFH(Y_{0,0}(L),\Gamma_{n, m})\iso \SFH(M^{\#}, \Gamma_{n+m}^{\#}),
\]
where the isomorphism is given by cutting along two product disks (i.e. using the inverse of the contact 1-handle map).

We consider the iterated mapping cone construction applied to the framed link $L$. This implies that there is a quasi-isomorphism taking the following form:
\begin{equation}
\CF(Y_{0,0}(L))\simeq
\begin{tikzcd}[labels=description, column sep=1.5cm, row sep=1.5cm]
\CF(Y_{1,1}(L))
	\ar[r, "F_{W(L_1^*)}"]
	\ar[d,"F_{W(L_2^*)}"]
	\ar[dr,dashed]
&
\CF(Y_{\infty, 1}(L))
	\ar[d,"F_{W(L_2^*)}"]
\\
\CF(Y_{1,\infty}(L))
	\ar[r, "F_{W(L_1^*)}"]
&
 \CF(Y_{\infty,\infty}(L))
\end{tikzcd}
\label{eq:iterated-mapping-cone}
\end{equation}
Here $L_i^*$ denotes a $-1$-framed dual of $L_i$. Since each chain complex is of $\bF$-vector spaces, we may replace the complexes with their homologies. Furthermore, the length 1 arrows (i.e. the solid arrows) are replaced by the induced maps on homology. The cube may be simplified even further by using~\eqref{eq:not-taut}. From these considerations, we obtain the following:
\begin{prop}\label{prop:link-surgery}
There is a homotopy equivalence
\[
\SFH(Y_{0,0}(L))\simeq
\begin{tikzcd}[labels=description, column sep=1.5cm, row sep=1.5cm]
\SFH(Y_{1,1}(L))
	\ar[r, "F_{W(L_1^*)}"]
	\ar[d,"F_{W(L_2^*)}"]
&
\SFH(Y_{\infty, 1}(L))
\\
\SFH(Y_{1,\infty}(L))
&
\end{tikzcd}
\]
Here $L_i^*$ denotes a $(-1)$-framed dual of $L_i$.
\end{prop}

In the subsequent section, we will show the following:

\begin{lem}\label{lem:reorganize-cone-Heegaard}
There is a homotopy equivalence of hypercubes between the diagram
\begin{equation}
\begin{tikzcd}[labels=description, column sep=1.5cm, row sep=1.5cm]
\SFH(Y_{1,1}(L))
	\ar[r, "F_{W(L_1^{*})}"]
	\ar[d,"F_{W(L_2^{*})}"]
&
\SFH(Y_{\infty, 1}(L))
\\
\SFH(Y_{1,\infty}(L))
&
\,
\end{tikzcd}
\label{eq:iterated-mapping-cone-1}
\end{equation}
and the diagram
\begin{equation}
\begin{tikzcd}[labels=description, column sep=3cm, row sep=1.5cm]
C_n\otimes D_{m-1}\oplus C_{n-1}\otimes D_m
	\ar[r, "\id| \phi_m^-\oplus\phi_n^-| \id"]
	\ar[d,"\id| \phi_m^+\oplus \phi_n^+| \id"]
&
C_n\otimes D_m
\\
C_n\otimes D_m
&
\,
\end{tikzcd}
\label{eq:expanded-derived-tensor-product}
\end{equation}
In the above, $C_n$ denotes $\SFH(M_1,\Gamma_n)$ and $D_m$ denotes $\SFH(M_2,\Gamma_m)$. 
\end{lem}

As a final step, we will show in Section~\ref{sec:reorganize-obtain-derived} the following:

\begin{lem}
\label{lem:obtain-direct-limit} The direct limit of the complex in Equation~\eqref{eq:expanded-derived-tensor-product} has homology which coincides with that of the direct limit of the transitive system
\[
\begin{tikzcd}[column sep=3cm]
C_n\otimes D_m\ar[r, "\phi_n^-|\phi_m^+ +\phi_n^+|\phi_m^-"]& C_{n+1}|D_{m+1}.
\end{tikzcd}
\]
Furthermore, the direct limit of this system coincides with the derived tensor product $\HFK^-(K_1)\tildeotimes_{\bF[U]} \HFK^-(K_2)$, which we identify with
\[
\begin{split}
\Cone\left(U_1+U_2\colon \HFK^-(K_1)\otimes_{\bF} \HFK^-(K_2)\to \HFK^-(K_1)\otimes_{\bF} \HFK^-(K_2)\right).
\end{split}
\]
\end{lem}

\subsection{Decomposing the iterated mapping cone}

In this section, we prove Lemma~\ref{lem:reorganize-cone-Heegaard}. We interpret the maps and complexes appearing in Proposition~\ref{prop:link-surgery} in terms of Heegaard diagrams. Although logically redundant with some of the more formal arguments we consider later in the setting of instantons and monopoles, we feel that this analysis is illuminating.

There is an alternate way to construct the complement of $K_1\# K_2$ in $Y_{1}\# Y_2$, with longitudinal sutures of slope $\Gamma_{m+n}$. Namely, we can attach a contact 1-handle connecting $M_1$ to $M_2$ along their sutures. Then we attach a single contact 2-handle which crosses the contact 1-handle twice. This construction is considered in \cite{GhoshLiWongTau}*{Figure~3}. The effect on the level of Heegaard diagrams is shown in Figure~\ref{fig:1}. Note that to arrive at $Y_{0,0}(L)$, we must attach two additional contact 1-handles.

\begin{figure}[ht]
\begingroup%
  \makeatletter%
  \providecommand\color[2][]{%
    \errmessage{(Inkscape) Color is used for the text in Inkscape, but the package 'color.sty' is not loaded}%
    \renewcommand\color[2][]{}%
  }%
  \providecommand\transparent[1]{%
    \errmessage{(Inkscape) Transparency is used (non-zero) for the text in Inkscape, but the package 'transparent.sty' is not loaded}%
    \renewcommand\transparent[1]{}%
  }%
  \providecommand\rotatebox[2]{#2}%
  \newcommand*\fsize{\dimexpr\f@size pt\relax}%
  \newcommand*\lineheight[1]{\fontsize{\fsize}{#1\fsize}\selectfont}%
  \ifx\svgwidth\undefined%
    \setlength{\unitlength}{262.29851979bp}%
    \ifx\svgscale\undefined%
      \relax%
    \else%
      \setlength{\unitlength}{\unitlength * \real{\svgscale}}%
    \fi%
  \else%
    \setlength{\unitlength}{\svgwidth}%
  \fi%
  \global\let\svgwidth\undefined%
  \global\let\svgscale\undefined%
  \makeatother%
  \begin{picture}(1,0.84425422)%
    \lineheight{1}%
    \setlength\tabcolsep{0pt}%
    \put(0,0){\includegraphics[width=\unitlength,page=1]{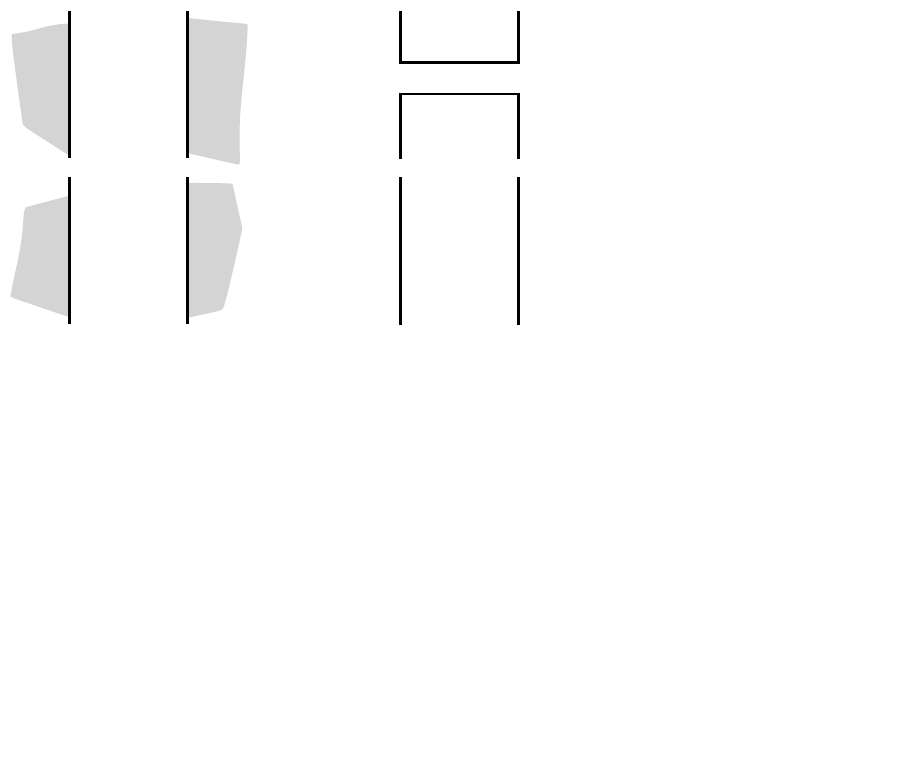}}%
    \put(0.06570388,0.73697142){\makebox(0,0)[rt]{\lineheight{1.25}\smash{\begin{tabular}[t]{r}$\Sigma_1$\end{tabular}}}}%
    \put(0.06570388,0.56187778){\makebox(0,0)[rt]{\lineheight{1.25}\smash{\begin{tabular}[t]{r}$\Sigma_1$\end{tabular}}}}%
    \put(0.21493921,0.73697142){\makebox(0,0)[lt]{\lineheight{1.25}\smash{\begin{tabular}[t]{l}$\Sigma_2$\end{tabular}}}}%
    \put(0.21493921,0.56187778){\makebox(0,0)[lt]{\lineheight{1.25}\smash{\begin{tabular}[t]{l}$\Sigma_2$\end{tabular}}}}%
    \put(0,0){\includegraphics[width=\unitlength,page=2]{fig1.pdf}}%
    \put(0.68446531,0.01795649){\makebox(0,0)[t]{\lineheight{1.25}\smash{\begin{tabular}[t]{c}$Y_{0,0}(L)$\end{tabular}}}}%
  \end{picture}%
\endgroup%

\caption{Constructing $Y_{0,0}(L)$ from diagrams of $M_1$ and $M_2$. The first arrow is a contact 1-handle. The second arrow is a contact 2-handle. The third arrow is an isotopy. The fourth arrow is two contact 1-handles. Note that in each frame, the two alpha arcs are part of the same alpha curve, and similarly for the beta arcs.}
\label{fig:1}
\end{figure}

 With this Heegaard diagram in hand, the diagrams appearing in the iterated mapping cone complex in ~\eqref{eq:iterated-mapping-cone} are shown in Figure~\ref{fig:10}. Note that the fact that $Y_{\infty,\infty}(L)$ is not taut is reflected in the fact that the diagram has no intersection points (and may chosen to be admissible, by picking admissible diagrams of $(M_1,\Gamma_m)$ and $(M_2,\Gamma_n)$). Note that the alpha arc on the top of each diagram connects to the alpha arc on the bottom by passing through other portions of the diagram, and similarly for the beta arcs shown. In particular, there is a single alpha curve shown, not two alpha curves.

\begin{figure}[ht]
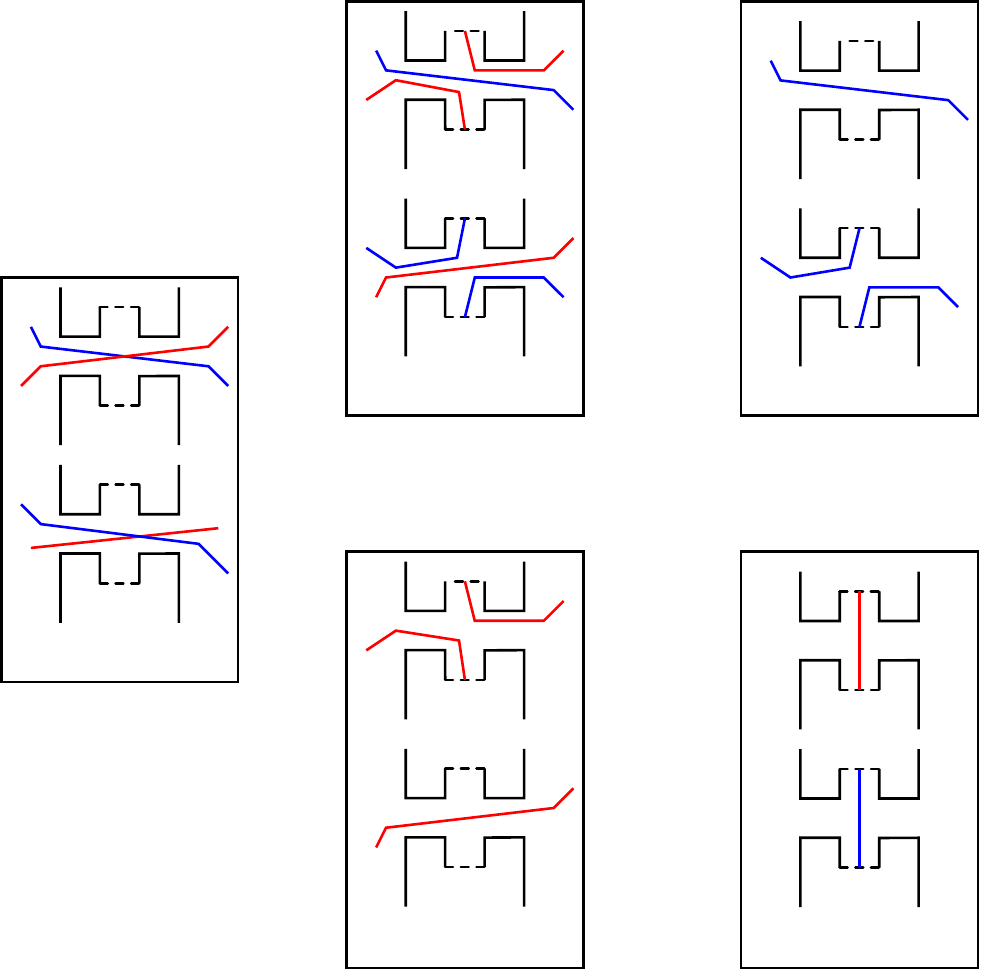
\caption{Heegaard diagrams for the manifolds appearing in~\eqref{eq:iterated-mapping-cone}.}
\label{fig:10}
\end{figure}

\begin{lem}\label{lem:diagramatic-isomorphisms} There are isomorphisms:
\begin{enumerate}
\item\label{claim:diagram-isomorphism-1} $\SFH(Y_{\infty,1}(L))\iso \SFH(M_1,\Gamma_n)\otimes \SFH(M_2,\Gamma_m)$
\item\label{claim:diagram-isomorphism-2} $\SFH(Y_{1,\infty}(L))\iso \SFH(M_1,\Gamma_n)\otimes \SFH(M_2,\Gamma_m)$.
\item\label{claim:diagram-isomorphism-3} $\SFH(Y_{1,1}(L))\iso \SFH(M_1,\Gamma_{n-1})\otimes \SFH(M_2,\Gamma_m)\oplus \SFH(M_1,\Gamma_n)\otimes \SFH(M_2,\Gamma_{m-1})$.
\end{enumerate}
\end{lem}
\begin{proof} Consider claims~\eqref{claim:diagram-isomorphism-1} and~\eqref{claim:diagram-isomorphism-2} first.  We consider the diagrams shown in Figure~\ref{fig:10}. These diagrams are both compound stabilizations (in the sense of \cite{JZContactHandles}*{Section~2.2}) of a diagram obtained by attaching two contact 1-handles to $(M_1,\Gamma_m)\sqcup(M_2,\Gamma_n)$. Since compound stabilizations and contact 1-handles have no effect on homology, the claim follows.

We now consider claim~\eqref{claim:diagram-isomorphism-3}. In this case, let $\a_0$ and $\b_0$ denote the alpha and beta curves which are shown in Figure~\ref{fig:10} for the manifold $Y_{1,1}(L)$. Suppose $\xs\in \bT_{\a}\cap \bT_{\b}$ on the diagram for $Y_{1,1}(L)$, and let $x_0$ denote the component in $\a_0$, and let $x_0'$ denote the component in $\b_0$. We claim that there are exactly two possible cases:
\begin{enumerate}
\item\label{item:x-case-1} $x_0$ and $x_0'$ are both on the left side of the Heegaard diagram.
\item\label{item:x-case-2} $x_0$ and $x_0'$ are both on the right side of the Heegaard diagram.
\end{enumerate}
In particular, it is not possible to form an intersection point where $x_0$ and $x_0'$ are on opposite sides of the diagram (due to the combinatorial constraints of forming a valid intersection point). Given an intersection point $\xs$ satisfying case~\eqref{item:x-case-1}, by cutting the diagram as in Figure~\ref{fig:18}, we may naturally view $\xs$ as an intersection point on a diagram for $(M_1,\Gamma_{m-1})\sqcup (M_2,\Gamma_n)$. Similarly any intersection point satisfying \eqref{item:x-case-2} may be viewed as an intersection point of $(M_1,\Gamma_n)\sqcup (M_2,\Gamma_{m-1})$. Furthermore, we obtain a direct sum decomposition at the chain level
\[
\CF(Y_{1,1}(L))\iso \CF(M_1,\Gamma_{n-1})\otimes \CF(M_2,\Gamma_m)\oplus \CF(M_1,\Gamma_n)\otimes \CF(M_2,\Gamma_{m-1}).
\]
Because of the presence of the boundary $\d \Sigma$ in the diagrams in Figure~\ref{fig:10}, holomorphic disks cannot pass into the region of $Y_{1,1}(L)$ shown therein. In particular, the above chain level decomposition is respected by the differential, and the proof is complete.
\end{proof}

\begin{figure}[ht]
\begingroup%
  \makeatletter%
  \providecommand\color[2][]{%
    \errmessage{(Inkscape) Color is used for the text in Inkscape, but the package 'color.sty' is not loaded}%
    \renewcommand\color[2][]{}%
  }%
  \providecommand\transparent[1]{%
    \errmessage{(Inkscape) Transparency is used (non-zero) for the text in Inkscape, but the package 'transparent.sty' is not loaded}%
    \renewcommand\transparent[1]{}%
  }%
  \providecommand\rotatebox[2]{#2}%
  \newcommand*\fsize{\dimexpr\f@size pt\relax}%
  \newcommand*\lineheight[1]{\fontsize{\fsize}{#1\fsize}\selectfont}%
  \ifx\svgwidth\undefined%
    \setlength{\unitlength}{228.96551441bp}%
    \ifx\svgscale\undefined%
      \relax%
    \else%
      \setlength{\unitlength}{\unitlength * \real{\svgscale}}%
    \fi%
  \else%
    \setlength{\unitlength}{\svgwidth}%
  \fi%
  \global\let\svgwidth\undefined%
  \global\let\svgscale\undefined%
  \makeatother%
  \begin{picture}(1,1.09858515)%
    \lineheight{1}%
    \setlength\tabcolsep{0pt}%
    \put(0,0){\includegraphics[width=\unitlength,page=1]{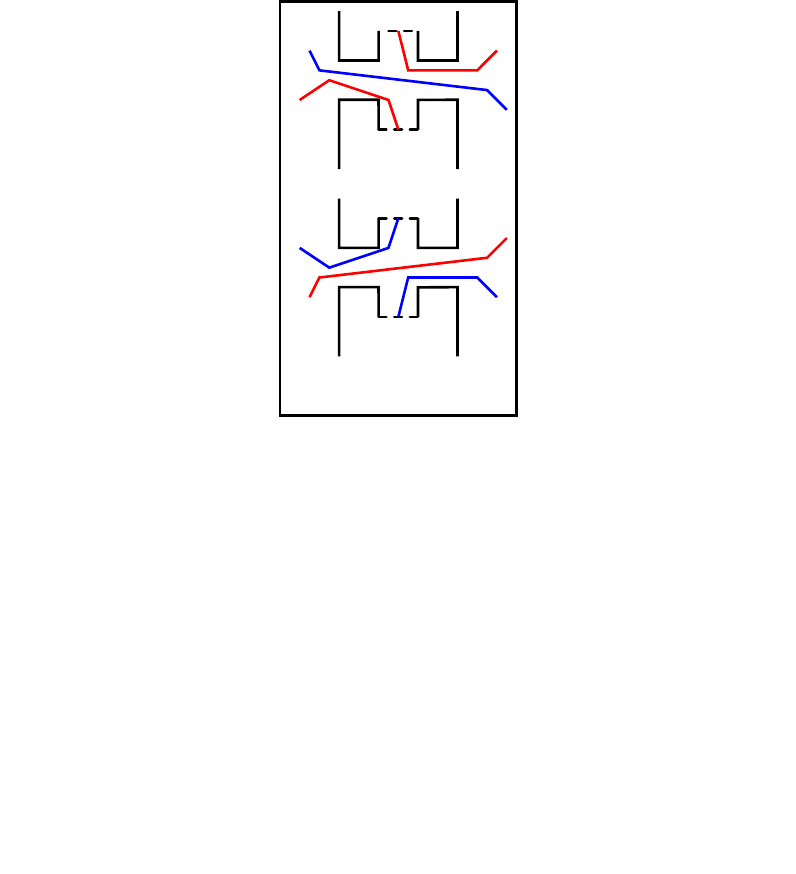}}%
    \put(0.49086152,0.60525078){\makebox(0,0)[t]{\lineheight{1.25}\smash{\begin{tabular}[t]{c}$Y_{1,1}(L)$\end{tabular}}}}%
    \put(0,0){\includegraphics[width=\unitlength,page=2]{fig18.pdf}}%
    \put(0.49995285,0.21948508){\makebox(0,0)[t]{\lineheight{1.25}\smash{\begin{tabular}[t]{c}$\oplus$\end{tabular}}}}%
  \end{picture}%
\endgroup%

\caption{Cutting the diagram in part~\eqref{claim:diagram-isomorphism-3} of Lemma~\ref{lem:diagramatic-isomorphisms}.}
\label{fig:18}
\end{figure}

\begin{lem}\label{lem:decomposition-of-maps} We may choose an orientation on $K$ so that the following holds with respect to the isomorphisms from Lemma~\ref{lem:diagramatic-isomorphisms}:
\[
F_{W(L_1^*)}=\phi_{m-1}^+|\id\oplus \id|\phi_{n-1}^+\qquad \text{and} \qquad F_{W(L_2^*)}=\phi_{m-1}^-|\id\oplus \id|\phi_{n-1}^-.
\]
\end{lem}
\begin{proof} We focus first on the claim about the map $F_{W(L_1^*)}$. Consider the special intersection point $c$ along the top of the diagram for $Y_{1,\infty}(L)$. Since $c$ is the only intersection of the alpha curve of a stabilization with any beta curve, $c$ must be an outgoing intersection point of any holomorphic triangle counted by $F_{W(L_1^*)}$. There are exactly two configurations the domain of a holomorphic triangle map have near $c$:
\begin{enumerate}[label=($\Delta$-\arabic*), ref=$\Delta$-\arabic*]
\item\label{case:triangle:left} The domain is supported to the left of $c$, but not to the right.
\item\label{case:triangle:right} The domain is supported to the right of $c$, but not to the left.
\end{enumerate}
See Figure~\ref{fig:19} for these two configurations.
With respect to the isomorphisms from Lemma~\ref{lem:diagramatic-isomorphisms}, the incoming intersection point of a holomorphic triangle satisfying case~\eqref{case:triangle:left} must lie in the summand $\CF(M_1,\Gamma_{m-1})\otimes \CF(M_2,\Gamma_n)$ of $\CF(Y_{1,1}(L))$. Analogously, the incoming intersection point of a holomorphic triangle satisfying case~\eqref{case:triangle:right} must lie in the summand $\CF(M_1,\Gamma_m)\otimes \CF(M_2,\Gamma_{n-1})$.  Furthermore, the diagram cutting operation of Figure~\ref{fig:18} may also be performed on the Heegaard triple for the map $F_{W(L_1^*)}$.  Using the description of the contact 2-handle map in terms of holomorphic triangles from Lemma~\ref{lem:ctct-2-handle=triangles}, we obtain the identifications
\[
F_{W(L_1^*)}|_{\CF(M_1,\Gamma_{m-1})\otimes \CF(M_2,\Gamma_n)}=\phi_{m-1}^{\pm}|\id\quad \text{and} \quad F_{W(L_1^*)}|_{\CF(M_1,\Gamma_{m})\otimes \CF(M_2,\Gamma_{n-1})}=\id|\phi_{n-1}^{\pm},
\]
where $\pm$ is consistent between the expressions.

An entirely analogous argument may be performed for $F_{W(L_2^*)}$, and we obtain
\[
F_{W(L_2^*)}|_{\CF(M_1,\Gamma_{m-1})\otimes \CF(M_2,\Gamma_n)}=\phi_{m-1}^{\mp}|\id\quad \text{and} \quad F_{W(L_2^*)}|_{\CF(M_1,\Gamma_{m})\otimes \CF(M_2,\Gamma_{n-1})}=\id|\phi_{n-1}^{\mp},
\]
where the $\mp$ above are consistent with the $\pm$ in the expression for $F_{W(L_1^*)}$.

The distinction between a positive bypass and a negative bypass is equivalent to an orientation of $K$. After making a choice, we obtain the lemma statement.
\end{proof}

\begin{figure}[ht]
\begingroup%
  \makeatletter%
  \providecommand\color[2][]{%
    \errmessage{(Inkscape) Color is used for the text in Inkscape, but the package 'color.sty' is not loaded}%
    \renewcommand\color[2][]{}%
  }%
  \providecommand\transparent[1]{%
    \errmessage{(Inkscape) Transparency is used (non-zero) for the text in Inkscape, but the package 'transparent.sty' is not loaded}%
    \renewcommand\transparent[1]{}%
  }%
  \providecommand\rotatebox[2]{#2}%
  \newcommand*\fsize{\dimexpr\f@size pt\relax}%
  \newcommand*\lineheight[1]{\fontsize{\fsize}{#1\fsize}\selectfont}%
  \ifx\svgwidth\undefined%
    \setlength{\unitlength}{170.83742169bp}%
    \ifx\svgscale\undefined%
      \relax%
    \else%
      \setlength{\unitlength}{\unitlength * \real{\svgscale}}%
    \fi%
  \else%
    \setlength{\unitlength}{\svgwidth}%
  \fi%
  \global\let\svgwidth\undefined%
  \global\let\svgscale\undefined%
  \makeatother%
  \begin{picture}(1,0.79967671)%
    \lineheight{1}%
    \setlength\tabcolsep{0pt}%
    \put(0,0){\includegraphics[width=\unitlength,page=1]{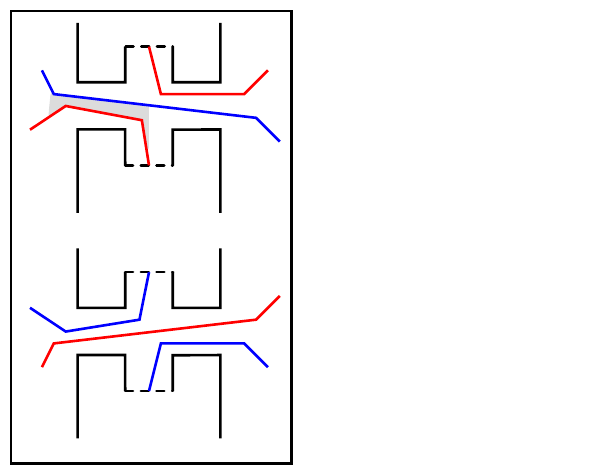}}%
    \put(0.23475382,0.05798639){\makebox(0,0)[t]{\lineheight{1.25}\smash{\begin{tabular}[t]{c}\eqref{case:triangle:left}\end{tabular}}}}%
    \put(0,0){\includegraphics[width=\unitlength,page=2]{fig19.pdf}}%
    \put(0.73240323,0.05798639){\makebox(0,0)[t]{\lineheight{1.25}\smash{\begin{tabular}[t]{c}\eqref{case:triangle:right}\end{tabular}}}}%
    \put(0,0){\includegraphics[width=\unitlength,page=3]{fig19.pdf}}%
  \end{picture}%
\endgroup%

\caption{The two possible configurations of a holomorphic triangle in Lemma~\ref{lem:decomposition-of-maps}.}
\label{fig:19}
\end{figure}

\subsection{Obtaining the derived tensor product}
\label{sec:reorganize-obtain-derived}

In this section, we prove Lemma~\ref{lem:obtain-direct-limit}, and show that the direct limit of the complexes
\begin{equation}
\begin{tikzcd}[labels=description, column sep=3cm, row sep=1.5cm]
C_n\otimes D_{m-1}\oplus C_{n-1}\otimes D_m
	\ar[r, "\id| \phi_m^-\oplus\phi_n^-| \id"]
	\ar[d,"\id| \phi_m^+\oplus \phi_n^+| \id"]
&
C_n\otimes D_m
\\
C_n\otimes D_m
&
\,
\end{tikzcd}
\label{eq:expanded-tensor-product}
\end{equation}
has the same homology as the direct limit of the complexes
\[
\begin{tikzcd}[column sep=3cm]
C_n\otimes D_m\ar[r, "\phi_n^-|\phi_m^+ +\phi_n^+| \phi_m^-"]& C_{n+1}|D_{m+1}.
\end{tikzcd}
\]
The proof follows by viewing the 2-dimensional complex as a mapping cone from the left column to the right column. Lemma~\ref{lem:staircase-lemma}, provides a map from $C_{n-1}\otimes D_{m-1}$ to the left column. This gives us a chain map from
\[
\begin{tikzcd}[column sep=4cm]
C_{n-1}\otimes D_{m-1}\ar[r, "\phi_{n-1}^-| \phi_{m-1}^+ +\phi_{n-1}^+| \phi_{m-1}^-"]& C_{n}|D_{m}
\end{tikzcd}
\]
to the complex in Equation~\eqref{eq:expanded-tensor-product}.
Similarly, there is also a map from the 2-dimensional complex above to
\begin{equation}
\begin{tikzcd}[column sep=3cm]
C_n\otimes D_m\ar[r, "\phi_n^-| \phi_m^+ +\phi_n^+| \phi_m^-"]& C_{n+1}\otimes D_{m+1}.
\end{tikzcd}
\label{eq:derived-tensor}
\end{equation}
It is straightforward to see that these maps induce isomorphisms on direct limits.

Finally, to identify these complexes with the derived tensor product, observe that the limiting groups $\HFK^-(K_1)$ and $\HFK^-(K_2)$ have an Alexander grading which is bounded from above. Furthermore, if $q\in \Z$ is fixed, then for sufficiently large $n\gg 0$, $(C_n)_{>q}\iso \HFK^-(K_1)_{>q}$. In particular, if $q$ is fixed and $n\gg 0$ and $m\gg 0$, then the complex in Equation~\eqref{eq:derived-tensor} is isomorphic to the derived tensor product complex in gradings above $q$. This completes the proof.

\section{Connected sums in the setting of instantons and monopoles}\label{sec:7}
In this section, we adapt our proof of the connected sum formula from Section~\ref{sec:connected-sum-HF} to setting the instanton and monopole Floer homologies. We work exclusively in the setting of instantons for concreteness, but our proofs work also in the setting of Heegaard and monopole Floer homology.

\subsection{Constructing the iterated mapping cone}

We now reformulate the argument from Section~\ref{sec:connected-sum-HF} in a way which avoids the use of Heegaard diagrams, and uses only formal properties about sutured Floer homology.

 In analogy to Equation~\eqref{eq:iterated-mapping-cone}, we have the following

\begin{thm}\label{thm:link-surgery-instanton}
There is a homotopy equivalence
\[
\SHI(Y_{0,0}(L))\simeq
\begin{tikzcd}[labels=description, column sep=1.5cm, row sep=1.5cm]
\SHI(Y_{1,1}(L))
	\ar[r, "F_{W(L_1^{*})}"]
	\ar[d,"F_{W(L_2^{*})}"]
	\ar[dr,dashed]
&
\SHI(Y_{\infty, 1}(L))
	\ar[d,"F_{W(L_2^{*})}"]
\\
\SHI(Y_{1,\infty}(L))
	\ar[r, "F_{W(L_1^{*})}"]
&
 \SHI(Y_{\infty,\infty}(L))
\end{tikzcd}
\]
Here $L_i^{*}$ is a $0$-framed dual of $L_i$. (Here we view $\SHI(Y_{9,0}(L))$ as being a chain complex with vanishing derivative).
\end{thm}

We now sketch the main ideas of the proof of Theorem~\ref{thm:link-surgery-instanton}. The first step is to use the iterated mapping cone construction applied to instanton Floer homology of the closures.  This gives an iterated mapping cone formula involving the instanton chain groups $C_*$. The iterated mapping cone construction in the context of instanton homology is described by Scaduto \cite{ScadutoOddKh}*{Section~6}. Next, we observe that the $\mu(R)$ and $\mu(p)$ operations on instanton Floer homology may be intertwined with the iterated mapping cone construction. This is because these operators may be interpreted in terms of ordinary cobordism maps by Remark~\ref{rem:localize-mu(R)}. Since $R$ is disjoint from the 2-handle attaching spheres used to construct the iterated mapping cone, the construction of metrics used to build the iterated mapping cone extends to yield hypercube morphisms $\mu(R)$ and $\mu(p)$ on the iterated mapping cone. Furthermore, the isomorphism from $\SFH(Y_{0,0}(L))$ with the iterated mapping cone intertwines $\mu(R)$ and $\mu(p)$ with these hypercube operators. Next, we take generalized eigenspaces with respect to these hypercube morphisms using the construction from Section~\ref{sec:simultaneous-eigenspaces}. This leaves the hypercube in the statement.

For convenience, we will typically also opt to take homology at each vertex. We can do this by the following argument:  Since each instanton chain chain group is finite dimensional over $\C$, it is homotopy equivalent to its homology. (In fact, it is easy to see that if $C$ is a finite dimensional chain complex over $\C$, then there is a strong deformation retraction of $C$ onto its homology). We may use homological perturbation theory for hypercubes (cf. \cite{HHSZDual}*{Section~2.7}) to find a homotopy equivalent hypercube with homology taken at each vertex.

The hypercube in Theorem~\ref{thm:link-surgery-instanton} may be further simplified, because $Y_{\infty,\infty}(L)=Y$ is not taut, and hence has trivial $\SHI$. See Figure~\ref{fig:10} for the Heegaard diagrams of the manifolds involved.

An important observation is that $Y_{1,1}(L)$, $Y_{1,\infty}(L)$ and $Y_{\infty,1}(L)$ are canonically diffeomorphic to $Y$ as 3-manifolds, but with sutures that differ by a Dehn twist along each component of $L$ with framing $+1$.
This is because performing Dehn surgery on a boundary parallel knot with $+1$ framing relative to the boundary framing corresponds to a boundary Dehn twist parallel to the knot. See Figure~\ref{fig:12} for a schematic.

\begin{figure}[ht]
\begingroup%
  \makeatletter%
  \providecommand\color[2][]{%
    \errmessage{(Inkscape) Color is used for the text in Inkscape, but the package 'color.sty' is not loaded}%
    \renewcommand\color[2][]{}%
  }%
  \providecommand\transparent[1]{%
    \errmessage{(Inkscape) Transparency is used (non-zero) for the text in Inkscape, but the package 'transparent.sty' is not loaded}%
    \renewcommand\transparent[1]{}%
  }%
  \providecommand\rotatebox[2]{#2}%
  \newcommand*\fsize{\dimexpr\f@size pt\relax}%
  \newcommand*\lineheight[1]{\fontsize{\fsize}{#1\fsize}\selectfont}%
  \ifx\svgwidth\undefined%
    \setlength{\unitlength}{197.67885248bp}%
    \ifx\svgscale\undefined%
      \relax%
    \else%
      \setlength{\unitlength}{\unitlength * \real{\svgscale}}%
    \fi%
  \else%
    \setlength{\unitlength}{\svgwidth}%
  \fi%
  \global\let\svgwidth\undefined%
  \global\let\svgscale\undefined%
  \makeatother%
  \begin{picture}(1,1.91964888)%
    \lineheight{1}%
    \setlength\tabcolsep{0pt}%
    \put(0,0){\includegraphics[width=\unitlength,page=1]{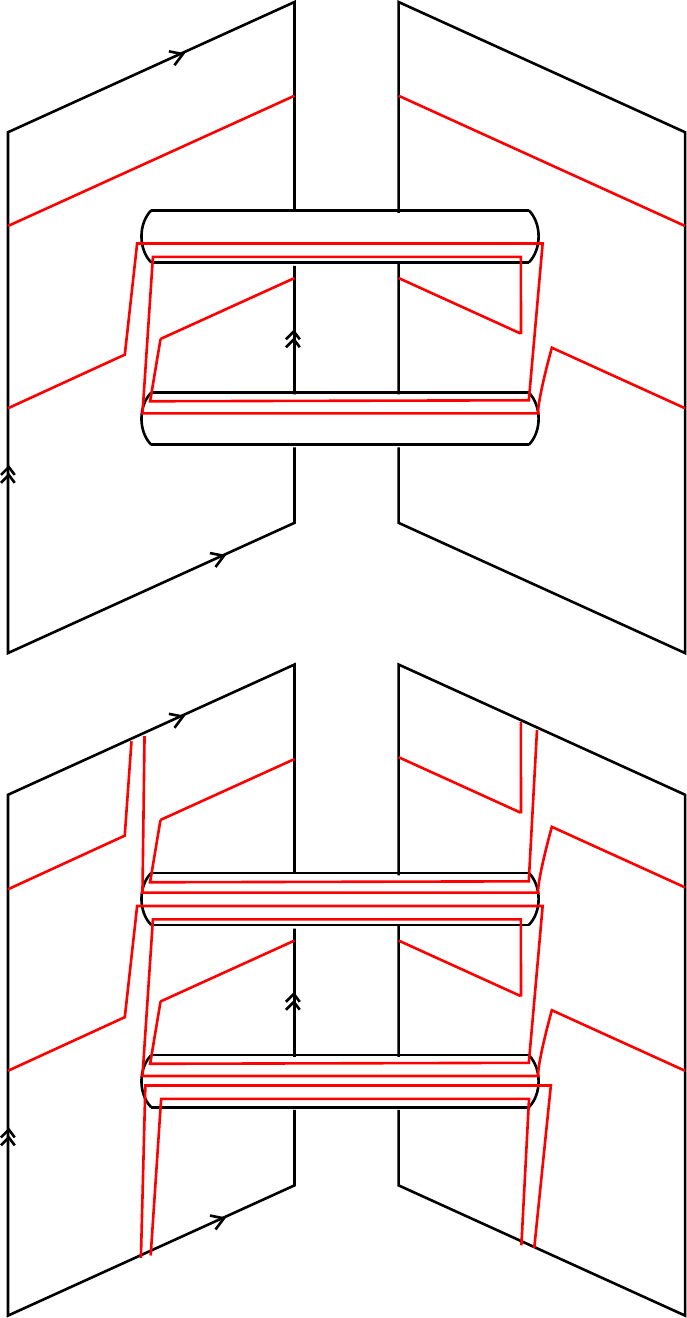}}%
    \put(0.49540884,1.09360873){\makebox(0,0)[t]{\lineheight{1.25}\smash{\begin{tabular}[t]{c}$Y_{1,\infty}(L)$\end{tabular}}}}%
    \put(0.49540884,0.08663942){\makebox(0,0)[t]{\lineheight{1.25}\smash{\begin{tabular}[t]{c}$Y_{1,1}(L)$\end{tabular}}}}%
    \put(0,0){\includegraphics[width=\unitlength,page=2]{fig12.pdf}}%
  \end{picture}%
\endgroup%

\caption{The manifolds $Y_{\infty,1}(L)$ and $Y_{1,1}(L)$. Also shown are the decomposing disks $D_1$ and $D_2$ (in the bottom figure).}
\label{fig:12}
\end{figure}

\subsection{Decomposing disks}

We now compute the groups and maps appearing in the iterated mapping cone of Theorem~\ref{thm:link-surgery-instanton}.  In analogy to Lemma~\ref{lem:reorganize-cone-Heegaard}, we have the following:

\begin{lem}\label{lem:reorganize-cone}
There is a homotopy equivalence of hypercubes between the diagram
\begin{equation}
\begin{tikzcd}[labels=description, column sep=1.5cm, row sep=1.5cm]
\SHI(Y_{1,1}(L))
	\ar[r, "F_{W(L_1^{*})}"]
	\ar[d,"F_{W(L_2^{*})}"]
&
\SHI(Y_{\infty, 1}(L))
\\
\SHI(Y_{1,\infty}(L))
&
\,
\end{tikzcd}
\label{eq:iterated-mapping-cone-1}
\end{equation}
and the diagram
\begin{equation}
\begin{tikzcd}[labels=description, column sep=3cm, row sep=1.5cm]
C_n\otimes D_{m-1}\oplus C_{n-1}\otimes D_m
	\ar[r, "\id| \phi_m^-\oplus\phi_n^-| \id"]
	\ar[d,"\id| \phi_m^+\oplus \phi_n^+| \id"]
&
C_n\otimes D_m
\\
C_n\otimes D_m
&
\,
\end{tikzcd}
\end{equation}
In the above, $C_n$ denotes $\SHI(M_1,\Gamma_n)$ and $D_m$ denotes $\SHI(M_2,\Gamma_m)$. 
\end{lem}

\begin{proof}
As a first step, we observe that the sutured manifolds $Y_{\infty,1}(L)$ and $Y_{1,\infty}$ are obtained by attaching contact 1-handles to $(M_1,\Gamma_n)\sqcup (M_2,\Gamma_m)$. In particular, their homologies are isomorphic to $\SHI(M_1,\Gamma_n)\otimes \SHI(M_2,\Gamma_m)$. We define the the component of the hypercube map in the statement from $\SHI(Y_{1,\infty})$ to $C_n\otimes D_m$ will be the isomorphism map for cutting along a contact 1-handle. The map from $\SHI(Y_{\infty,1})$ is similarly the inverse of a contact 1-handle map.

Defining the morphism from $\SHI(Y_{11}(L))$ to $C_n\otimes D_{m-1}\oplus C_{n-1}\otimes D_m$ is more complicated. In $Y_{1,1}$, there are two disks, which each intersect the dividing set in four points. See Figure~\ref{fig:12}. It is possible to decompose along these two disks, however they intersect $L_1^{*}$ and $L_2^{*}$, which makes analyzing hypercube differentials challenging.

Instead, we argue by first modifying the 3-manifold before decomposing. Firstly, attach four contact 1-handles to $Y_{1,1}(L)$, as in Figure~\ref{fig:42}, and call the resulting manifold $Y_{1,1}'(L)$. Similarly, let $Y_{1,\infty}'(L)$ and $Y_{\infty,1}'(L)$ be obtained by attaching four contact 1-handles to $Y_{1,\infty}(L)$ and $Y_{\infty,1}(L)$, respectively.  For the purposes of placing the contact 1-handles, it is convenient to think of $Y_{\infty,1}'(L)$ as $Y_{1,1}'(L)(L_1^{*})$. 

\begin{figure}[ht]
\begingroup%
  \makeatletter%
  \providecommand\color[2][]{%
    \errmessage{(Inkscape) Color is used for the text in Inkscape, but the package 'color.sty' is not loaded}%
    \renewcommand\color[2][]{}%
  }%
  \providecommand\transparent[1]{%
    \errmessage{(Inkscape) Transparency is used (non-zero) for the text in Inkscape, but the package 'transparent.sty' is not loaded}%
    \renewcommand\transparent[1]{}%
  }%
  \providecommand\rotatebox[2]{#2}%
  \newcommand*\fsize{\dimexpr\f@size pt\relax}%
  \newcommand*\lineheight[1]{\fontsize{\fsize}{#1\fsize}\selectfont}%
  \ifx\svgwidth\undefined%
    \setlength{\unitlength}{396.75117241bp}%
    \ifx\svgscale\undefined%
      \relax%
    \else%
      \setlength{\unitlength}{\unitlength * \real{\svgscale}}%
    \fi%
  \else%
    \setlength{\unitlength}{\svgwidth}%
  \fi%
  \global\let\svgwidth\undefined%
  \global\let\svgscale\undefined%
  \makeatother%
  \begin{picture}(1,0.65560482)%
    \lineheight{1}%
    \setlength\tabcolsep{0pt}%
    \put(0,0){\includegraphics[width=\unitlength,page=1]{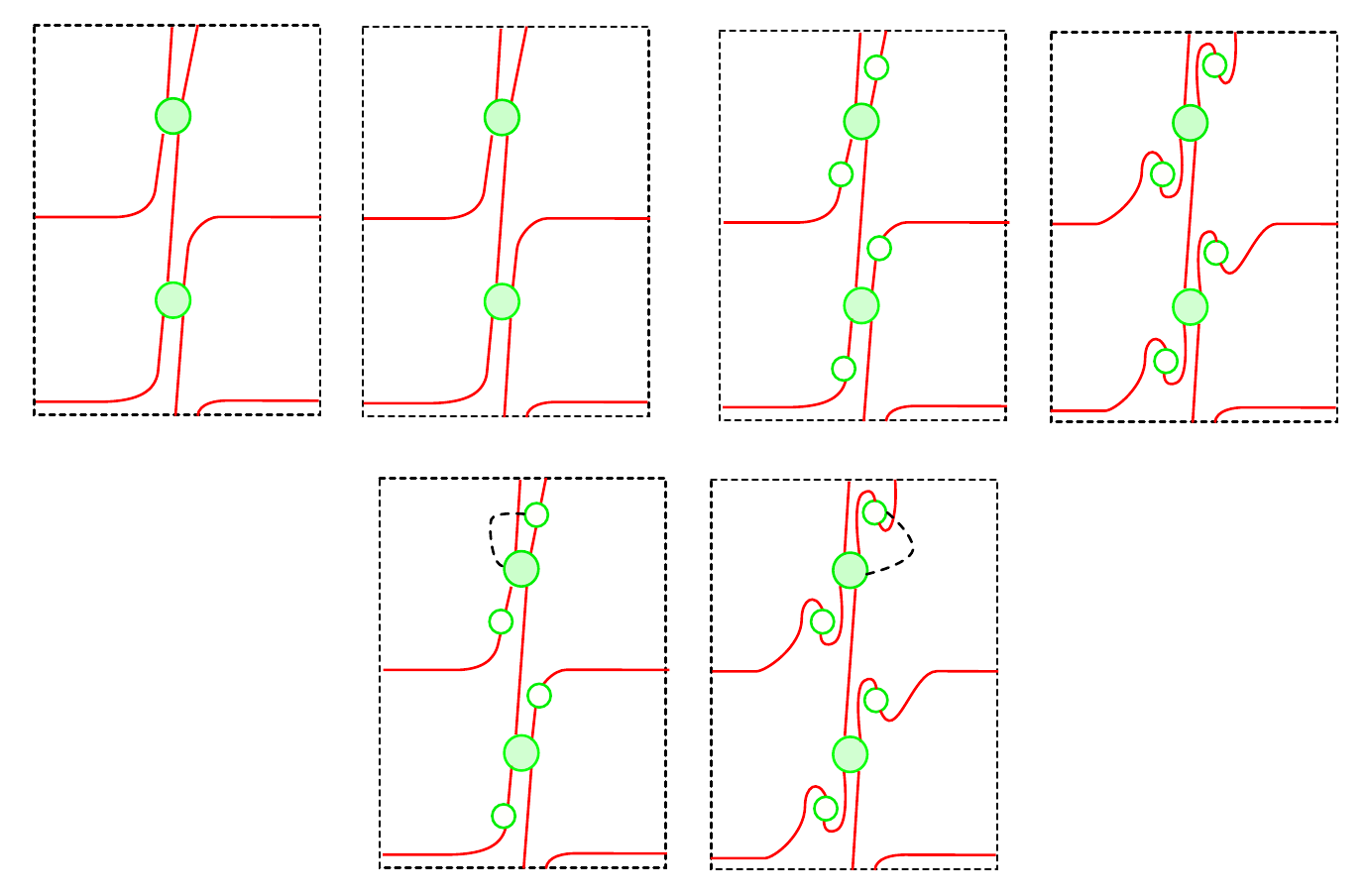}}%
    \put(0.6745797,0.25444289){\makebox(0,0)[lt]{\lineheight{1.25}\smash{\begin{tabular}[t]{l}$c_2^1$\end{tabular}}}}%
    \put(0.4181192,0.21487415){\makebox(0,0)[lt]{\lineheight{1.25}\smash{\begin{tabular}[t]{l}$c_2^2$\end{tabular}}}}%
    \put(0.35351238,0.26166041){\makebox(0,0)[rt]{\lineheight{1.25}\smash{\begin{tabular}[t]{r}$c_2^1$\end{tabular}}}}%
    \put(0.57047467,0.21487421){\makebox(0,0)[rt]{\lineheight{1.25}\smash{\begin{tabular}[t]{r}$c_2^2$\end{tabular}}}}%
    \put(0,0){\includegraphics[width=\unitlength,page=2]{fig42.pdf}}%
  \end{picture}%
\endgroup%

\caption{Top left: $Y_{1,1}(L)$. Top right: $Y_{1,1}'(L)$. Bottom: A contact 2-handle which cancels one of the contact 1-handles. The handles are glued via mirror reflection across the vertical axis.}
\label{fig:42}
\end{figure}

Since contact 1-handles induce isomorphisms on Floer homology which commute with the cobordism maps, the chain complex in Equation~\eqref{eq:iterated-mapping-cone-1} is chain isomorphic to the following chain complex
\[
\begin{tikzcd}[labels=description, column sep=1.5cm, row sep=1.5cm]
\SHI(Y_{1,1}'(L))
	\ar[r, "F_{W(L_1^{*})}"]
	\ar[d,"F_{W(L_2^{*})}"]
&
\SHI(Y_{\infty, 1}'(L))
\\
\SHI(Y_{1,\infty}'(L))
&
\,
\end{tikzcd}
\]
Next, we note that we may cancel contact 1-handles with contact 2-handles. The contact 2-handle map is defined as the composition of a 4-dimensional 2-handle map, followed by the inverse of a contact 1-handle map. Since the map for a contact 2-handle which cancels a contact 1-handle is an isomorphism, we may handleslide $L$ across the contact 2-handle maps without changing the homotopy type of the hypercube. In more detail, $c_2^1,c_2^2,c_2^3,c_2^4$ be the contact 2-handles which cancel the contact 1-handles from earlier. Let $F_{C_2}$ be the cobordism map for surgering on all four 2-handles, composed with the inverse of 4 contact 1-handle maps. Hence, the following diagram commutes
\[
\begin{tikzcd}\SHI(Y'_{1,1})\ar[r, "F_{C_2}"] \ar[d, "F_{W(J_2)}"]&\SHI(Y_{1,1}) \ar[d, "F_{W(L_2^{*})}"]\\
\SHI(Y_{1,\infty}')\ar[r, "F_{C_2}"]& \SHI(Y_{1,\infty})
\end{tikzcd}.
\]
In the above, $J_2$ is obtained by handlesliding $L_2^{*}$ across two contact 2-handles, as in Figure~\ref{fig:43}. The above diagram commutes because the compositions $F_{W(L_2^{*})}\circ F_{C_2}$ and $F_{C_2}\circ F_{W(J_2)}$ represent topologically equivalent cobordisms. A similar diagram holds for $W(L_2^{*})$ and a similar map $F_{W(J_1)}$.

Inside of $Y_{1,1}'$, there are two disks $D_1$ and $D_2$, which are disjoint from $J_1$ and $J_2$, and furthermore each intersect the dividing set in four points. We may decompose along these two disks using \cite{GhoshLiDecomposing}*{Corollary~4.3}, \cite{BaldwinLiYeHeegaard}*{Proposition~3.3}. When we do this, the sutured Floer homology of $Y_{1,1}'$ decomposes into four summands, corresponding to the four choices of orientation on $D_1\sqcup D_2$. For two orientations, the resulting manifold is not taut, and hence the sutured Floer homology vanishes. It is clear that after deleting contact 1 handles, the remaining two sutured manifolds are $(M_1,\Gamma_{n-1})\sqcup (M_2,\Gamma_{m})$ and $(M_1,\Gamma_n)\sqcup (M_2,\Gamma_{m-1})$.  See Figure~\ref{fig:13}.
\begin{figure}[ht]
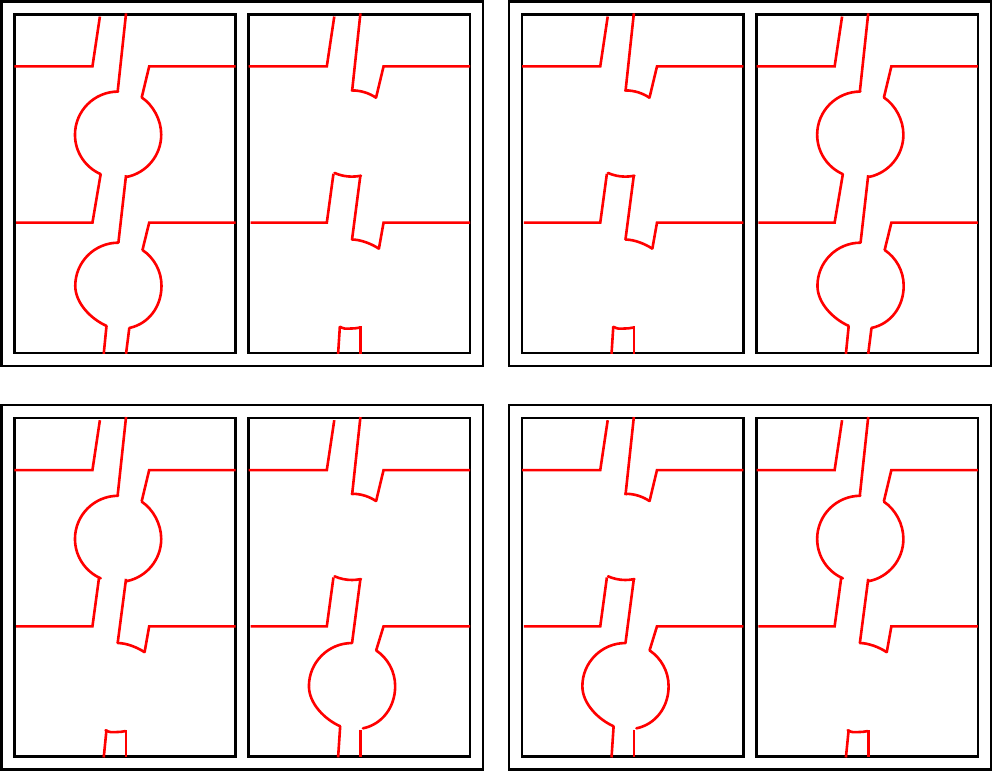
\caption{The possible dividing sets  obtained by decomposing $Y_{11}(L)$ on $D_1$ and $D_2$. The bottom two dividing sets contain null-homotopic closed curves and hence are not taut.}
\label{fig:13}
\end{figure}

In fact, Kronheimer and Mrowka constructed an inclusion map \cite[Theorem~6.8]{KMSutures}
\[
C_n\otimes D_{m-1}\hookrightarrow \SHI(Y_{1,1}'(L)).
\]
Their work shows that the sutured manifolds $(M_1,\Gamma_n)\sqcup (M_2,\Gamma_{m-1})$ and $\SHI(Y_{1,1}'(L))$ have a common closure, and the inclusion is a restriction of $\Spin^c$ structures. We claim that the following diagrams commute:
\[
\begin{tikzcd}[labels=description] C_n\otimes D_{m-1}\ar[r,hookrightarrow]\ar[dr, "\id| \phi_m^+"] & \SHI(Y_{1,1}')\ar[d, "F_{W(J_2)}"]\\
& \SHI(Y_{\infty,1})\iso C_n\otimes D_m
\end{tikzcd}
\quad
\begin{tikzcd}[labels=description] C_{n-1}\otimes D_{m}\ar[r,hookrightarrow]\ar[dr, "\phi_n^+| \id"] & \SHI(Y_{1,1}')\ar[d, "F_{W(J_2)}"]\\
& \SHI(Y_{\infty,1})\iso C_n\otimes D_m
\end{tikzcd},
\]
and that analogous diagrams commute with $F_{W(J_1)}$ (but with negative bypass maps instead of positive ones). Commutativity of the above diagrams is proven by the manipulation shown in Figure~\ref{fig:43}. Therein, we show the knot components $J_1$ and $J_2$ after decomposing along $D_1$ and $D_2$. We handleslide the contact 1-handles as indicated with the arrows.  After we do this, $J_1$ becomes the attaching cycle of the 2-handle for a negative  bypass while $J_2$ becomes the attaching cycle for a positive bypass.
\end{proof}

\begin{figure}[ht]
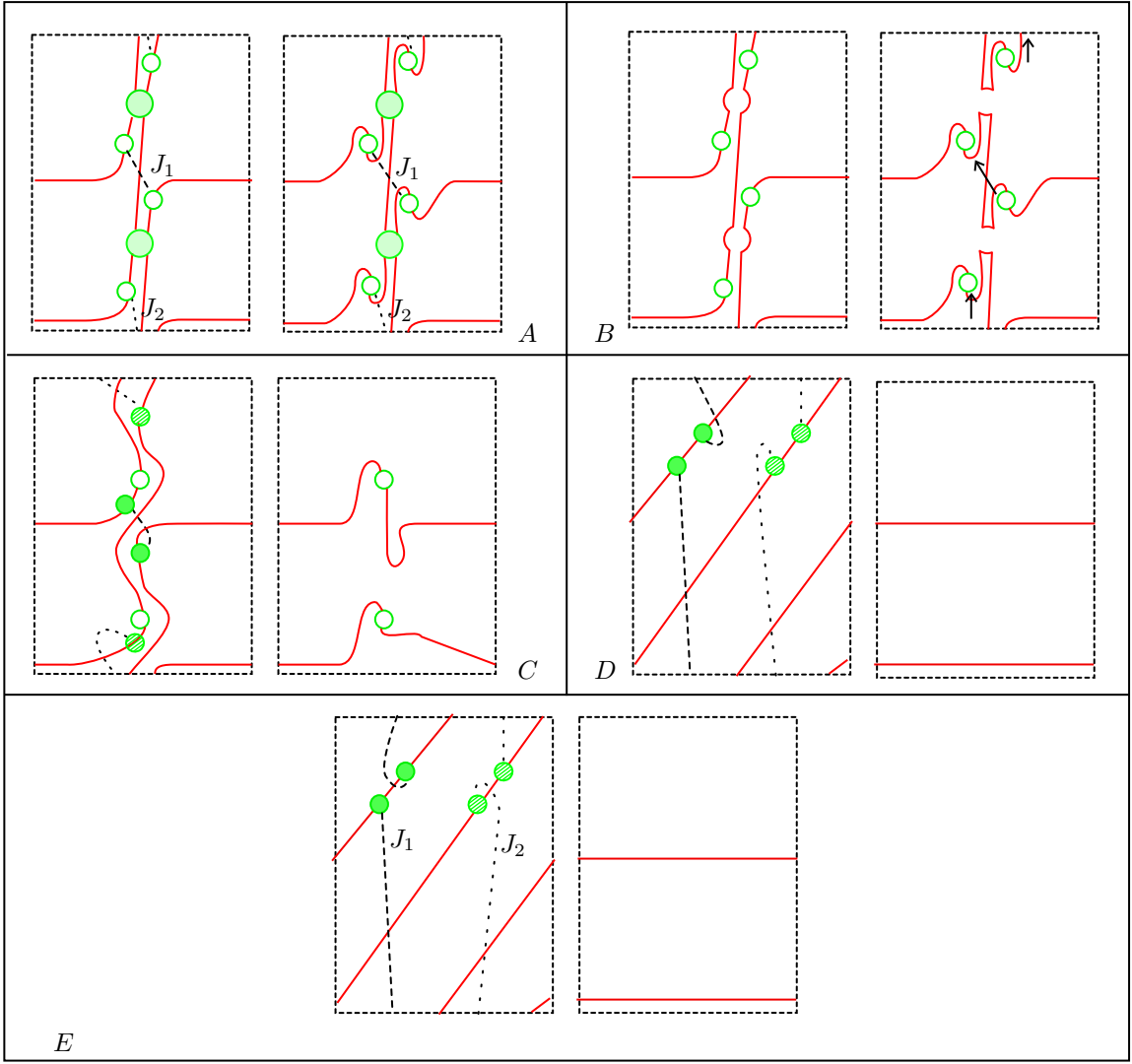
\caption{Circles denote 1-handles. Open circles are always glued to the mirrow circle. A: The manifold $Y_{1,1}'(L)$ and the knots $J_1$ and $J_2$. B: The manifold obtained by decomposing $Y_{1,1}'(L)$ along $D_1$ and $D_2$. The arrows indicate directions to handleslide 1-handles. C: The result of handlesliding 1-handles. D: The presentation obtained by isotopy and deletion of contact 1-handles. E: Changing from surface framing $-1$ to surface framing $0$ by performing a Dehn twist along the meridians of the contact 1-handles. }
\label{fig:43}
\end{figure}

\begin{rem}Identifying a particular 2-handle map with the positive or negative bypass requires a choice of orientation of the knot $K$. Note that the complex in Lemma~\ref{lem:reorganize-cone} is independent of a choice of orientation of $K_i$, since the positive and negative bypasses appear in a symmetric manner.
\end{rem}

\section{Gradings}\label{sec:8}
The main theorem of this section is the following:
\begin{thm}
\label{thm:main-gradings}
The (instanton, monopole, Heegaard Floer) limit Floer homology $\HFK^-(K_1\#K_2)$ coincides with the derived tensor product
\[
\HFK^-(K_1)\tildeotimes_{\bF[U]} \HFK^-(K_2)
\]
as a graded group.
\end{thm}

In the above, if $M$ and $N$ are Alexander graded modules over $\bF[U]$, we equip the derived tensor product $M\tildeotimes N$ with the Alexander grading
\[
\Cone(U\otimes \id-\id\otimes U'\colon M\otimes N [-1] \to M\otimes N).
\]

The main technical step in our proof of the above theorem is the following result:
\begin{prop}\label{thm:gradings}
With respect to the gradings below and the following shift from Equation \eqref{Eq: shift limit}, the quasi-isomorphism
\begin{equation}
X_{n, m}=\begin{tikzcd}
\SHI \left(Y_{0, 0}, \Gamma_{n+m}, S_{n+m}^{\tau({n+m})}\right)
\arrow[drr, dashed ]
\arrow[ddr, dashed]
\arrow[dr] & & \\
& (C_n \otimes D_{m-1} \oplus C_{n-1} \otimes D_m) [-1]  \arrow[r] \arrow[d]
& C_n \otimes D_m  \\
& C_n \otimes D_m [-1]
\end{tikzcd}
\label{eq:map-X-n-m}
\end{equation}
from the iterated mapping cone construction is grading preserving. In the above, $\tau\colon \Z\to \Z$ is the function $\tau(x)=0$ if $x$ is odd, and $\tau(x)=1$ is $x$ is even.
\end{prop}
In the above, if $M$ is a vector space, then $M[n]$ denotes the graded vector space $M\otimes \bF_{(n)}$, where $\bF_{(n)}\iso \bF$, concentrated in grading $n$.

We now describe a grading on $\SHI(Y_{0,0}(L),\Gamma_{n+m})$ which is compatible with the isomorphisms in Lemma~\ref{lem:reorganize-cone}. Recall that 
\[
C_n:=\SHI (Y\setminus \nu(K_1), \Gamma_{n})\quad \text{and} \quad D_m:=\SHI (Y\setminus \nu(K_2), \Gamma_{m}).
\]
 Following \cite{LiLimits}, each of $\SHI (Y_{0, 0} (L), \Gamma_{n+m})$, $C_n$, $C_{n-1}$, $D_n$ and $D_{n-1}$ has a natural Alexander grading which is compatible with the direct limit construction, in the sense that the positive bypass maps are grading preserving, and the negative bypass maps shift grading by $-1$. In particular, $\varinjlim C_n \iso \HFK^-(Y, K)$ as a graded $\bF[U]$-module, and similarly for $D_n$ and $\SHI(Y_{0,0}(L),\Gamma_{n+m})$. 
 
 We recall more generally, Li's construction gives an Alexander grading for any choice of Seifert surface.
 If $n\ge 0$, we write $S_n \subset Y \setminus \nu(K)$ for a Seifert surface which intersects $\Gamma_n$ in $2n$ points, each positively. If $\tau\in \Z$, we write $S_n^\tau$ for the surface $S_n$, which is stabilized $\tau$ times algebraically. See Subsection~\ref{sub: grading} for more background.

\subsection{Background on gradings}\label{sub: grading}

Baldwin and Sivek \cite{baldwin2021khovanov}*{Section~3} defined a grading on sutured monopole and instanton homology associated to a properly embedded surface $S\subset M$, such that $\d S$ is connected and intersects the sutures in two points. Their construction was generalized by Li \cite{LiLimits}*{Section~3} to the case where $\d  S$ is connected and intersects the sutures in $2n$ points. Kavi \cite{kavi2019cutting}
generalized the construction further to the case where $\d S$ disconnected. In this section, we review some properties of the grading.

\begin{figure}[ht]
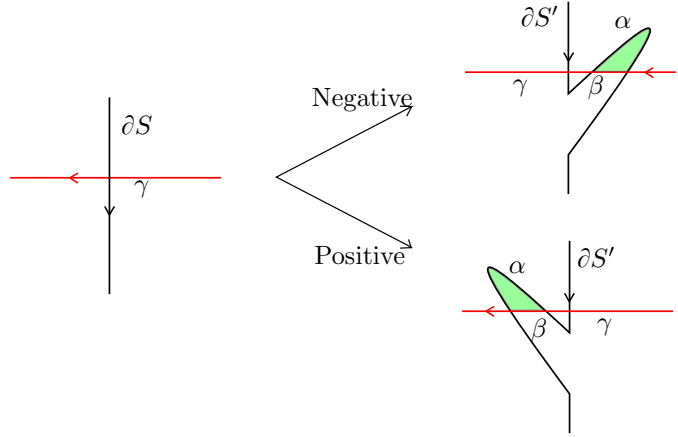
\caption{The positive and negative stabilizations of $S$. Here, $\a$ and $\b$ form the boundary of the shaded bigon.}
\label{fig:22}
\end{figure}

\begin{define}

 Suppose $(M, \gamma)$ is a balanced sutured manifold, and $S$ is an oriented and
properly embedded surface. A \emph{stabilization} of $S$ is another surface $S'$, obtained by a performing an isotopy to $S$ to create a new pair of intersection points, so that
\[
 \partial S' \cap \gamma = ( \partial S \cap \gamma) \cup \{p_+, p_{-}\}.
 \]
We require that there are arcs $\alpha \subset \partial S'$ and $\beta\subset \gamma$, which are oriented consistently with respect to $\partial S'$ and $\gamma$, respectively, so that the following hold:
\begin{enumerate}
\item We have $\partial \alpha= \partial \beta= \{p_+, p_{-}\}$.
\item The curves $\alpha$ and $\beta$ cobound a disk $D$ so that ${\rm int}(D) \cap (\gamma \cup \partial S') = \emptyset$.
\end{enumerate}
   
 The stabilization is called {\it negative} if $D$ can be oriented so that
 $\partial D= \alpha \cup \beta$ as oriented curves. It is called {\it positive} if $\partial D= (-\alpha) \cup \beta$.
\end{define}

Denote by $S^{\pm k}$ the result of performing $k$ many positive or negative stabilizations of $S$.
Sometimes we will write $S^-$ for $S^{-1}$ and $S^+$ for $S^{+1}$. 

\begin{lem}[\cite{LiLimits}*{Lemma~4.2}]\label{lem: decomposition under stabilization}
 Suppose $(M, \gamma)$ is a balanced sutured manifold, and $S$ is a properly embedded and oriented surface. Then the following hold:
 \begin{enumerate}
 \item If we decompose $(M, \gamma)$ along $S$ or $S^+$, then the resulting two balanced sutured manifolds are diffeomorphic.
\item If we decompose $(M, \gamma)$ along $S^-$, then the resulting balanced sutured manifold $(M', \gamma')$ is not taut, because $R_{+}(\gamma')$ and $R_{-}(\gamma')$ are both compressible.
\end{enumerate}
\end{lem}

\begin{define}\label{admissibility of surfaces}
Suppose $(M,\gamma)$ is a balanced sutured manifold, and $S\subset M$ is a properly embedded surface. Suppose further that $\d S$ intersects $\gamma$ transversely. We say $S$ is \emph{admissible} if every component of $\partial S$ intersects $\gamma$ and the value $(\frac{1}{2}|S\cap \gamma|-\chi(S))$ is an even integer.
\end{define}
Kavi \cite{kavi2019cutting} and Li \cite{LiLimits}*{Theorem~5.5} proved that for a fixed, oriented, admissible surface $S\subset M$, the grading $\SHI(M,\gamma,S)$ is well-defined.

We now consider the case when $M$ is a knot complement and $\g$ are longitudinal sutures. To construct the Alexander grading on $\HFK^-(Y, K)$, Li \cite{LiLimits} defines the following grading shift:
\[
     \SHI (Y\setminus \nu(K), \Gamma_n, S_n^{\tau (n)}) [\sigma (n)]
   \]
where, \begin{equation}\label{Eq: shift limit}
    \sigma (n) = -\frac{n-1+\tau(n)}{2}.
 \end{equation} Here, $\tau(n)=1$ if $n$ is even and $\tau(n)=0$ if $n$ is odd. Under this shift, the positive bypass maps, $\phi^+_n$, are grading preserving, and the negative bypass maps, $\phi^-_n$, shift the grading by $-1$.  Li \cite{LiLimits} defines $\HFK^-(Y, K)$ as the direct limit of the following directed system:
\[
\cdots \SHI (Y\setminus \nu(K), \Gamma_n, S_n^{\tau (n)}) [\sigma (n)] \xrightarrow{\phi^+_n} \SHI (Y\setminus \nu(K), \Gamma_{n+1}, S_n^{\tau (n+1)}) [\sigma (n+1)] \cdots
.\]
The negative bypass maps, $\phi^-_n$, induce the $U$ action on $\HFK^-(Y, K)$.

We now recall the effect of positive and negative stabilizations on gradings, and proven in  \cite{LiLimits}*{Proposition~4.17} and \cite{GhoshLiDecomposing}*{Proposition~4.1}:

\begin{thm}\label{thm: shift stabilization}
 Let $S$ be an admissible surface in a balanced sutured manifold $(M, \gamma)$. Suppose $T$
is obtained from $S$ by $p$ positive stabilizations and $q$ negative stabilizations with $p-q=2k$. Then for any $i \in \mathbb{Z}$,  \[\SHI( M, \gamma, S, i) = \SHI (M, \gamma, T, i+k).\]
\end{thm}

We will also use a generalization of the above result, due to Wang \cite{wang2020cosmetic}*{Lemma~4.42}:

\begin{thm} \label{thm: Generalized shift}
Let $S$ and $T$ be admissible surfaces in a balanced sutured manifold $(M, \gamma)$ which are in the same
relative homology class. Suppose that $\partial T$ and $\partial S$ are disjoint, and that $\partial T - \partial S$  is the boundary of
a subsurface $F$ of $\partial M$, where the orientation of $F$ matches that of $\partial M$. Then $S$ and $T$ determine the
same relative $\mathbb{Z}$-grading on $\SHI (M, \gamma)$. In particular 
\[
\SHI (M, \gamma, S, i) = \SHI (M, \gamma, T, i+k)
\] where $k = \chi(F \cap R_-(\gamma)) - \chi(F \cap R_+(\gamma)).$
\end{thm}

\subsection{Overview of the proof of Theorem \ref{thm:gradings}}
 We now provide an overview of the proof of Theorem~\ref{thm:gradings}. We recall that the quasi-isomorphism $X_{n, m}$ appearing in the statement is the composition of two maps, $F_{n,m}$ and $I_{n,m}$. The map $F_{n,m}$ is the homotopy equivalence described in Theorem~\ref{thm:link-surgery-instanton}, and the map $I_{n,m}$ is the homotopy equivalence described in Theorem~\ref{lem:reorganize-cone}. We will define grading shifts on the domains and codomains of these maps, with respect to which they are grading preserving. To define these shifts, recall that $\sigma\colon \Z\to \Z$ is the function defined in Equation~\eqref{Eq: shift limit}. Define $\zeta(n, m) = \sigma(n) + \sigma(m)$.

 \begin{prop}\label{prop: F_{m, n}} The map $F_{n,m}$, shown as the dashed arrows in the diagram below, is grading preserving with respect to the gradings shown below:
 \begin{equation}\label{equation:I}
 F_{n, m}=\begin{tikzcd}[labels=description]
 \SHI(Y_{0,0}(L)) [\sigma(n+m)]
 \arrow[drr, dashed,  "{h}"]
 \arrow[ddr, dashed,"{j}"]
 \arrow[dr, "{F}",dashed, description] & & \\
 & \SHI(Y_{1,1}(L)) [\sigma(n+m)] \arrow[r] \arrow[d]
 & \SHI(Y_{1,\infty}(L)) [\sigma(n+m)] \\
 & \SHI(Y_{\infty,1}(L))  [\sigma(n+m)]
 \end{tikzcd}
 \end{equation}
 On each sutured Floer complex in Equation~\eqref{equation:I}, we define the gradings using a properly embedded surface induced by the Seifert surface $S_{n+m}^{\tau(n+m)}$ of $K_1 \# K_2$, which intersects the sutures $\Gamma_{n+m}$ in $2(n+m+\tau(n+m))$ points.
 \end{prop}

 Next, we consider the map $I_{n,m}$:

 \begin{prop}\label{prop: I_{m, n}}
 The map $I_{n,m}$, shown as the dashed arrows in the following diagram, is grading preserving with respect to the grading shifts shown below:
 \begin{equation}\label{equation:II}
 \begin{tikzcd}[column sep = {5cm,between origins}, row sep = .7 cm, labels=description]
 \SHI(Y_{1,1}(L)) [\sigma(n+m)]
 	\arrow[rdd, "{i_{1, 1}}", description,dashed]
 	\arrow[rr]
 	\arrow[d]
 &&[-3.7cm]
 \SHI(Y_{1,\infty}(L)) [\sigma(n+m)]
 	\arrow[rdd, "{i_{1, \infty}}", description, dashed]  &   
 \\
 \SHI(Y_{\infty, 1}(L)) [\sigma(n+m)]
 	\arrow[rdd, "{i_{\infty, 1}}", pos=.2,dashed]                     
 &
 \\
 &C_n \otimes D_{m-1} [\zeta(n, m-1)-1] \oplus C_{n-1} \otimes D_m[\zeta(n-1, m)-1] 
 	\arrow[rr]
 	\arrow[d] 
 && C_n \otimes D_m [\zeta(n, m)] \\
  & C_n \otimes D_m [\zeta(n, m)-1]                     &  
 \end{tikzcd}
 \end{equation}
  On the domain of $i_{1, 1}$, $i_{\infty, 1}$ and $i_{1, \infty}$ in Equation~\eqref{equation:II}, we use the gradings induced by the Seifert surface $S_{n+m}^{\tau(n+m)}$. On the codomain of $i_{1, 1}$, we use the grading surface $S_{n-1}^{\tau(n-1)} \sqcup S_{m}^{\tau(m)}$ on the summand $C_{n-1} \otimes D_m$ and  $S_{n}^{\tau(n)} \sqcup S_{m-1}^{\tau(m-1)}$ on the summand $C_n \otimes D_{m-1}$.  On $C_n \otimes D_m$, we use the grading surface $S_n^{\tau(n)} \sqcup S_m^{\tau(m)}$.
 \end{prop}

The proofs of Propositions~\ref{prop: F_{m, n}} and ~\ref{prop: I_{m, n}} will be given in Sections~\ref{sec:Fmn} and ~\ref{sec:Imn}, respectively. Assuming Propositions~\ref{prop: F_{m, n}} and~\ref{prop: I_{m, n}}, we may stack the diagrams in Equations~\eqref{equation:I} and~\eqref{equation:II} and use the shift from Equation~\ref{Eq: shift limit} to obtain that the map $X_{n,m}$ is grading preserving with respect to the stated shifts in Theorem~\ref{thm:gradings}. Our proofs will focus on the case where $m$ and $n$ are odd. The other parities can be dealt with similarly.

\subsection{Gradings and $F_{n, m}$}
\label{sec:Fmn}
In this subsection, we prove Proposition~\ref{prop: F_{m, n}}. In Lemma~\ref{lem: topology of cobordism} we describe the cobordisms used to define the map $F_{n, m}$ in Equation~\eqref{equation:I}. Proposition~\ref{prop: F_{m, n}} will follow quickly from this description. 

\begin{lem}\label{lem: topology of cobordism} Let $F$, $j$ and $h$ be the components of $F_{n,m}$, as labeled in Equation~\eqref{equation:I}.
\begin{enumerate}
    \item The map $F$ counts index $0$ instantons with respect to a fixed metric on the cobordism obtained by attaching two $-1$ framed 2-handles along the meridians $L_1^*$ and $L_2^*$ of $L_1$ and $L_2$.
    \item The map $j$ is a sum of two maps. Both summands are compositions of two maps. Each of these maps counts either index 0 or $-1$ instantons on certain 2-handle cobordisms. In both summands, the composition of the corresponding cobordisms is the natural cobordism from $Y_{0,0}$ to $Y_{1,\infty}$,  obtained by attaching a $-1$ framed 2-handle along a meridian $L_1^*$ of $L_1$, a $-1$ framed 2-handle along a meridian $L_2^*$ of $L_2$, and a $-1$ framed meridian $L_2^{*}$ along $L_2^*$. 
\item The map $h$ is a composition of two maps, each of which counts index 0 or index $-1$ instantons on certain 2-handle cobordisms. The composition of the corresponding cobordisms is the natural cobordism from $Y_{0,0}$ to $Y_{\infty,1}$, obtained by attaching a $-1$ framed 2-handle along a meridian $L_2^*$ of $L_2$, a $-1$ framed 2-handle along a meridian $L_1^*$  of $L_1$ and a $-1$ framed 2-handle along a meridian $L_1^{*}$ of $L_1^*$. 
\end{enumerate}
  
\end{lem}

\begin{proof}
The map $F_{n, m}$ in Equation~\eqref{equation:I} is the composition of two quasi-isomorphisms, 

\begin{equation}\label{eq:f}
f \colon \SHI (Y_{0, 0} (L)) \to \Cone \left(\begin{tikzcd}
{\SHI (Y_{1, 0}(L))} \arrow[r] & {\SHI (Y_{\infty, 0}(L))}
\end{tikzcd}\right)
\end{equation}
and 
\begin{equation}\label{eq:g}
g \colon \Cone 
\left(\begin{tikzcd}[column sep=.5cm]
{\SHI (Y_{1, 0} (L))} \arrow[r] & {\SHI (Y_{\infty, 0}(L))}
\end{tikzcd}\right)
\to
\Tot \left(
\begin{tikzcd}[column sep=.5cm, row sep=.5cm]
{\SHI (Y_{1,1} (L))} \arrow[r] \arrow[d] & {\SHI (Y_{1, \infty}(L))} \\
{\SHI (Y_{\infty, 1}(L))}                &                          
\end{tikzcd}\right)
\end{equation}
Here, $\Tot$ means the total space of the diagram, i.e. we view the entire diagram as a single chain complex. Thus, we may view $F_{n,m}$ as a composition, as in the following diagram:
\begin{equation}\label{eq: Proof of F_{n,m}}
F_{n, m} = 
\begin{tikzcd}
{\SHI(Y_{0,0}(L))} 
    \arrow[rd, "f" description]
    \arrow[rrddd, "j" description, dotted, bend right=40]
    \arrow[rrrdd, "h" description, dotted,bend left=40]
    \arrow[rrd, "f" description, dashed] &&&\\
& {\SHI(Y_{0,1}(L))}
    \arrow[rrd, "g" description, dashed]
    \arrow[rdd, "g" description, dashed]
    \arrow[rd, "g" description]
    \arrow[r] 
& {\SHI(Y_{0, \infty}(L))} 
    \arrow[rd, "g" description]&
    \\
    && 
{\SHI(Y_{1,1}(L))}
    \arrow[r]
    \arrow[d]
&{\SHI(Y_{1,\infty}(L))} 
\\
&                                                                          & {\SHI(Y_{\infty,1}(L))}&                        
\end{tikzcd}
\end{equation}
 In Equation~\eqref{eq: Proof of F_{n,m}} the solid arrows denote maps that count monopoles of index 0 (resp. instantons) in the corresponding cobordism over a fixed metric and the dashed arrows denote maps that count monopoles (resp. instantons) of  index $-1$ in the corresponding cobordism over a one-parameter family of metrics.
\end{proof}

\begin{rem}
In the Heegaard Floer setting, $F$ is a composition of holomorphic triangle maps. The maps, $h$ and $j$ are sums of compositions of a holomorphic quadrilateral map and a holomorphic triangle map.
\end{rem}

We are now able to prove Proposition~\ref{prop: F_{m, n}}:

\begin{proof}[Proof of Proposition~\ref{prop: F_{m, n}}]
Consider the map $f$ from Equation~\eqref{eq:f}. The map $f$ is sum of two maps, $f_{(0, 0), (0, 1)}$ and $f_{(0, 0), (0, \infty)}$ as in the following diagram
\begin{equation}\label{equation: f}
f =     
\begin{tikzcd}[column sep=1cm, row sep=1cm]
{\SHI (Y_{0, 0} (L))} \arrow[d, "f_{(0, 0), (0, 1)}" description] \arrow[rd, "f_{(0, 0), (0, \infty)}" description, dashed] &                            \\
{\SHI (Y_{0, 1} (L))} \arrow[r, ]                                      & {\SHI (Y_{0, \infty} (L))}
\end{tikzcd}
\end{equation}
The map $f_{(0, 0), (0, 1)}$ is the cobordism map obtained by attaching a $-1$ framed 2-handle along the meridian $L_2^*$ of $L_2$. The attaching circle of the $2$-handle can be easily made disjoint from the grading surface  $S^{\tau(n+m)}_{n+m}$. Hence the surface survives in the underlying cobordism and as a result the map $f_{(0, 0), (0, 1)}$ is grading preserving. See the proof of \cite{GhoshLiWongTau}*{Proposition~1.12} for more details.

We  now consider the map
\[
f_{(0, 0), (0, \infty)}\colon \SHI(Y_{0,0}(L)) \rightarrow \SHI(Y_{0, \infty}(L))
.\]
  Although the map $f_{(0, 0), (0, \infty)}$ is obtained by counting solutions on several cobordisms, we observe that each cobordism topologically is obtained by a attaching $4$-dimensional $2$-handles along knots which are disjoint from the grading surface $S^{\tau(n+m)}_{n+m}$. Hence, our argument for the ordinary cobordism map $f_{(0, 0), (0, 1)}$ also applies for the map $f_{(0, 0), (0, \infty)}$. We can similarly argue that the map $g$ from Equation~\eqref{eq:g} is also grading preserving. Hence the map $F_{m, n}$ is grading preserving as it is obtained by composing grading preserving maps.
\end{proof}
\begin{figure}[ht]
\centering
\begingroup%
  \makeatletter%
  \providecommand\color[2][]{%
    \errmessage{(Inkscape) Color is used for the text in Inkscape, but the package 'color.sty' is not loaded}%
    \renewcommand\color[2][]{}%
  }%
  \providecommand\transparent[1]{%
    \errmessage{(Inkscape) Transparency is used (non-zero) for the text in Inkscape, but the package 'transparent.sty' is not loaded}%
    \renewcommand\transparent[1]{}%
  }%
  \providecommand\rotatebox[2]{#2}%
  \newcommand*\fsize{\dimexpr\f@size pt\relax}%
  \newcommand*\lineheight[1]{\fontsize{\fsize}{#1\fsize}\selectfont}%
  \ifx\svgwidth\undefined%
    \setlength{\unitlength}{396.69171578bp}%
    \ifx\svgscale\undefined%
      \relax%
    \else%
      \setlength{\unitlength}{\unitlength * \real{\svgscale}}%
    \fi%
  \else%
    \setlength{\unitlength}{\svgwidth}%
  \fi%
  \global\let\svgwidth\undefined%
  \global\let\svgscale\undefined%
  \makeatother%
  \begin{picture}(1,0.5560823)%
    \lineheight{1}%
    \setlength\tabcolsep{0pt}%
    \put(0,0){\includegraphics[width=\unitlength,page=1]{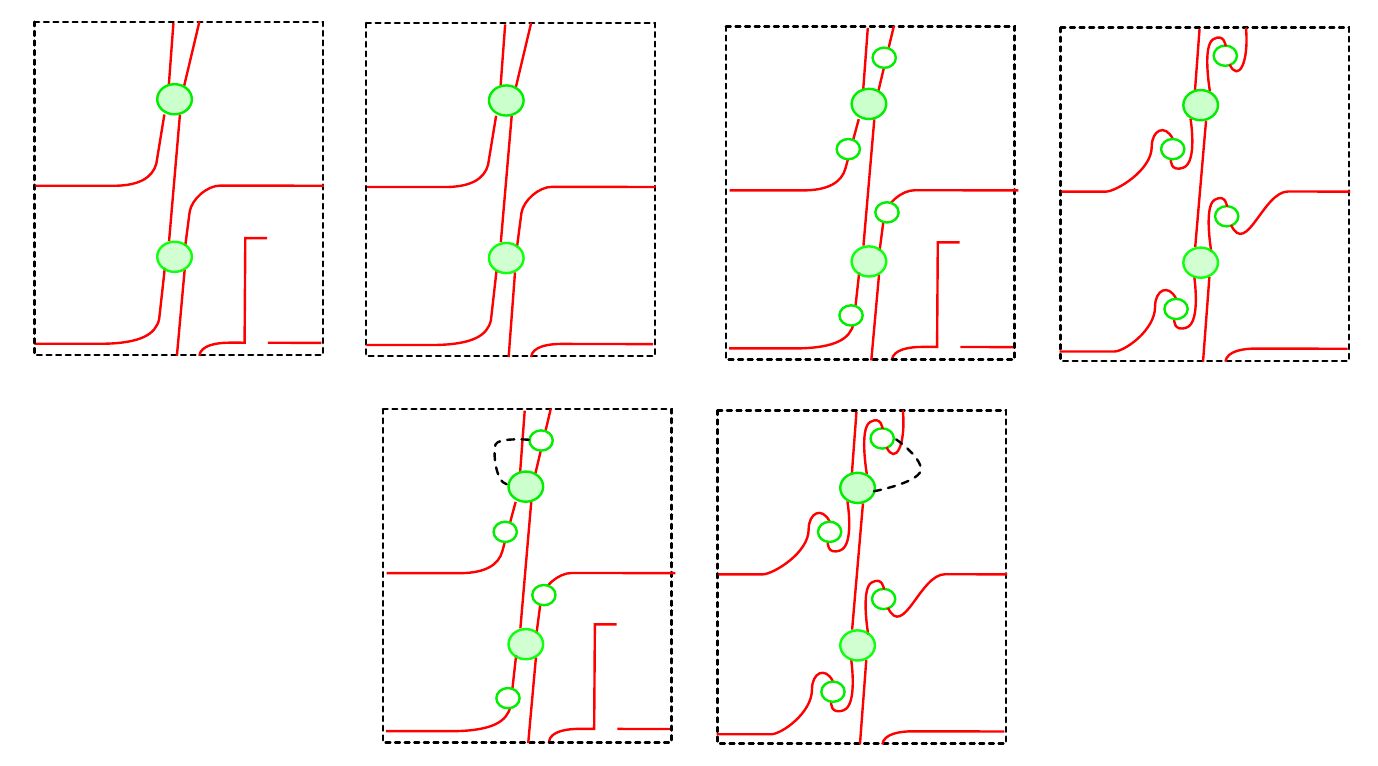}}%
    \put(0.67460593,0.21586003){\makebox(0,0)[lt]{\lineheight{1.25}\smash{\begin{tabular}[t]{l}$c_2^i$\end{tabular}}}}%
    \put(0,0){\includegraphics[width=\unitlength,page=2]{fig47.pdf}}%
    \put(0.28577265,0.05717391){\color[rgb]{0,0,0}\makebox(0,0)[lt]{\lineheight{1.25}\smash{\begin{tabular}[t]{l}$S^{\tau}_{n+m}$\end{tabular}}}}%
    \put(0.03544092,0.34010671){\color[rgb]{0,0,0}\makebox(0,0)[lt]{\lineheight{1.25}\smash{\begin{tabular}[t]{l}$S^{\tau}_{n+m}$\end{tabular}}}}%
    \put(0.53744107,0.3366459){\color[rgb]{0,0,0}\makebox(0,0)[lt]{\lineheight{1.25}\smash{\begin{tabular}[t]{l}$S^{\tau}_{n+m}$\end{tabular}}}}%
    \put(0,0){\includegraphics[width=\unitlength,page=3]{fig47.pdf}}%
  \end{picture}%
\endgroup%

\caption{Grading surfaces in the contact handle manipulation from Figure~\ref{fig:42} used to construct the isomorphism between $\SHI(Y_{1, 1} (L), S^\tau_{n+m})$ and $\SHI(Y'_{1,1}(L), S^\tau_{n+m})$. In the diagram, both $n$ and $m$ are odd, hence $\tau=1$.}
\label{fig:47}
\end{figure} 
\begin{figure}[ht]
\centering
\begingroup%
  \makeatletter%
  \providecommand\color[2][]{%
    \errmessage{(Inkscape) Color is used for the text in Inkscape, but the package 'color.sty' is not loaded}%
    \renewcommand\color[2][]{}%
  }%
  \providecommand\transparent[1]{%
    \errmessage{(Inkscape) Transparency is used (non-zero) for the text in Inkscape, but the package 'transparent.sty' is not loaded}%
    \renewcommand\transparent[1]{}%
  }%
  \providecommand\rotatebox[2]{#2}%
  \newcommand*\fsize{\dimexpr\f@size pt\relax}%
  \newcommand*\lineheight[1]{\fontsize{\fsize}{#1\fsize}\selectfont}%
  \ifx\svgwidth\undefined%
    \setlength{\unitlength}{271.3346447bp}%
    \ifx\svgscale\undefined%
      \relax%
    \else%
      \setlength{\unitlength}{\unitlength * \real{\svgscale}}%
    \fi%
  \else%
    \setlength{\unitlength}{\svgwidth}%
  \fi%
  \global\let\svgwidth\undefined%
  \global\let\svgscale\undefined%
  \makeatother%
  \begin{picture}(1,0.69849092)%
    \lineheight{1}%
    \setlength\tabcolsep{0pt}%
    \put(0.4254581,0.26834909){\color[rgb]{0,0,0}\makebox(0,0)[lt]{\lineheight{1.25}\smash{\begin{tabular}[t]{l}$\d S^+_{n+m}$\end{tabular}}}}%
    \put(0,0){\includegraphics[width=\unitlength,page=1]{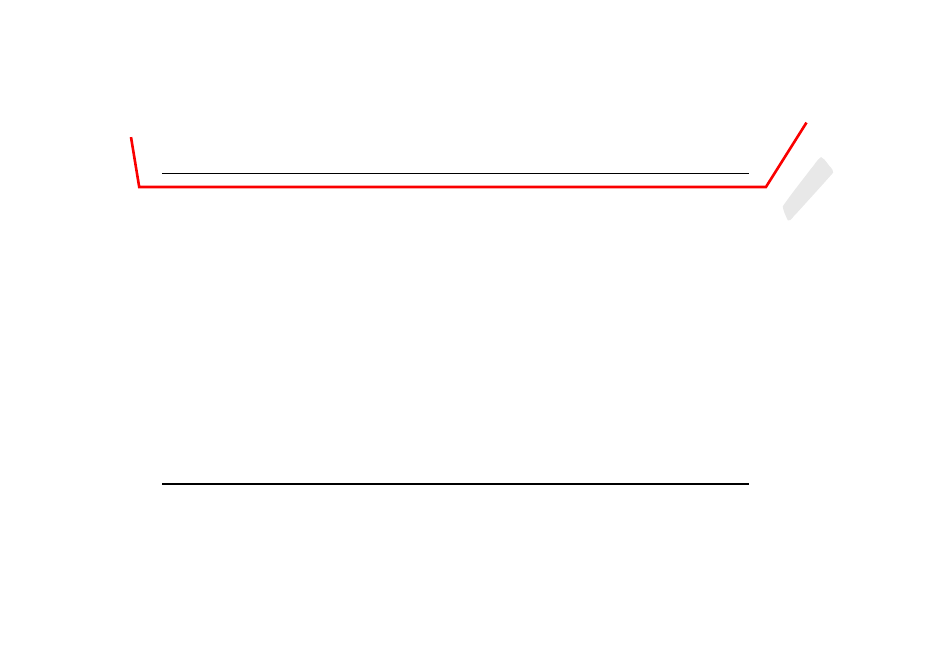}}%
    \put(0.71402569,0.38088432){\color[rgb]{0,0,0}\makebox(0,0)[lt]{\lineheight{1.25}\smash{\begin{tabular}[t]{l}$R_-(\gamma)$\end{tabular}}}}%
    \put(0.71832011,0.29910789){\color[rgb]{0,0,0}\makebox(0,0)[lt]{\lineheight{1.25}\smash{\begin{tabular}[t]{l}$R_+(\gamma)$\end{tabular}}}}%
    \put(0.59312429,0.38191017){\color[rgb]{0,0,0}\makebox(0,0)[lt]{\lineheight{1.25}\smash{\begin{tabular}[t]{l}$\d T$\end{tabular}}}}%
    \put(0,0){\includegraphics[width=\unitlength,page=2]{fig64.pdf}}%
    \put(0.14987937,0.35110149){\color[rgb]{1,0,0}\makebox(0,0)[lt]{\lineheight{1.25}\smash{\begin{tabular}[t]{l}$\Gamma_{n+m}$\end{tabular}}}}%
  \end{picture}%
\endgroup%

\caption{The boundaries of the grading surfaces $S^+_{n+m}$ (solid blue) and $T$ (dashed blue) in the tube region. The shaded region is the subsurface $F$ which has boundary $\d T-\d S_{n+m}^+$.}
\label{fig:50}
\end{figure}

\begin{figure}[ht]
\centering
\begingroup%
  \makeatletter%
  \providecommand\color[2][]{%
    \errmessage{(Inkscape) Color is used for the text in Inkscape, but the package 'color.sty' is not loaded}%
    \renewcommand\color[2][]{}%
  }%
  \providecommand\transparent[1]{%
    \errmessage{(Inkscape) Transparency is used (non-zero) for the text in Inkscape, but the package 'transparent.sty' is not loaded}%
    \renewcommand\transparent[1]{}%
  }%
  \providecommand\rotatebox[2]{#2}%
  \newcommand*\fsize{\dimexpr\f@size pt\relax}%
  \newcommand*\lineheight[1]{\fontsize{\fsize}{#1\fsize}\selectfont}%
  \ifx\svgwidth\undefined%
    \setlength{\unitlength}{412.70075708bp}%
    \ifx\svgscale\undefined%
      \relax%
    \else%
      \setlength{\unitlength}{\unitlength * \real{\svgscale}}%
    \fi%
  \else%
    \setlength{\unitlength}{\svgwidth}%
  \fi%
  \global\let\svgwidth\undefined%
  \global\let\svgscale\undefined%
  \makeatother%
  \begin{picture}(1,0.57039583)%
    \lineheight{1}%
    \setlength\tabcolsep{0pt}%
    \put(0,0){\includegraphics[width=\unitlength,page=1]{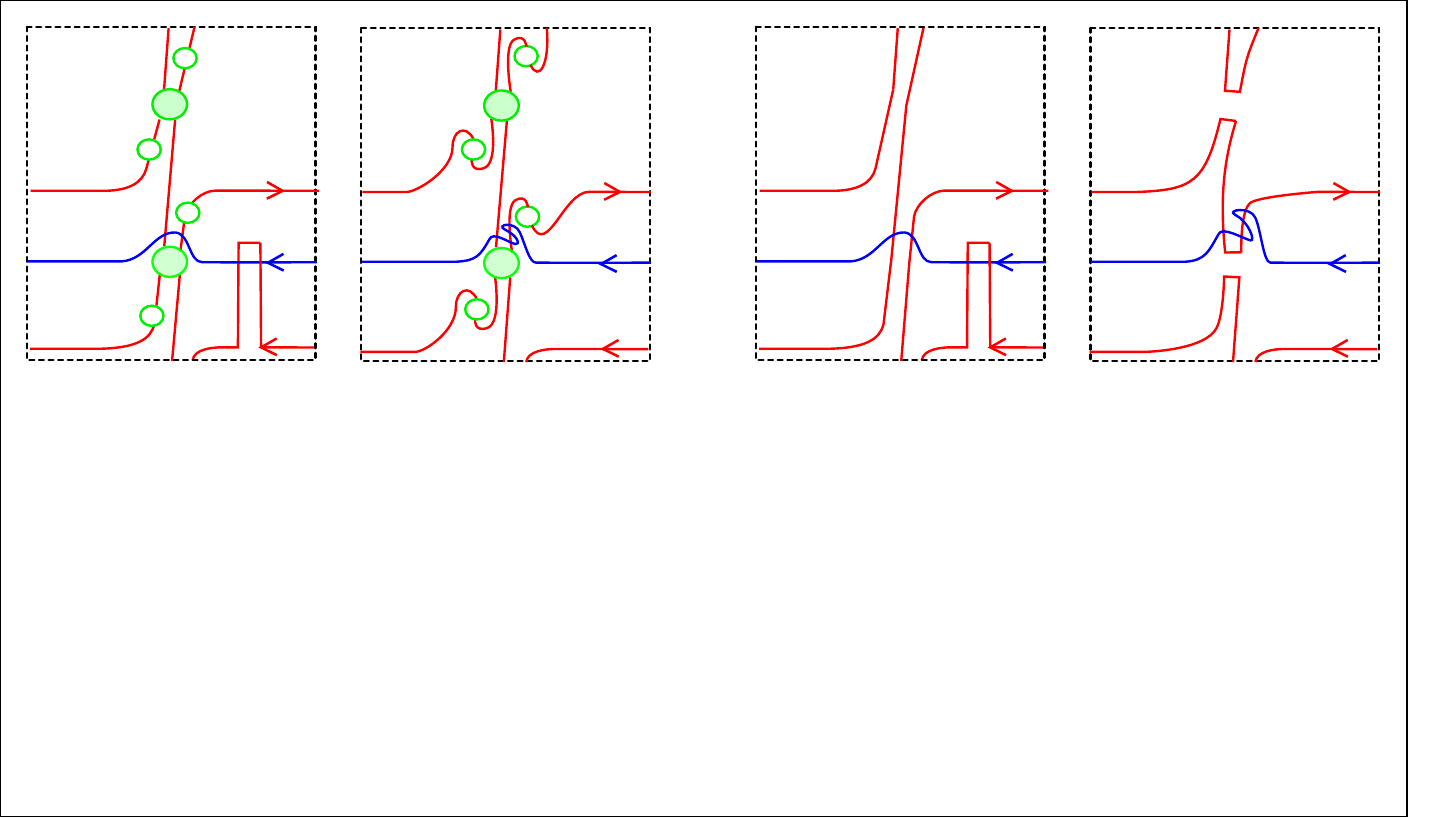}}%
    \put(0.58474385,0.29134804){\color[rgb]{0,0,0}\makebox(0,0)[lt]{\lineheight{1.25}\smash{\begin{tabular}[t]{l}$C_{n-1}$\end{tabular}}}}%
    \put(0.26230746,0.01514691){\color[rgb]{0,0,0}\makebox(0,0)[lt]{\lineheight{1.25}\smash{\begin{tabular}[t]{l}$C_{n-1}$\end{tabular}}}}%
    \put(0.82465436,0.29113366){\color[rgb]{0,0,0}\makebox(0,0)[lt]{\lineheight{1.25}\smash{\begin{tabular}[t]{l}$D_m$\end{tabular}}}}%
    \put(0.72269483,0.01493252){\color[rgb]{0,0,0}\makebox(0,0)[lt]{\lineheight{1.25}\smash{\begin{tabular}[t]{l}$D_m$\end{tabular}}}}%
    \put(0,0){\includegraphics[width=\unitlength,page=2]{fig65.pdf}}%
    \put(0.01372091,0.28681165){\makebox(0,0)[lt]{\lineheight{1.25}\smash{\begin{tabular}[t]{l}$A$\end{tabular}}}}%
    \put(0.50464355,0.28681165){\makebox(0,0)[lt]{\lineheight{1.25}\smash{\begin{tabular}[t]{l}$B$\end{tabular}}}}%
    \put(0.01978169,0.01104648){\makebox(0,0)[lt]{\lineheight{1.25}\smash{\begin{tabular}[t]{l}$C$\end{tabular}}}}%
    \put(0,0){\includegraphics[width=\unitlength,page=3]{fig65.pdf}}%
  \end{picture}%
\endgroup%

\caption{Blue denotes boundary of the surface $T$ from Figure~\ref{fig:50} and red denotes the suture. A: Isotopy of the grading surface $T$ from Figure~\ref{fig:50}. B: Decomposition along the cocore disks of the 1-handles in $Y'_{1,1}(L)$ disjoint from $T$. C: The grading surfaces obtained at $C_{n-1}$ and $D_m$. In the diagram of $C_{n-1}$, the left square shows $n$ positive twists and the right square shows 1 negative twist of the suture $\Gamma_{n-1}$. The figure depicts the case, $n=3$.}
\label{fig:45}
\end{figure}

\begin{figure}[ht]
\centering
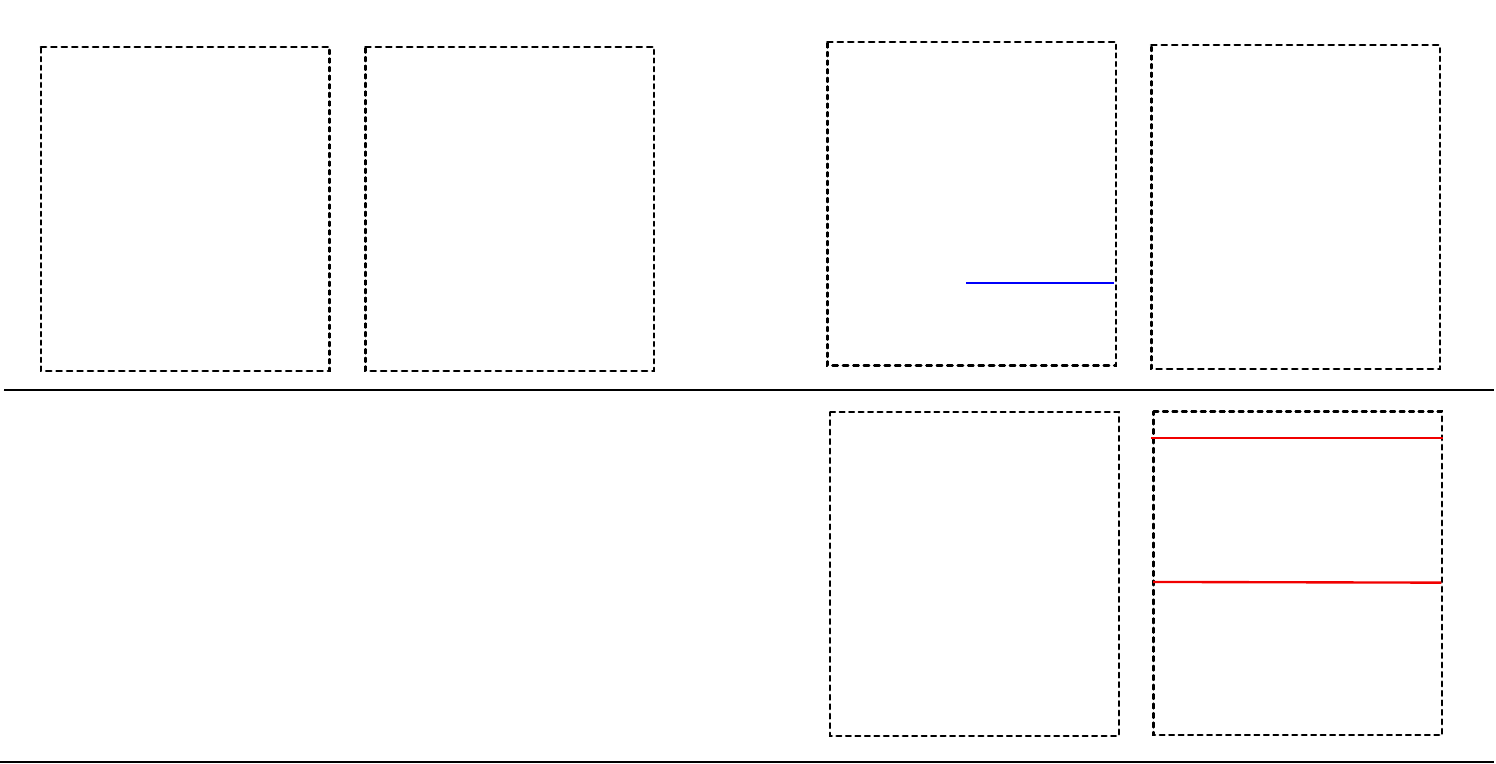
\caption{The grading shift of the map, $i_{\infty, 1}$ may be computed by the manipulation shown above. The grading shift from A to B is $-1$. The grading shift from B to C is $0$. The grading shift from C to D is $+1$.}
\label{fig:39}
\end{figure}

\begin{figure}[ht]
\centering
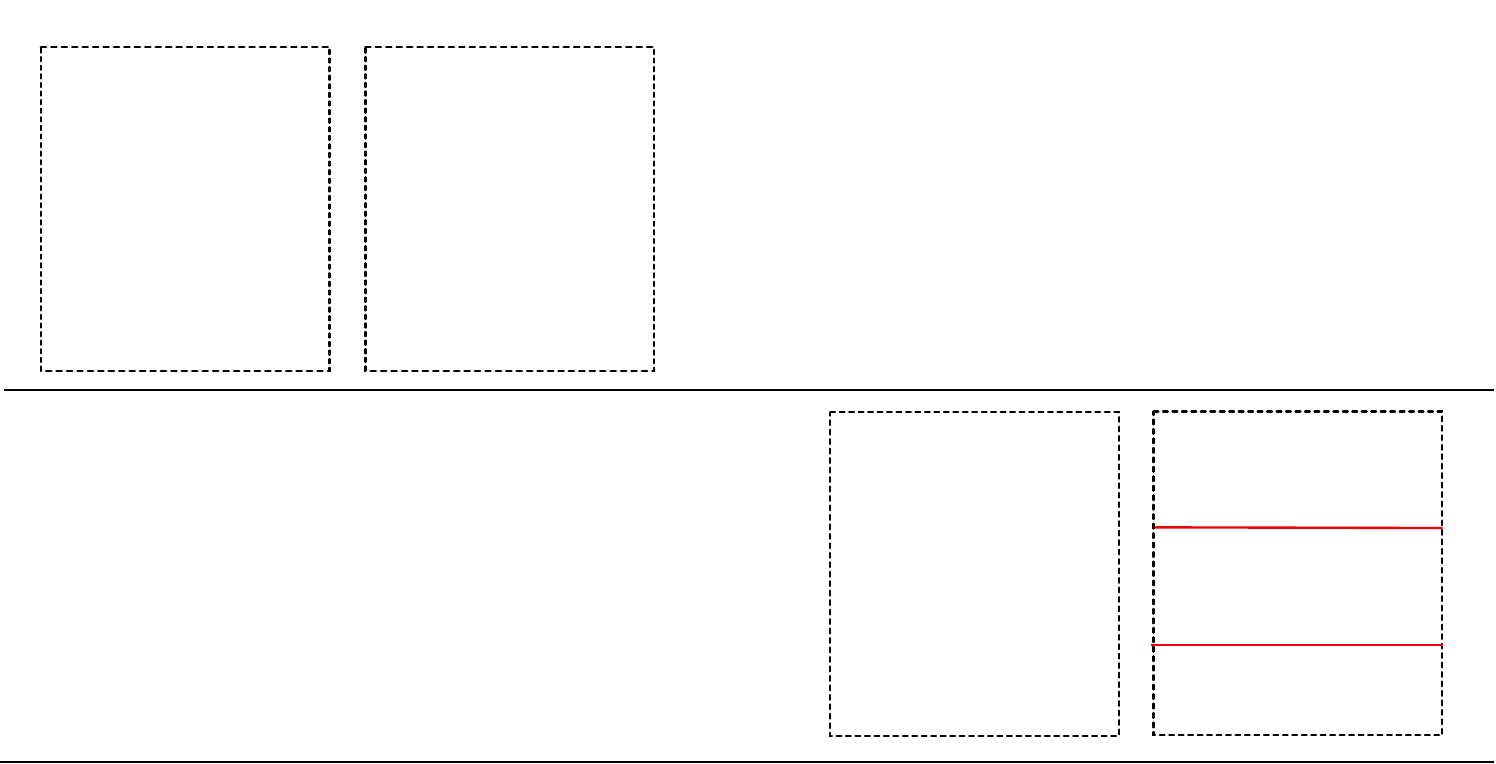
\caption{The grading shift of the map, $i_{1, \infty}$ may be computed by the manipulation shown above. The grading shift from A to B is $0$. The grading shift from B to C is $0$. The grading shift from C to D is $-1$.}
\label{fig:40}
\end{figure}

\subsection{Gradings and $I_{n, m}$}
\label{sec:Imn}
We now compute the grading change of the map $I_{n,m}$ and prove Proposition~\ref{prop: I_{m, n}}.

\begin{proof}[Proof of Proposition~\ref{prop: I_{m, n}}]
     Each component of $I_{n,m}$ in Equation~\eqref{equation:II} is a composition of two maps. The first map in each component corresponds to the contact handle manipulation used to identify $\SHI(Y_\veps(L))$ with $\SHI(Y_\veps'(L))$, where $\veps\in \{(1,1), (\infty,1),(1,\infty)\}$. The second map in each composition is an  isomorphism arising from a disk decomposition.

      We first observe that the isomorphisms from  $\SHI(Y_{\veps}(L),S^+_{n+m})$ to $\SHI(Y'_{\veps}(L), S^+_{n+m})$ are grading preserving because they are obtained by attaching contact 1-handles disjoint from the grading surface $S^+_{n+m}$. See Figure~\ref{fig:47}. 
     
     We now consider the maps arising from disk decompositions in our definition of $I_{n,m}$. We focus first on the map $i_{1,1,}$.   We only discuss the grading of the component mapping to the summand $C_{n-1} \otimes D_m$ of the codomain. The grading of the component mapping to $C_n\otimes D_{m-1}$ is analyzed similarly. Furthermore,  we focus on the case when $n$ and $m$ both are odd and leave the other parities for the reader. Note that if $n$ and $m$ are both odd, then $\tau(n+m)=1$. Hence, the surface $S_{n+m}^{\tau(n+m)}$ is obtained from $S_{n+m}$ by a single positive stabilization. We write $S_{n+m}^+$ for this surface. Furthermore, $\zeta(n,m)=1$.

 We now perform a local modification to the grading surface $S^+_{n+m}$ near $\d Y_{1,1}'(L)$  in the tube region, as shown in Figure~\ref{fig:50}. We call this new surface $T$. We form $T$ by adjoining a surface $F$ to $S^+_{n+m}$. The surface $F$ is a subsurface of $\partial Y'_{1, 1} (L)$ and has boundary $\partial T - \partial S^+_{n+m}$. Since $\chi(F \cap R_-(\gamma)) - \chi(F \cap R_+(\gamma)) = 0$, by Theorem~\ref{thm: Generalized shift} the grading induced by $T$ differs from the grading induced by $S^+_{n+m}$ by $0$.
 
  We now examine the disk decomposition used to define $i_{1,1}$. The original grading surface $S_{n+m}^+$ had a single positive stabilization, which we assume is on the side of the diagram for the complexes $C_n$ and $C_{n-1}$
  
  We cut along cocores of the 1-handles, which are now disjoint from the grading surface $T$. This is shown in frame B of Figure~\ref{fig:45}. The grading surface on diagram for $C_{n-1}$ intersects the sutures homologically $2(n-1)$ times.  If $n>0$, it additionally has two positive and one negative stabilizations. One of the positive stabilizations is from our original stabilization of $S_{n+m}^+$, and the other pair of stabilizations are due to the extra twist in the sutures in the negative direction, after we perform the disk decomposition. See frame C of Figure~\ref{fig:45}.
  
     To obtain $S^+_{n-1}\sqcup S_m$, we perform one positive and one negative stabilization to the surface on the left hand side of frame C of Figure~\ref{fig:45}. Additionally, we perform one positive and one negative stabilization to the $D_m$ side of the diagram. Therefore by Theorem~\ref{thm: shift stabilization}, the total grading shift of $i_{1,1}$ with respect to the grading surface $S_n^+\sqcup S_m$ is 0.

We now consider the map $i_{\infty, 1}$ from Equation~\eqref{equation:II}. We begin with the manifold $Y_{\infty,1}'(L)$ and remove contact 1-handles which are disjoint from the grading surface $S^+_{n+m}$. The result is shown in Frame A of Figure~\ref{fig:39}.  As before, this does not change the grading. Next, we change the grading surface by deleting a positive stabilization and adding a negative stabilization. The resulting surface $S_{n+m}^-$ is shown in Frame B of Figure~\ref{fig:39}. By Theorem~\ref{thm: shift stabilization}, the grading changes by $-1$ from Frame A to Frame B. Next, we perform an isotopy of the boundary of the grading surface to obtain Frame C. The grading shift from frame B to frame C is 0 as they are isotopic. Frame D is obtained by deleting the contact $1$-handle and undoing the two negative stabilizations. Deleting the contact 1-handle has no effect on grading as there is a closure which is simultaneously the closure of the manifold before and after after attaching the 1-handle, where both surfaces, $S^-_{n+m}$ and $S^-_n \sqcup S^-_m$ extend to the same closed surface. Finally by Theorem~\ref{thm: shift stabilization}, undoing the two negative stabilizations shifts the grading by $1$. Hence, the map $i_{\infty, 1}$ is grading preserving. See Figure~\ref{fig:39}.

Consider the map $i_{1, \infty}$ from Equation~\eqref{equation:II}. We now begin with the manifold $Y_{1, \infty}'(L)$ and remove contact 1-handles which are disjoint from the grading surface $S^+_{n+m}$. The result is shown in Frame A of Figure~\ref{fig:40}.  As before, this does not change the grading. Next, we perform an isotopy of the boundary of the grading surface to obtain Frame B. The grading shift from frame A to frame B is 0 as they are isotopic. Frame C is obtained by deleting the contact $1$-handle. Deleting the contact 1-handle has no effect on grading as there is a closure which is simultaneously the closure of the manifold before and after after attaching the 1-handle, where both surfaces, $S^+_{n+m}$ and $S^+_n \sqcup S^+_m$ extend to the same closed surface. Finally by Theorem~\ref{thm: shift stabilization}, undoing the two positive stabilizations shifts the grading by $-1$. Hence, the map, $i_{1, \infty}$ shifts the grading by $-1$. See Figure~\ref{fig:40}.

Finally, to finish the proof we consider the grading shift from the construction of direct limit in both the domain and co-domain of the map $I_{n, m}$. We shift the grading of $\SHI (Y_{1, 1}(L),  \Gamma_{n+m})$, $\SHI (Y_{\infty, 1}(L),  \Gamma_{n+m})$ and $\SHI (Y_{1, \infty}(L),  \Gamma_{n+m})$ by $\sigma (n+m)$. See Equation~\eqref{Eq: shift limit}. In $C_n \otimes D_m$, we shift the grading by $\zeta(n, m) = \sigma(n) + \sigma(m)$. On $C_n \otimes D_{m-1}$, we shift the grading by $\zeta(n, m-1)$ and on $C_{n-1} \otimes D_m$, we shift the grading by $\zeta(n-1, m)$. We observe that when $n$ and $m$ are odd,
\[
\sigma(n+m)=\sigma(n)+\sigma(m)-1,\qquad 
\sigma(m)=\sigma(m-1),\quad \text{and} \quad 
\sigma(n)=\sigma(n-1).
\]
Hence, with respect to the grading shifts from the direct limit, the maps $i_{\infty, 1}$ and $i_{1, 1}$ shift the grading by 1, leading to the shifts in the main statement.

A similar argument to the above argument for $i_{\infty,1}$ shows that $i_{1,\infty}$ drops the grading by $-1$ with respect to the gradings from $S_{n+m}^+$ and $S_n\sqcup S_{m}$. Hence, with respect to the shifts from the direct limit construction of knot Floer homology, the map $i_{1,\infty}$ preserves the grading, leading to the shift in the statement.  This completes the proof.
\end{proof}

\begin{rem}
In this paper, we are taking the direct limit with respect to the positive bypass maps. In \cite{LiLimits}, the author considered the direct limit with respect to the negative bypass maps with the oppositely oriented sutures. When we reverse the orientation of the sutures, the sign of the bypass also switches.
\end{rem}

We are finally ready to prove the main theorem of this section, Theorem~\ref{thm:main-gradings}, which we recall states that the isomorphism 
\[
\HFK^-(K_1\# K_2)\iso \HFK^-(K_1)\tildeotimes_{\bF[U]} \HFK^-(K_2)
\]
holds at the level of graded groups.

\begin{proof}[Proof of Theorem~\ref{thm:main-gradings}] Let $q\in \Z$ be a fixed integer. If $G$ is a graded group, write $G_{>q}$ for the subgroup concentrated in gradings above $q$. Note that the knot Floer groups $\HFK^-$ have Alexander grading which is bounded from above, and $U$ acts by $-1$.  We will show for each $q\in \Z$ that $\HFK^-(K_1\#K_2)_{>q}$ is isomorphic to $(\HFK^-(K_1)\tildeotimes \HFK^-(K_2) )_{>q}$. First, we assume that $N$ is chosen so that if $n>N$, then the positive  bypass map $\phi_{n}^+\colon \SHI(Y_1\setminus \nu(K_1),\Gamma_n)\to \SHI(Y_1\setminus \nu(K_1),\Gamma_{n+1})$ is an isomorphism on gradings above $q$. Similarly, pick $M$ so that if $m>M$, the analog holds for $K_2$. We may assume, additionally, that $N$ and $M$ are large enough that the positive bypass maps for $K_1\#K_2$ are also isomorphisms on gradings above $q$ when the framing is at least $N+M$.

We have proven that there is a graded isomorphism between $\SHI(Y_1\# Y_2\setminus \nu(K_1\# K_2))$ and the diagram
\begin{equation}
\begin{tikzcd}[column sep=4cm, row sep=1.5cm, labels=description] C_{n-1}\otimes D_m\oplus C_n\otimes D_{m-1} \ar[r,"\phi_{n-1}^-| \id\oplus \id| \phi_{m-1}^-"] \ar[d, "\phi_{n-1}^+| \id+\id| \phi_{m-1}^+"]& C_n\otimes D_m\\
C_n\otimes D_m&\,
\end{tikzcd}
\label{eq:expanded-complex}
\end{equation}
where $C_n= \SHI(Y_1\setminus \nu(K_1),\Gamma_n)$ and $D_m= \SHI(Y_2\setminus \nu(K_2)).$

By Lemma~\ref{lem:staircase-lemma}, there is a grading preserving chain map
\[
\Phi\colon C_{n-1}\otimes D_{m-1}\to \Cone\left(\begin{tikzcd}[column sep=3cm]
C_{n-1}\otimes D_m\oplus C_n\otimes D_{m-1} \ar[r, "\phi_{n-1}^+| \id+\id| \phi_{m-1}^+"]& C_n\otimes D_m
\end{tikzcd}\right),
\]
as well as a grading preserving chain map $\Psi$ in the opposite direction (technically, to $C_n\otimes D_m$) which are isomorphisms on homology in degrees above $q$. Since the differential preserves the Alexander grading, we conclude that  the subcomplex of Equation~\eqref{eq:expanded-complex} in grading above $q$ is homotopy equivalent to the complex
\[
\Cone(
\begin{tikzcd}[column sep=3cm]
(\HFK^-(K_1)\otimes_{\bF} \HFK^-(K_2))_{>q+1} \ar[r, "U_1\otimes \id+\id\otimes U_2"]& (\HFK^-(K_1)\otimes_{\bF} \HFK^-(K_2))_{>q})
\end{tikzcd}
\]
which coincides with $(\HFK^-(K_1)\tildeotimes \HFK^-(K_2))_{>q}$.  This completes the proof. 
\end{proof}

\begin{rem} Note that in the above, we used $\Cone(U_1+U_2)$ as a model of the derived tensor product, whereas usually one considers $\Cone(U_1-U_2)$. Note that the bypass maps $\phi_{n}^{\pm}$ are themselves only well-defined up to multiplication by an element of $\C^\times.$ We may view replacing $U_2$ with $-U_2$ as the result of changing the module structure on $\KHI^-(K_2)$ by declaring $U_2$ to instead act by $\a U_2$, where $\a\in \C^\times$. Lemma~\ref{lem:rescale-module} implies that the resulting modules are isomorphic. In particular, $\Cone(U_1+U_2)$ is isomorphic to $\Cone(U_1-U_2)$.
\end{rem}

\section{Oriented skein exact triangle}\label{sec:10}

In this section, we prove the Skein exact triangle for the minus version of instanton link homology. This relation has long been established in the setting of Heegaard knot and link Floer homology \cite{OSKnots}*{Theorem~8.2}. In the setting of instanton Floer homology, Kronheimer and Mrowka \cite{KMskein}*{Theorem~3.1} showed the existence of similar exact triangle in the version of instanton Floer homology which is analogous to $\widehat{\HFK}$ (denoted in our paper by $\widehat{\KHI}$).

In this section, we consider the minus version of instanton knot and link Floer homology, defined in \cite{LiLimits} and \cite{GhoshLiDecomposing}*{Section~6.2}. For an $r$ component link $L$ in a homology sphere $Y$, this takes the form of a module $\mathit{KHL}^-(Y, L)$ over $\C[U_1,\dots, U_r]$. This is the direct limit of the sutured Floer homologies of $Y\setminus \nu(L)$, equipped with longitudinal slopes. The action of $U_i$ is by a negative bypass on the boundary corresponding to the $i$-th link component.

The skein exact triangle focuses on a triple of links $L_+$, $L_-$ and $L_0$ inside of $Y$, which coincide except in a 3-ball, where they take the form shown in Figure~\ref{fig:54}.

\begin{figure}[ht]
\centering
\begingroup%
  \makeatletter%
  \providecommand\color[2][]{%
    \errmessage{(Inkscape) Color is used for the text in Inkscape, but the package 'color.sty' is not loaded}%
    \renewcommand\color[2][]{}%
  }%
  \providecommand\transparent[1]{%
    \errmessage{(Inkscape) Transparency is used (non-zero) for the text in Inkscape, but the package 'transparent.sty' is not loaded}%
    \renewcommand\transparent[1]{}%
  }%
  \providecommand\rotatebox[2]{#2}%
  \newcommand*\fsize{\dimexpr\f@size pt\relax}%
  \newcommand*\lineheight[1]{\fontsize{\fsize}{#1\fsize}\selectfont}%
  \ifx\svgwidth\undefined%
    \setlength{\unitlength}{259.82498169bp}%
    \ifx\svgscale\undefined%
      \relax%
    \else%
      \setlength{\unitlength}{\unitlength * \real{\svgscale}}%
    \fi%
  \else%
    \setlength{\unitlength}{\svgwidth}%
  \fi%
  \global\let\svgwidth\undefined%
  \global\let\svgscale\undefined%
  \makeatother%
  \begin{picture}(1,0.24803236)%
    \lineheight{1}%
    \setlength\tabcolsep{0pt}%
    \put(0,0){\includegraphics[width=\unitlength,page=1]{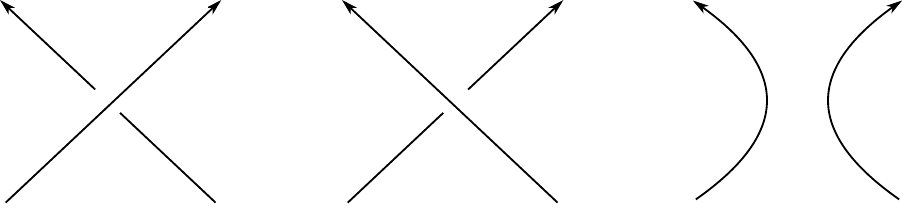}}%
    \put(0.08259549,0.00257908){\color[rgb]{0,0,0}\makebox(0,0)[lt]{\lineheight{1.25}\smash{\begin{tabular}[t]{l}$L_+$\end{tabular}}}}%
    \put(0.48374704,0.00506161){\color[rgb]{0,0,0}\makebox(0,0)[lt]{\lineheight{1.25}\smash{\begin{tabular}[t]{l}$L_-$\end{tabular}}}}%
    \put(0.8545862,0.00486987){\color[rgb]{0,0,0}\makebox(0,0)[lt]{\lineheight{1.25}\smash{\begin{tabular}[t]{l}$L_0$\end{tabular}}}}%
  \end{picture}%
\endgroup%

\caption{The skein triple. The diagram of a positive crossing, negative crossing and their oriented resolution.}
\label{fig:54}
\end{figure}
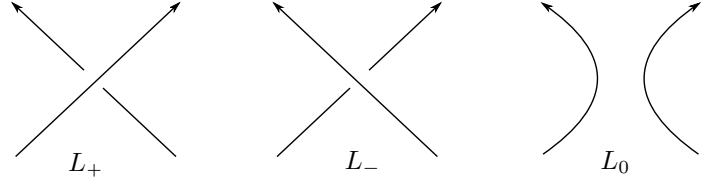

\begin{thm}\label{thm: skein}
Let $L_{+}$, $L_-$ and $L_0$ be oriented links in $Y$ as above. If the two strands meeting at the distinguished
crossing in $L_+$ belong to the same component, so that in the oriented resolution the two
strands corresponding to two distinct components of $L_0$, then there is an exact triangle
\[ 
\cdots \KHI^-(L_{+}) \xrightarrow{f^-} \KHI^-(L_-) \xrightarrow{g^-} \KHI^-(L_0)\tildeotimes \mathbb{C}[U_1,U_2]/(U_1-U_2) \xrightarrow{h^-} \KHI^-(L_{+})\cdots
.\]
If they belong to different components, then there is an exact triangle
\[
\cdots \KHI^-(L_{+}) \xrightarrow{f^-} \KHI^-(L_-) \xrightarrow{g^-} \KHI^-(L_0)\otimes W \xrightarrow{h^-} \KHI^-(L_{+})\cdots
\]
where $W=\oplus_{s\in \Z} W_s$ is the Alexander graded module given by 
\[
    W_s= 
\begin{cases}
    \C,& \text{if s=0, 1 } \\
    0,              & \text{otherwise.}
\end{cases}
\]
Furthermore, each of $f^-$, $g^-$ and $h^-$ preserves the Alexander grading. Here, we equip $\bC[U_1,U_2]$ with the Alexander grading determined by setting $A(1)=0$ and $A(U_1)=A(U_2)=-1$. 
\end{thm}

\subsection{Alexander gradings on instanton link homology} \label{Sub: grading for link}

Let $L$ be an oriented null-homologous link of $r$ components in a closed, oriented 3-manifold $Y$. Let $L_1, L_2, \dots ,L_r$ be the components of $L$ and let $T_i$ be the boundary component of $Y \setminus \nu(L)$ corresponding to the component $L_i$ where $i= 1, \dots, r $. The Seifert surface $S$ of the link $L$ induces a framing
on each $T_i$. We write $\lambda_i$ for the longitude induced by $S$, and we write $\mu_i$ for the meridian. Note that $\lambda_i$
is oriented in the same way as $L_i$, and $\mu_i$ is oriented so that $\#(\lambda_i \cap \mu_i) = 1$. If $\ve{n}=(n_1,\dots,n_r)\in \Z^r$, let $\Gamma_{\ve{n}}$ be the suture on $Y \setminus{\nu(L})$ so that $\Gamma_{\ve{n}} \cap T_i$ consists of two parallel simple closed curves of class $\lambda_i - n_i\cdot  \mu_i$. 

Analogous to the knot case, we write $S_{\ve{n}}$ for the Seifert surface that intersects the suture $\Gamma_{\ve{n}}$ in $2(\sum_{i=1}^{r} n_i)$ points. Following Li \cite{LiLimits}*{Section~5}, if $(\sum_{i=1}^{r} n_i) + r -1$ is odd, the surface $S_{\ve{n}}$ is admissible in the sense of Definition~\ref{admissibility of surfaces}. Let $\ve{e}_i$ be the standard unit vector in $\Z^r$. Then, we have the following proposition.

\begin{prop}\label{grading shift links}
Fix any $i\in \{1,\dots,r\}$ and $\ve{n} = (n_1, \dots, n_r) \in \mathbb{Z}^r$. Then if $(\sum_{i=1}^{r} n_i) + r -1$ is even, for any $j\in \mathbb{Z}$, we have
\[\phi^\pm_{\ve{n}, {\ve{n}+ {\ve{e}_i}}}(\SHI (Y \setminus \nu(L), \Gamma_{\ve{n}}, S_{\ve{n}}^{\pm}, j))\subset \SHI (Y\setminus \nu(L), \Gamma_{\ve{n}+ {\ve{e}_i}}, S_{\ve{n}+\ve{e}_i},j)\]
and if $(\sum_{i=1}^{r} n_i) + r -1$ is odd, then for any $j\in \mathbb{Z}$, we have
\[\phi^\pm_{\ve{n}, {\ve{n}+ {\ve{e}_i}}}(\SHI (Y\setminus{\nu(L)}, \Gamma_{\ve{n}}, S_{\ve{n}}^{\pm 2}, j))\subset \SHI (Y \setminus \nu(L), \Gamma_{\ve{n}+ {\ve{e}_i}}, S_{\ve{n}+ \ve{e}_i}^{\pm},j)\]
\end{prop}

\begin{proof}
$S_{\ve{n}}$ is admissible if and only if $(\sum_{i=1}^{r} n_i) + r -1$ is odd. The proof is now exactly the same as \cite{LiLimits}*{Proposition~5.5}. 
\end{proof}
We now extend our definition of $\tau \colon \Z \to \Z$ from Equation~\eqref{Eq: shift limit} to $\Z^r \to \Z$  and define, \begin{equation}\label{tau for link}
  \tau(\ve{n})= 
\begin{cases}
    0,              & \text{if } (\sum_{i=1}^{r} n_i) + r -1 \text{ is odd}\\
        1,& \text{if } (\sum_{i=1}^{r} n_i) + r -1 \text{ is even}.
\end{cases}   
 \end{equation} Following Equation~\eqref{tau for link}, let $S^{\tau}_{\ve{n}}$ be  $S_{\ve{n}}$ if $\tau(\ve{n})=0$ and $S^{\tau}_{\textbf{n}}$ be a positive stabilization of $S_{\ve{n}}$ performed on $T_i$ if $\tau(\ve{n})=1$. We can use $S^{\tau}_{\ve{n}}$ to define an Alexander grading on $\SHI(Y\setminus \nu(L),\Gamma_{\ve{n}})$.
To construct an Alexander grading on $\KHI^-(Y, L)$, we define the following grading shift:
\[
     \SHI (Y\setminus \nu(L), \Gamma_{\ve{n}}, S_{\ve{n}}^{\tau (\ve{n})}, i) [\sigma (\ve{n})]
   \]
where, \begin{equation}\label{Eq: shift limit for link}
    \sigma (\ve{n}) =
    \begin{cases}
    -\frac{(\sum_{i=1}^{r} n_i) + r -2}{2} + r-1, & \text{if } \tau(\ve{n})=0\\
    -\frac{(\sum_{i=1}^{r} n_i) + r -1}{2} + r-1, & \text{if } \tau(\ve{n})=1
    \end{cases}
 \end{equation}
  \begin{lem} \label{lem: positive negative shift}
   Under this shift, the positive bypass maps, $\phi^+_{\ve{n}, \ve{n} + \ve{e}_i}$ are grading preserving, and the negative bypass maps, $\phi^-_{\ve{n}, \ve{n} + \ve{e}_i}$ shift the grading by $-1$. 
  \end{lem}
  \begin{proof}
  Let $\ve{n} \in \Z^r$, then there are two possibilities. Either $(\sum_{i=1}^{r} n_i) + r -1$ is even or odd. Suppose $(\sum_{i=1}^{r} n_i) + r -1$ is even then $\phi^+_{\ve{n}, \ve{n} + \ve{e}_i}\colon \SHI (Y\setminus{\nu(L)}, \Gamma_{\ve{n}}, S_{\ve{n}}^{\tau(\ve{n})}, j) \to \SHI (Y\setminus \nu(L), \Gamma_{\ve{n}+ {\ve{e}_i}} S_{\ve{n}+\ve{e}_i}^{\tau({\ve{n}+\ve{e}_i})},j)$ is grading preserving from Proposition~\ref{grading shift links}. After we consider the shift from Equation~\eqref{Eq: shift limit for link}, each of the domain and codomain of $\phi^+_{\ve{n}, i}$ gets equal shift. Now suppose $(\sum_{i=1}^{r} n_i) + r -1$ is odd then $\phi^+_{\ve{n}, \ve{n} + \ve{e}_i}\colon \SHI (Y\setminus{(\nu(L)}, \Gamma_{\ve{n}}, S_{\ve{n}}^{\tau(n)}, j) \to \SHI (Y\setminus \nu(L), \Gamma_{\ve{n}+ {\ve{e}_i}} S_{\ve{n}+\ve{e}_i}^{\tau (\ve{n}+ {\ve{e}_i})}
,j)$ shifts the grading by 1. After we consider the shift from Equation~\eqref{Eq: shift limit for link}, we observe that the codomain of $\phi^+_{\ve{n}, \ve{n} + \ve{e}_i}$ gets $-1$ shift relative to the domain. Hence the positive bypass maps, $\phi^+_{\ve{n}, \ve{n} + \ve{e}_i}$ are grading preserving under the shift from Equation~\eqref{Eq: shift limit for link}. Similarly one can argue for the negative bypass maps all $\phi^-_{\ve{n}, \ve{n} + \ve{e}_i}$ shift the grading by $-1$ after the shift. Hence the lemma follows.
  \end{proof}
  The first author and Li \cite{GhoshLiDecomposing}*{Section~6.2} showed that the direct limit construction defines a functor from the category $(\N^r, \le)$ to groups. Here $(\N^r,\le )$ denotes the category whose objects are tuples $(n_1,\dots, n_r)\in \N^r$ such that there is a morphism from $\ve{n}=(n_1,\dots, n_r)$ to $\ve{n}'=(n_1',\dots, n_r')$ if and only if $n_i\le n_i'$ for all $i$. The first author and Li define the functors value on the morphism from $\ve{n}=(n_1,\dots, n_r)$ to $\ve{n}'=(n_1',\dots, n_r')$ to be a composition of positive bypass maps. The first author and Li define $\KHI^-(Y,L)$ as the direct limit over this system. The action of $U_i$ is given by the negative bypass attachment $\phi^-_{\ve{n},i}$.

\begin{prop}\label{Prop: Z grading Link}
Using the grading on $\SHI (Y\setminus \nu(L), \Gamma_n, S_n^{\tau (\ve{n})}) [\sigma (\ve{n})]$, we can construct a $\mathbb{Z}$ grading in $\KHI^-(Y, L)$.
\end{prop}
\begin{proof}
The Proposition is an immediate consequence of Lemma~\ref{lem: positive negative shift}.
\end{proof}

\subsection{Proof of Theorem~\ref{thm: skein}}

\begin{figure}[ht]
\centering
\begingroup%
  \makeatletter%
  \providecommand\color[2][]{%
    \errmessage{(Inkscape) Color is used for the text in Inkscape, but the package 'color.sty' is not loaded}%
    \renewcommand\color[2][]{}%
  }%
  \providecommand\transparent[1]{%
    \errmessage{(Inkscape) Transparency is used (non-zero) for the text in Inkscape, but the package 'transparent.sty' is not loaded}%
    \renewcommand\transparent[1]{}%
  }%
  \providecommand\rotatebox[2]{#2}%
  \newcommand*\fsize{\dimexpr\f@size pt\relax}%
  \newcommand*\lineheight[1]{\fontsize{\fsize}{#1\fsize}\selectfont}%
  \ifx\svgwidth\undefined%
    \setlength{\unitlength}{83.82180174bp}%
    \ifx\svgscale\undefined%
      \relax%
    \else%
      \setlength{\unitlength}{\unitlength * \real{\svgscale}}%
    \fi%
  \else%
    \setlength{\unitlength}{\svgwidth}%
  \fi%
  \global\let\svgwidth\undefined%
  \global\let\svgscale\undefined%
  \makeatother%
  \begin{picture}(1,0.98493885)%
    \lineheight{1}%
    \setlength\tabcolsep{0pt}%
    \put(0,0){\includegraphics[width=\unitlength,page=1]{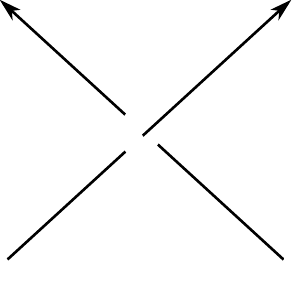}}%
    \put(0.33671223,0.01024251){\color[rgb]{0,0,0}\makebox(0,0)[lt]{\lineheight{1.25}\smash{\begin{tabular}[t]{l}$L_+$\end{tabular}}}}%
    \put(0,0){\includegraphics[width=\unitlength,page=2]{fig55.pdf}}%
    \put(0.42934052,0.67914766){\color[rgb]{0,0,0}\makebox(0,0)[lt]{\lineheight{1.25}\smash{\begin{tabular}[t]{l}$l$\end{tabular}}}}%
  \end{picture}%
\endgroup%

\caption{The link $L_+$, with an unknot $l$ around a crossing, with linking number zero. }
\label{fig:55}
\end{figure}

Let $l$ be an unknot which encircles the two strands of $L_{+}$ with total linking number zero, as shown in Figure~\ref{fig:55}. Let $Y_{-1}$ and $Y_0$ be the $3$-manifolds obtained from $Y$ by $-$1-surgery and 0-surgery on $l$, respectively. We view $L_-$ and $L_0$ as being the knots obtained from $L_+$ after performing surgery on $l$. We observe that $(Y_{-1} \setminus \nu(L_+), \Gamma_{\ve{n}})= (Y \setminus \nu(L_-), \Gamma_{\ve{n}})$ for any $\ve{n} \in \Z^r$. The surgery exact triangle immediately gives us 
\[
\cdots \SHI(Y\setminus \nu(L_{+}), \Gamma_{\ve{n}}) \xrightarrow{} \SHI(Y\setminus \nu(L_{-}), \Gamma_{\ve{n}}) \xrightarrow{} \SHI(Y_0\setminus \nu(L_{+}), \Gamma_{\ve{n}})\xrightarrow{} \SHI(Y\setminus \nu(L_{+}), \Gamma_{\ve{n}}) \cdots
\] Since direct limits preserve exactness, we obtain an exact triangle
\[
\cdots \KHI^-(Y, L_{+}) \xrightarrow{} \KHI^-(Y, L_{-}) \xrightarrow{} \ \KHI^-(Y_0, L_{+})\xrightarrow{} \KHI^-(Y, L_{+}) \cdots 
\]

\begin{prop} Let $L_+,$ $L_-$ and $L_0$ be as in the Skein exact triangle.
\begin{enumerate}
\item If the two strands from $L_0$ belong to different link components say $L_1$ and $L_2$, then
\[
\KHI^-(Y_0,L_+)\iso  \Cone(U_1 - U_2\colon \KHI^-(Y,L_0)[-1]\to \KHI^-(Y,L_0)).
\]
\item If the two strands from $L_0$ belong to the same link component, then
\[
\KHI^-(Y_0,L_+)\iso \KHI^-(L_0)\otimes W  
\] where $W=\oplus_{s\in \Z} W_s$ is the Alexander graded module given by 
\[
    W_s= 
\begin{cases}
    \C,& \text{if s=0, 1 } \\
    0,              & \text{otherwise.}
\end{cases}
\]
\end{enumerate}
\end{prop}
\begin{proof}

Since $\KHI^-(Y_0, L_+) \simeq \varinjlim{(\SHI(Y_0\setminus \nu(L_{+}), \Gamma_{\ve{n}})}$, we further note that topologically the sutured manifold $(Y\setminus \nu(L_0), \Gamma_{\ve{n}})$ can be obtained by decomposing along an annulus from the sutured manifold $(Y_0\setminus \nu(L_+), \Gamma_{\ve{n}})$. This was initially observed by Kronheimer and Mrowka \cite{KMskein} and is a key ingredient in their proof of \cite{KMskein}*{Theorem 3.1}. See \cite{KMskein}*{Figure~4}. While working with merdional sutures, this annulus becomes a product annulus and hence decomposing along this annulus gives an isomorphism between their corresponding sutured Floer homologies. In our case, we can consider the sutured link complement $(Y\setminus \nu(L_0), \Gamma_{\ve{n}})$ and apply our construction from Section~\ref{sec:connected-sum-HF} to obtain $(Y_0\setminus \nu(L_{+}), \Gamma_{\ve{n}''})$ by gluing this annulus. If the two strands from $L_0$ belong to different link components, by Theorem~\ref{thm:gradings link 1} there is a grading preserving isomophism
\begin{equation}\label{eq:expanded-complex skein 1}
\SHI(Y_0\setminus \nu(L_{+}), \Gamma_{\ve{n''}}) \simeq \Cone \left (
\begin{tikzcd}[column sep=4cm, row sep=1.5cm, labels=description]C_{\ve{n}-\ve{e}_1} \oplus C_{\ve{n}-\ve{e}_2} [-1] \ar[r,"\phi_{\ve{n}-\ve{e}_1, \ve{n}}^-\oplus \phi_{\ve{n}-\ve{e}_2, \ve{n}}^-"] \ar[d, "\phi_{\ve{n}-\ve{e}_1, \ve{n}}^+ \oplus \phi_{\ve{n}-\ve{e}_2, \ve{n}}^+"]& C_{\ve{n}}\\
C_{\ve{n}} [-1]&\,
\end{tikzcd} \right ),
\end{equation} where $\ve{n}= (n_1, n_2, n_3, \dots, n_r)$, $\ve{n''}=(n_1 + n_2, n_3, \dots, n_r)$, $\ve{e_1}=(1, 0, \dots,0)$ and $\ve{e_2}=(0, 1, \dots,0)$.

As in the proof of Theorem~\ref{thm: main}, the complex from Equation~\eqref{eq:expanded-complex skein 1} can be reduced to a simpler complex.  Let $q\in \Z$ be a fixed integer and as before if $G$ is a graded group, write $G_{>q}$ for the subgroup concentrated in gradings above $q$. Note that the link Floer groups $\KHI^-$ has an Alexander grading from Proposition~\ref{Prop: Z grading Link} which is bounded from above by the same argument from the proof of \cite{LiLimits}*{Theorem~1.1}. Furthermore, each $U_i$ shifts the grading by $-1$.  We will show for each $q\in \Z$ that $\KHI^-(Y_0,L_+)_{>q}$ is isomorphic to $(\KHI^-(L_0)\tildeotimes \mathbb{C}[U_1,U_2]/(U_1-U_2))_{>q}$. First, we assume that $\ve{N}$ is chosen so that if $\ve{n}>\ve{N}$, then the positive bypass map $\phi_{\ve{n} - \ve{e}_i, \ve{n}}^+\colon \SHI(Y\setminus \nu(L_+),\Gamma_{\ve{n}- \ve{e}_i}) \to \SHI(Y\setminus \nu(L_+),\Gamma_{\ve{n}})$ is an isomorphism on gradings above $q$ for $i= 1, 2$. By Lemma~\ref{lem:staircase-lemma}, there is a grading preserving chain map
\[
\Phi\colon C_{\ve{n} - \ve{e}_1 - \ve{e}_2}\to \Cone(\begin{tikzcd}[column sep=3cm]
C_{\ve{n}-\ve{e}_2}\oplus C_{\ve{n}-\ve{e}_1} \ar[r, "\phi_{\ve{n}-\ve{e}_1,\ve{n}}^+\oplus \phi_{\ve{n}-\ve{e}_2,\ve{n}}^+"]& C_{\ve{n}}
\end{tikzcd}),
\]
as well as a grading preserving chain map $\Psi$ in the opposite direction which are isomorphisms in degrees above $q$. Since the differential preserves the Alexander grading, we conclude that  the subcomplex of Equation~\eqref{eq:expanded-complex skein 1} in grading above $q$ is homotopy equivalent to the complex
\[
\Cone(
\begin{tikzcd}[column sep=2cm]
(C_{\ve{n}})_{>q+1} \ar[r, "U_1-U_2"]& (C_{\ve{n}})_{>q})
\end{tikzcd}
\]
which is exactly $(\KHI^-(L_0)\tildeotimes \mathbb{C}[U_1,U_2]/(U_1-U_2))_{>q}$.  Thus the proof of Theorem~$\ref{thm: main}$ adapts to show that 
\[
 \KHI^-(Y_0, L_{+}) \iso \Cone(
\begin{tikzcd}[column sep = 2cm]
\KHI^-(Y, L_0)[-1]\ar[r, "U_1 -  U_2"]& \KHI^-(Y, L_0))
\end{tikzcd}, 
\]
where $U_1$ and $U_2$ correspond to the two $U$-actions on the link Floer homology $\KHI^-(Y, L_0)$ corresponding to the two strands of the components $L_1$ and $L_2$ inside of B. But this is exactly $\KHI^-(K_0, s)\tildeotimes \mathbb{C}[U_1,U_2]/(U_1-U_2)$.  We remark, that the $U_i$ actions are only defined up to overall multiplication by elements of $\C^\times$. Nonetheless, by Lemma~\ref{lem:rescale-module}, scaling each $U_i$ by elements of $\C^\times$, results in isomorphic module structure on $\KHI^-(Y,L_0)$, so our argument still identifies the mapping cone with the derived tensor product on the level of graded groups.

If the two strands from $L_0$ belong to the same link component, 
we will show in Theorem~\ref{thm:gradings link 1} (below) that there is a grading preserving isomorphism
\begin{equation}\label{eq:expanded-complex skein 2}
\SHI(Y_0\setminus \nu(L_{+}), \Gamma_{\ve{n}}) \simeq \Cone \left (\begin{tikzcd}[column sep=3cm, row sep=1.5cm, labels=description]C_{\ve{n'}} \oplus C_{\ve{n'}} \ar[r,"\phi_{\ve{n'},\ve{n''}}^-\oplus \phi_{\ve{n'}, \ve{n''}}^-"] \ar[d, "\phi_{\ve{n'},\ve{n''}}^+\oplus \phi_{\ve{n'},\ve{n''}}^+"]& C_{\ve{n''}} [1]\\
C_{\ve{n''}}&\,
\end{tikzcd}\right),
\end{equation}where $\ve{n}= (n_1, n_2, n_3, \dots, n_r)$, $\ve{n'}=(n_1 + n_2 -1, n_3, \dots, n_r)$, $\ve{n''}=(n_1 + n_2, n_3, \dots, n_r)$. 

Similarly to the above case, the proof of Theorem~\ref{thm: main} adapts to show that 
\begin{equation}
 \KHI^-(Y_0, L_{+}) \iso \Cone(
\begin{tikzcd}[column sep = 3 cm]
\KHI^-(Y, L_0)\ar[r, "U_1- U_2"]& \KHI^-(Y, L_0)[+1])
\end{tikzcd}.
\label{eq:derived-tensor-w/U=0}
\end{equation}
We observe now that the $U_i$ actions now correspond to positive bypass attachments on the same link component of $L_0$. Hence $U_1=\a U_2$ for some $\a\in \C^\times$. We claim that $\a$ must be 1.

To show that $\a=1$, we consider the homology of the mapping cone in sufficiently negative Alexander gradings. The above argument shows that
\begin{equation}
\KHI^-(Y_0,L_+, \ve{i})\iso \Cone((1-\a)U_1\colon \KHI^-(Y,L_0, \ve{i})\to \KHI^-(Y,L_0,\ve{i}+\ve{e}))\label{eq:mapping-cone-Alexander-grading}
\end{equation}
where $\ve{e}$ is the unit vector in the component for $L_0$. We will show in Lemma~\ref{lem:localization link}, below, that
\[
\KHI^-(Y_0,L_+,\ve{i})\iso I^{\#}(Y\#^{r} S^1\times S^2)
\]
if all components of $\ve{i}$ are sufficiently negative. On the other hand, using the bypass exact triangle for the $U_i$ map, it is straightforward to see that when $i_1\ll 0$, the map $U_i$ is an isomorphism in each Alexander grading. In particular, if $\a\neq 1$, then Equation~\eqref{eq:mapping-cone-Alexander-grading} will have trivial homology in Alexander grading $\ve{i}$, a contradiction.

Hence, Equation~\eqref{eq:derived-tensor-w/U=0},  reduces to $\KHI^-(L_0)\otimes W$, where $W=\bigoplus_{s\in \Z} W_s$ is the graded module
\[
    W_s= 
\begin{cases}
    \C,& \text{if s=0, 1 } \\
    0,              & \text{otherwise.}
\end{cases}
\]

\end{proof}

\begin{lem}\label{lem:localization link}
 Let $L$ be an $r$ component link in a closed oriented 3-manifold $Y$. Let $\ve{i}=(i_1, \dots, i_r)$ denote the multi-grading in $\KHI^-(Y,L,\ve{i})$. Write $Y(r)$ for the sutured manifold obtained from $Y$ by deleting $r$ 3-balls and putting a simple closed curve on each sphere component as the suture. Then \[\KHI^-(Y,L,\ve{i})\iso \SHI (Y(r))\iso I^\# (Y \#^{r-1} S^1\times S^2)\] if $\ve{i}$ is sufficiently negative. 
\end{lem}
\begin{proof}
Note that in the case of knots, the  lemma is proven in \cite{GhoshLiWongTau}*{Proposition~1.13}.  Write $\ve{n}=(n_1,\dots, n_r)$ and $\ve{i}=(i_1,\dots, i_r)$. There is a map
\[
C\colon \SHI(Y\setminus \nu(L), \Gamma_{\ve{n}})\to \SHI(Y(r))
\]
corresponding to attaching $r$ contact 2-handles along meridians of the components of $L$. We claim that if $\ve{i}$ is fixed, and sufficiently negative, and $\ve{n}$ is sufficiently large, then $C$ restricts to an isomorphism from $\SHI(Y\setminus \nu(L), \Gamma_{\ve{n}}, \ve{i})$ onto $\SHI(Y(r))$. 

To see this, one may argue by induction, using the techniques of Li and Ye, as we now sketch. We sketch the proof. Let $\ve{n}'=(n_2,\dots, n_r)$ and $\ve{i}'=(i_2,\dots, i_r)$. There is a contact 2-handle map 
\[
C'\colon \SHI(Y\setminus \nu(L), \Gamma_{\ve{n}})\to \SHI(Y(1)\setminus \nu(L'), \Gamma_{\ve{n}'})
\]
where $L'=L_2\cup \cdots \cup L_r$. It follows from \cite{LYEuler}*{Proposition~4.6} that $C'$ maps grading $\ve{i}$ to grading $\ve{i}'$.

Next, it follows from the same technique as in \cite{LYHeegaard}*{Lemma~4.26} that if $i_1\ll 0$ is fixed, then for all sufficiently large $n_1$, the 2-handle map $C'$ will be an isomorphism from grading $\ve{i}$ to grading $\ve{i}'$. Proceeding by induction, the claim follows.
\end{proof}

\subsection{Gradings}
We now prove Theorem~\ref{thm:gradings link 1} and describe a grading on $\SHI(Y_{0} \setminus \nu(L_+), \Gamma_{\ve{m}})$ which is compatible with the isomorphism from Theorem~\ref{thm:gradings link 1}, where $m=(m_1, m_2, \dots, m_k)$. Let $C_{\ve{n}}:=\SHI (Y\setminus \nu(L_0), \Gamma_{\ve{n}})$, where $\ve{n} = (n_1, n_2, \dots, n_r)$. Note that $k-r=\pm 1$. We then consider the isomorphism arising from the iterated mapping cone construction arising from gluing an annulus from Section~\ref{sec:connected-sum-HF}.
Following \cite{GhoshLiDecomposing} and  Proposition~\ref{Prop: Z grading Link}, each of $\SHI(Y_{0} \setminus \nu(L_+), \Gamma_{\ve{m}})$ and $C_{\ve{n}}$, has  a natural Alexander grading which is compatible with the direct limit construction, in the sense that the positive bypass maps are grading preserving, and the negative bypass maps, $U_i$ shift grading by $-1$. In particular, $\varinjlim C_{\ve{n}} \iso \KHI^-(Y, L_0)$ as a graded $\mathbb{C}[U_1, U_2, \dots, U_r]$-module, and similarly for $\SHI(Y_{0} \setminus \nu(L_+), \Gamma_{\ve{m}})$. 
 
 We recall more generally, the first author and Li's construction gives an Alexander grading for any choice of Seifert surface.
 Again as before, we write $S_{\ve{n}} \subset Y \setminus \nu(K_0)$ for a Seifert surface which intersects $\Gamma_{\ve{n}}$ geometrically in $2(\sum_{i=1}^{r} n_i)$ points. If $\tau \in \Z$, we write $S_{\ve{n}}^{\tau}$ for the surface $S_{\ve{n}}$, which is stabilized $\tau$ times algebraically.  

The main theorem of this section is the following:
\begin{thm}\label{thm:gradings link 1}
Let $L_+,$ $L_-$ and $L_0$ be as in the Skein exact triangle.
\begin{enumerate}
    \item If the two strands from $L_0$ belong to different link components say $L_1$ and $L_2$, with respect to the above gradings and the following shift from Equation \eqref{Eq: shift limit for link}, the quasi-isomorphism
\begin{equation}\label{eq:map-X-n-m oriented 1}
X_{n_1, n_2}=\begin{tikzcd}
\SHI \left(Y_0 \setminus \nu(L_+), \Gamma_{\ve{n''}}, S_{\ve{n''}}^{\tau(\ve{n''})}\right)
\arrow[drr, dashed ]
\arrow[ddr, dashed]
\arrow[dr] & & \\
& C_{\ve{n}-\ve{e_2}} \oplus C_{\ve{n}-\ve{e_1}} [-1]  \arrow[r] \arrow[d]
& C_{\ve{n}} \\
& C_{\ve{n}} [-1]
\end{tikzcd}
\end{equation}
from the iterated mapping cone construction is grading preserving. In the above, $\tau\colon \Z^r\to \Z$ is the function defined in Equation \eqref{Eq: shift limit for link}, $\ve{n}= (n_1, n_2, n_3, \dots, n_r)$, $\ve{n''}=(n_1 + n_2, n_3, \dots, n_r)$, $\ve{e_1}=(1, 0, \dots,0)$ and $\ve{e_2}=(0, 1, \dots,0)$.

\item If the two strands from $L_0$ belong to the same link component, with respect to the above gradings and the following shift from Equation~\eqref{Eq: shift limit for link}, the quasi-isomorphism
\begin{equation}\label{eq:map-X-n-m oriented 2}
X_{n_1, n_2}=\begin{tikzcd}
\SHI \left(Y_0 \setminus \nu(L_+), \Gamma_{\ve{n}}, S_{\ve{n}}^{\tau(\ve{n})}\right)
\arrow[drr, dashed ]
\arrow[ddr, dashed]
\arrow[dr] & & \\
& (C_{\ve{n'}} \oplus C_{\ve{n'}})  \arrow[r] \arrow[d]
& C_{\ve{n''}} [+1] \\
& C_{\ve{n''}}
\end{tikzcd}
\end{equation}
from the iterated mapping cone construction is grading preserving. In the above, $\tau\colon \Z^r\to \Z$ is the function defined in Equation~\eqref{Eq: shift limit for link}, $\ve{n}= (n_1, n_2, n_3, \dots, n_r)$, $\ve{n'}= (n_1+n_2-1, n_3, \dots, n_r)$ and $\ve{n''}=(n_1 + n_2, n_3, \dots, n_r) $. On $C_{\ve{n'}}$, we use the grading surface $S_{\ve{n'}}^{\tau(\ve{n'})}$ and on $C_{\ve{n''}}$ we use the grading surface $S_{\ve{n''}}^{\tau(\ve{n''})}$.
\end{enumerate}

\end{thm}

 \begin{proof}
The proof is analogous to the proof of Theorem~\ref{thm:gradings}. As before, we will only deal with the case when $\tau(\ve{n}) =0$ in Equation~\ref{eq:map-X-n-m oriented 1} and $\tau(\ve{n''})=0$ in Equation~\ref{eq:map-X-n-m oriented 2}. By the same reasoning as in the proof of Theorem~\ref{thm:gradings}, the map $X_{n_1,n_2}$ shifts the grading in the same way as the map $X_{n, m}$ from Theorem~\ref{thm:gradings}. Finally we consider the grading shift from Equation~\eqref{Eq: shift limit for link}. Our formula for the shift depends on the number of link components and hence we get the two different statements based on whether $L_0$ increases or decreases the number of components.
 \end{proof}
 
Finally we prove the graded version of Theorem~\ref{thm: skein}.
\begin{lem}
The maps arising in the skein exact triangle in Theorem~\ref{thm: skein} are all grading preserving with respect to the Alexander grading from Subsection~\ref{Sub: grading for link}.
\end{lem}

\begin{proof}
The map $f$ is a cobordism map whereas the maps $g$ and $h$ are more complicated. The map $f$ is given by attaching a $-1$ framed 2-handle along the unknot $l$. Note that $l$ has linking number 0 with the link $L_+$. Hence, we may pick a Seifert surface $S$ for $L_+$ which is disjoint from $l$. Since the grading is independent of the choice of  Seifert surface and $S$ survives in the cobordism, $f$ is grading preserving. See the proof of \cite{GhoshLiWongTau}*{Proposition~1.12} for more details.

 The map $g$ and $h$ are both compositions of two maps. The map $g$ can be written as ${X_{n_1, n_2}} \circ g_0$, where the map $g_0$ is the cobordism map corresponding to attaching a $-1$ framed 2-handle along the meridian of $l$ and $X_{n_1,n_2}$ is the quasi-isomorphism from Section~\ref{sec:connected-sum-HF} corresponding to gluing the meridional annulus in the link exterior. The map $g$ is grading preserving as it is given by attaching a $-1$ framed 2-handle along the meridian of $l$, which can again be made disjoint from the grading surface. The map ${X_{n_1, n_2}}$ is grading preserving by Theorem~\ref{thm:gradings link 1}. Hence, $g$ is grading preserving. The map $h$ can be written as $ h_0 \circ Y_{n_1,n_2}$ where $Y_{n_1,n_2}$ is the inverse map of $X_{n_1,n_2}$ and the map $h_0$ is the cobordism map arising from attaching $0$-framed $2$-handle along the meridian of $l$. By similar reasoning, the map $h$ is grading preserving.
\end{proof}

\section{The $U$-action on connected sums}
\label{sec:U-action}

In this section, we prove the following:

\begin{thm}
\label{thm:U-action} Suppose that $K$ and $K'$ are knots. The $\C[U]$-module $\KHI^-(K\# K')$ is isomorphic to $\KHI^-(K)\tildeotimes \KHI^-(K')$, where $U$ acts by $1\otimes U$ on
\[
\Cone\left(U|1-1|U\colon\KHI^-(K)\otimes_{\C} \KHI^-(K')\to \KHI^-(K)\otimes_{\C} \KHI^-(K')\right). 
\]
\end{thm}

Our argument is somewhat indirect. Instead of computing the $U$ action on the connected sum directly, we will view $K\# K'$ as a triple tensor product $K\# u\# K'$, where $u$ is an unknot. We will view the $U$ action as corresponding to a positive bypass attached to the unknot $u$. In Section~\ref{sec:algebraic-interpretation-U-action}, we give an algebraic interpretation of our argument in terms of $A_\infty$-modules and bimodules over $\C[U]$. In Sections~\ref{sec:U-map-first-version} and~\ref{sec:thm-U-action}, we prove the above theorem.

\begin{rem} Our proof does not hold if we replace $K$ and $K'$ with links. In this case, the instanton link Floer complexes have a $U_i$ action for each link component. The correct statement should involve an $A_\infty$-module structure on $\KHI^-(L)$.
\end{rem}

\subsection{Algebraic motivation}
\label{sec:algebraic-interpretation-U-action}
In this section, we discuss a purely algebraic result which we use as a model for our proof of the formula for the $U$-action on the tensor product.  Although our proof does not directly invoke this result, it is the conceptual basis of our argument. Here, we assume the reader is familiar with the notions of $A_\infty$-modules and bimodules. See \cite{KellerNotes} for an introduction, or \cite{LOTBimodules} for an exposition using our present notation.

\begin{lem}
\label{lem:U-action-A_infty-tensor-product}
Write $\cA=\bF[U]$. Suppose that $M_{\cA}$ and ${}_{\cA} N$ are $A_\infty$ modules, and ${}_{\cA} P_{\cA}$ is an $A_\infty$-bimodule,
with a striclty unital bimodule morphism 
\[
f_{i,1,j}\colon {}_{\cA} P_{\cA}\to {}_{\cA} P_{\cA}
\]
 Suppose furthermore that the following hold:
\begin{enumerate}
\item $P\iso \bF[U]$ as a vector space. Furthermore, $P$ is equipped with an Alexander grading and the structure maps are Alexander grading preserving.
\item The structure maps $m_{0,1,0}$ and $m_{1,1,1}$ vanish on $P$, and $m_{1,1,0}$ and $m_{0,1,1}$ coincide with ordinary multiplication by $U$.
\item The morphism $f_{i,1,j}$ satisfies $f_{0,1,0}=U$ and $f_{1,1,0}=0$ and $f_{0,1,1}=0$. 
\end{enumerate}
Then
\[
M_\cA \tildeotimes {}_{\cA} P_{\cA} \tildeotimes {}_{\cA} N\simeq M_{\cA}\tildeotimes {}_{\cA} N.
\]
Furthermore, the induced map $\id\tildeotimes f\tildeotimes \id$ coincides with the action of $U$ on homology.
\end{lem}
\begin{proof} Here, we are using the type-$DD$ module interpretation of the derived tensor product described in Equation~\eqref{eq:derived-tensor-bimodules}, and also we are viewing $\id\tildeotimes f\tildeotimes \id$ correspondingly as a box tensor product of morphisms (cf. \cite{LOTBimodules}*{Section~2}). The statement that about $M\tildeotimes P\tildeotimes N\iso M\tildeotimes N$ is seen by observing that $M\tildeotimes P\tildeotimes N=M\tildeotimes \bF[U]\tildeotimes N$ as groups. The underlying complexes are clearly isomorphic, and our assumption about the structure maps $m_{i,1,j}$ imply that the differentials are identical. However, $M\tildeotimes \bF[U]\simeq M$, as a right module over $\cA$, so the claim follows.

We now consider the claim about $\id\tildeotimes f\tildeotimes \id$. We will show that $\id\tildeotimes f\tildeotimes \id$ has the same action on homology as $\id\tildeotimes U\tildeotimes \id$, which is clearly sufficient. 

We may view $M\tildeotimes P\tildeotimes N$ as the following 2-dimensional cube:
\[
\begin{tikzcd}[row sep=2cm, column sep=2cm]
M\otimes \bF[U]\otimes N
 \ar[r, "U|1|1+1|U|1"] 
 \ar[d, "1|1|U+1|U|1"]
&
 M\otimes \bF[U]\otimes N
 \ar[d,"1|1|U+1|U|1"]
 \\
M\otimes \bF[U]\otimes N
	\ar[r, "U|1|1+1|U|1"]
&
M\otimes \bF[U]\otimes N
\end{tikzcd}
\]
(Tensor products taken over $\bF$).
With respect to this description, the map $\id\tildeotimes f\tildeotimes \id$ takes the following form:
\[
 \begin{tikzcd}[row sep=5cm, labels=description]
 M\tildeotimes P\tildeotimes N
 \ar[d, "\id\tildeotimes f\tildeotimes \id"]
 \\
  M\tildeotimes P\tildeotimes N
 \end{tikzcd}
 \!\!\!\!\!\!
 =
 \!\!\! 
 \begin{tikzcd}[column sep={2.5cm,between origins},row sep=1.2cm,labels=description]
M|F[U]|N
	\ar[rr, "U|1|1+1|U|1"]
	\ar[dr,"1|1|U+1|U|1"]
	\ar[ddd, "1|f_{0,1,0}(-)|1"]
	\ar[ddddrrr,dotted]
	\ar[rrddd,dashed,"1|{f_{1,1,0}(U{,},-)|1}", sloped]
&[.7 cm]\,
&M|F[U]|N
	\ar[ddd, "1|f_{0,1,0}(-)|1"]
	\ar[dr,"1|1|U+1|U|1"]
	\ar[ddddr,dashed,"{1|f_{0,1,1}(-,U)|1}"]
&[.7 cm]\,
\\
&[.7 cm] 
M|F[U]|N
	\ar[rr, crossing over,"U|1|1+1|U|1"]
	\ar[dddrr,dashed, crossing over, "{1|f_{1,1,0}(U,-)|1}",sloped]
&\,
&[.7 cm]M|F[U]|N
	\ar[ddd, "1|f_{0,1,0}(-)|1"]
\\
\\
M|F[U]|N
	\ar[dr,"1|1|U+1|U|1"]	
	\ar[rr,"U|1|1+1|U|1"]
&[.7 cm]\,
&
M|F[U]|N
	\ar[dr, "1|1|U+1|U|1"]
&[.9 cm]\,\\
& [.7 cm]
M|F[U]|N
	\ar[from=uuuul, dashed, crossing over, "{1|f_{0,1,1}(-,U)|1}"]
	\ar[from=uuu,crossing over, "1|f_{0,1,0}(-)|1"]
	\ar[rr, "U|1|1+1|U|1"]&\,
&[.9 cm]
M|F[U]|N
\end{tikzcd}
 \]
 The long dashed arrow is $1|f_{1,1,1}(U,-,U)|1$. Note that our assumptions imply that $f_{0,1,1}(-,U)$ and $f_{1,1,0}(U,-)$ vanish.

We think of the  chain complex $M\tildeotimes P\tildeotimes N$ as being  of the form
\[
M\tildeotimes P\tildeotimes N=(C_2\to C_1\to C_0),
\]
where the index $q$ in $C_q$ denotes the cube grading. However, it is clear that the map $C_2\to C_1$ is injective. Hence, every cycle may be written as a sum of elements in $C_1$ and $C_0$. Our assumptions on $f$ imply that 
\[
(\id\tildeotimes f\tildeotimes \id)|_{C_1\oplus C_0}=(\id\tildeotimes U\tildeotimes \id)|_{C_1\oplus C_0},
\]
so in particular the two maps coincide on all cycles, and hence on homology. 
\end{proof}

\subsection{The $U$-action on connected sums}
\label{sec:U-map-first-version}

In this section, we describe preliminary results towards the $U$ action on the Floer homology of the connected sum $K\# K'$. The description in this setting is not sufficient for our proof of Theorem~\ref{thm:U-action}, but is an important first step.
We attach a positive bypass to a portion of the boundary corresponding to $K$.  On the level of complexes, this may be understand by adding an extra dimension to the iterating mapping cone corresponding to this direction. Write $\phi^-$ for the negative bypass map on the sutured Floer homology of $K\# K'$ (for longitudinal framing $n+m$).  Arguing as in Theorem~\ref{thm:link-surgery-instanton}, one obtains that 
\[
X_{n+1,m}\circ \phi^-\simeq F^-_{n,m}\circ X_{n,m}
\]
where $F_{n,m}^-$ takes the form of the solid arrows in the diagram below:
\[
F_{n,m}^-=
\begin{tikzcd}[labels=description, row sep=.7cm, column sep=2.5cm]
&C_{n}|D_{m-1}
	\ar[ddl, "\id|\phi^+",dashed]
	\ar[ddr, "\id|\phi^-",dashed]
	\ar[ddd,shift left=3, "\phi^-|\id"]
	\ar[dddddl, "j_1"]
	\ar[dddddr, "j_2"]
&\,
\\
&C_{n-1}|D_m
	\ar[dl, "\phi^+|\id",dashed]
	\ar[dr, "\phi^-|\id",dashed]
	\ar[ddd,shift right=3,"\phi^-|\id"]
	\ar[ddddl, "h_1"]
	\ar[ddddr, "h_2"]
\\
C_n|D_m
	\ar[ddd,"\phi^-|\id"]
&&
 C_n|D_m
 	\ar[ddd,"\phi^-|\id"]
\\[1cm]
&C_{n+1}|D_{m-1}
	\ar[ddl, "\id|\phi^+",dashed]
	\ar[ddr, "\id|\phi^-",dashed]
&\,
\\
&C_{n}|D_m
	\ar[dl, "\phi^+|\id",dashed]
	\ar[dr, "\phi^-|\id",dashed]
\\
C_{n+1}|D_m&& C_{n+1}|D_m
\end{tikzcd}
\]
In the above,
\begin{enumerate}
\item $j_1$ is a homotopy between $(1|\phi^+)\circ (\phi^-|1)$ and $(\phi^-|1)\circ (1|\phi^+)$.
\item $j_2$ is a homotopy between $(\phi^-|\id)\circ (\id|\phi^-)$ and $(\id|\phi^-)\circ (\phi^-|\id)$.
\item $h_1$ is a homotopy between $(\phi^-|1)\circ (\phi^+|1)$ and $(\phi^+|1)\circ (\phi^-|1)$.
\item $h_2$ is a homotopy between the two compositions of $(\phi^-|1)\circ (\phi^-|1)$. (Note that since these two copies of $\phi^+$ are realized by two different handle attachments, this homotopy corresponds to a handle rearrangement).
\end{enumerate}
Next, we observe that the above is natural with respect to choices of closures. I.e. by picking different closures of $K$ and $K'$, we obtain a homotopy equivalence between diagrams constructed as above. By picking a disconnected closure of 
\[
(Y\setminus \nu(K))\sqcup (Y'\setminus \nu(K')),
\]
 we may take $j_1$ and $j_2$ to be zero on the nose. Indeed, they both correspond to rearranging 4-dimensional handles which are attached to different components of the sutured 3-manifold. Hence a 1-parameter family of metrics relating the two handle attachments may be taken to be the constant family. In particular, no solutions are counted by transversality considerations.

Further, $h_1$ and $h_2$ correspond to reordering 2-handles on a single component of the underlying 3-manifold. Hence, we may write $h_1=h^{[+,-]}|1$ and $h_2=h^{[+,+]}|1$, where $h^{[+,-]}$ and $h^{[+,+]}$ are the canonical homotopies corresponding to handle rearrangements on $Y\setminus \nu(K)$.

We obtain that $F_{n,m}^-$ may be simplified to the following diagram. 
\[
F_{n,m}^-=
\begin{tikzcd}[labels=description, row sep=.7cm, column sep=2.5cm]
&C_{n}|D_{m-1}
	\ar[ddl, "\id|\phi^+",dashed]
	\ar[ddr, "\id|\phi^-",dashed]
	\ar[ddd,shift left=3, "\phi^-|\id"]
&\,
\\
&C_{n-1}|D_m
	\ar[dl, "\phi^+|\id",dashed]
	\ar[dr, "\phi^-|\id",dashed]
	\ar[ddd,shift right=3,"\phi^-|\id"]
	\ar[ddddl,"h^{[-,+]}|1"]
	\ar[ddddr,"h^{[-,-]}|1"]
\\
C_n|D_m
	\ar[ddd,"\phi^-|\id"]
&&
 C_n|D_m
 	\ar[ddd,"\phi^-|\id"]
\\[1cm]
&C_{n+1}|D_{m-1}
	\ar[ddl, "\id|\phi^+",dashed]
	\ar[ddr, "\id|\phi^-",dashed]
&\,
\\
&C_{n}|D_m
	\ar[dl, "\phi^+|\id",dashed]
	\ar[dr, "\phi^-|\id",dashed]
\\
C_{n+1}|D_m&& C_{n+1}|D_m
\end{tikzcd}
\]

Applying the homotopy equivalence between the top face and the derived tensor product from Section~\ref{sec:tensor-products-transitive-systems}, we obtain that
\[
\Cone(F_{n,m}^-)\simeq 
\begin{tikzcd}[column sep=3cm, row sep=2cm, labels=description,labels=description] C_n|D_m
	\ar[r, "\phi^+|\phi^--\phi^-|\phi^+"]
	 \ar[d, "\phi^-|1"] 
	 \ar[dr, "h^{[-,+]}|\phi^-+h^{[-,-]}|\phi^+",dashed]
	 & C_{n+1}|D_{m+1}
	 	\ar[d, "\phi^-|1"]
	 \\
C_{n+1}|D_m\ar[r, "\phi^+|\phi^--\phi^-|\phi^+"] & C_{n+2}|D_{m+1}
\end{tikzcd}
\]

\subsection{Proof of Theorem~\ref{thm:U-action}}
\label{sec:thm-U-action}

We view $K\# K'$ as a triple tensor product of $K$, and unknot $u$, and $K'$. We apply the $U$ action (i.e. attach a positive bypass) to $u$. We write $\cU_\ell$ for the correspoding sutured complex for $u$ with longitudinal framing $\ell$.

 Performing our argument from Section~\ref{sec:U-map-first-version} twice, we obtain the diagram 
\begin{equation}
\begin{tikzcd}[column sep=3.5cm, row sep=3cm, labels=description,labels=description] C_n|\cU_\ell|D_m
	\ar[r, "\phi^+|\phi^-|1-\phi^-|\phi^+|1"]
	 \ar[d, "1|\phi^+|\phi^--1|\phi^-|\phi^+"] 
	 \ar[dr, "\substack{\phi^+|h^{[-,+]}|\phi^- \\+\phi^+|h^{[-,-]}|\phi^+\\
	 +\phi^-|h^{[-,+]}|\phi^+\\
	 +\phi^-|h^{[+,+]} |\phi^-}",dashed]
	 & C_{n+1}|\cU_{\ell+1}|D_m
	 	\ar[d, "1|\phi^+|\phi^--1|\phi^-|\phi^+"]
	 \\
C_{n}|\cU_{\ell+1}|D_{m+1}\ar[r, "\phi^+|\phi^-|1-\phi^-|\phi^+|1"] & C_{n+1}|\cU_{\ell+2}|D_{m+1}
\end{tikzcd}
\label{eq:iterated-derived-tensor-product}
\end{equation}
We pick linear maps
\[
\pi_n \colon C_n\to H_*(C_n),\quad i_n\colon H_*(C_n)\to C_n\quad \text{and}\quad  h_n\colon C_n\to C_n
\]
so that $i_n$ and $\pi_n$ are chain maps (viewing $H_*(C_n)$ as having vanishing differential),
\begin{equation}
\pi_n\circ i_n=\id\quad i_n\circ \pi_n=\id+[\d, h_n],\quad h_n\circ h_n=0,\quad h_n\circ i_n=0\quad \text{and} \quad \pi_n\circ h_n=0.
\label{eq:strong-deformation-retract}
\end{equation}
We pick similar homotopy equivalences for each $\cU_{\ell}$ and $D_m$. We may do this because the above complexes are finitely generated chain complexes over a field.

We can use these homotopy equivalences to construct similar homotopy equivalences $C_n|\cU_\ell|D_m\simeq H_*(C_n)|H_*(\cU_\ell)|H_*(D_m)$ which satisfy the same relations. For example, for inclusion and projection maps, we may take $i_n|i_\ell|i_m$ and $\pi_n|\pi_\ell|\pi_m$. For the homotopy, we take
\[
h_n| i_\ell \pi_\ell|i_m\pi_m+1|h_\ell|i_m\pi_m+1|1|h_m.
\]
Observe that these still satisfy the relations in Equation~\eqref{eq:strong-deformation-retract}.

Next, we apply homological perturbation of hypercubes to the above diagram. See \cite{HHSZDual}*{Section~2.7}. We observe that $H_*(\cU_{\ell})$ admits a mod 2 homological grading. This is because the instanton complex of the closure admits a mod 8 grading, and the eigenspace operators $\mu(pt)$ and $\mu(R)$ are homogeneously 4 and 2 graded. In particular, their generalized eigenspaces will admit a mod 2 grading. The space $H_*(\cU_\ell)$ is supported only in a single mod 2 grading. Furthermore, the maps $h^{[\circ,\circ']}$ shift the mod 2 homological grading by 1, and similarly we may assume the homotopies $h_{\ell}\colon \cU_\ell\to \cU_\ell$ also shift the mod 2 grading by 1. Furthermore, the maps $\phi^{\pm}$ on $\cU_\ell$ preserve the mod 2 grading. In particular, no homotopies will survive homological perturbation of hypercubes, and we obtain a homotopy equivalence with the diagram
\begin{equation}
\begin{tikzcd}[column sep=3.5cm, row sep=3cm, labels=description,labels=description] H_*C_n|H_*\cU_\ell|H_*D_m
	\ar[r, "\phi^+_*|\phi^-_*|1+\phi^-_*|\phi^+_*|1"]
	 \ar[d, "1|\phi^+_*|\phi^-_*+1|\phi^-_*|\phi^+_*"] 
	 & H_*C_{n+1}|H_*\cU_{\ell+1}|H_*D_m
	 	\ar[d, "1|\phi^+_*|\phi^-_*+1|\phi^-_*|\phi^+_*"]
	 \\
H_*C_{n}|H_*\cU_{\ell+1}|H_*D_{m+1}\ar[r, "\phi^+_*|\phi^-_*|1+\phi^-_*|\phi^+_*|1"] & H_*C_{n+1}|H_*\cU_{\ell+2}|H_*D_{m+1}
\end{tikzcd}
\label{eq:iterated-derived-tensor-product-Hom}
\end{equation}
Note that in Alexander gradings above some fixed $q$, if $n$, $\ell$ and $m$ are sufficiently large, then the above coincides with the triple derived tensor product $\KHI^-(K)\tildeotimes \C[U]\tildeotimes \KHI^-(K')$. Compare the first part of Lemma~\ref{lem:U-action-A_infty-tensor-product}.

We now add an extra dimension for the $U$ map. Performing the natural extension of the above argument,  we obtain a homotopy equivalence between the following complexes:
 \[
 \begin{tikzcd}[row sep=5cm]
 \SHI\left(Y\# Y'\setminus \nu(K\# K'),\Gamma_{n+\ell+m}\right)
 \ar[d, "\phi^-"]
 \\
 \SHI\left(Y\# Y'\setminus \nu(K\# K'),\Gamma_{n+\ell+m}\right)
 \end{tikzcd}
 \hspace{-1cm}
 \simeq
 \!\!\! 
 \begin{tikzcd}[column sep={2.5cm,between origins},row sep=1.2cm,labels=description]
C_n|\cU_\ell|D_m
	\ar[rr, "\phi^-|\phi^+|1+\phi^{+}|\phi^-|1"]
	\ar[dr,
	"1|\phi^+|\phi^-+1|\phi^{-}|\phi^+"
	]
	\ar[ddd, "1|\phi^-|1"]
	\ar[ddddrrr,dotted]
	\ar[rrddd,dashed]
&[.7 cm]\,
&C_{n+1}|\cU_{\ell+1}|D_m
	\ar[ddd, "1|\phi^-|1"]
	\ar[dr,"1|\phi^+|\phi^-+1|\phi^{-}|\phi^+"]
&[.7 cm]\,
\\
&[.7 cm] 
C_n|\cU_{\ell+1}|D_{m+1}
	\ar[rr, crossing over,"\phi^-|\phi^+|1+\phi^{+}|\phi^-|1"]
	\ar[dddrr,dashed, ]
&\,
&[.7 cm] C_n|\cU_{\ell+2}|D_{m+1}
	\ar[from=ulll, dashed, ]
	\ar[ddd, "1|\phi^-|1"]
\\
\\
C_n|\cU_{\ell+1}|D_{m}
	\ar[dr,"1|\phi^+|\phi^-+1|\phi^{-}|\phi^+"]
	\ar[drrr,dashed]	
	\ar[rr,"\phi^-|\phi^+|1+\phi^{+}|\phi^-|1"]
&[.7 cm]\,
&
C_{n+1}|\cU_{\ell+2}|D_{m}
	\ar[dr, "1|\phi^+|\phi^-+1|\phi^{-}|\phi^+"]
&[.9 cm]\,\\
& [.7 cm]C_n|\cU_{\ell+2}|D_{m+2}
	\ar[from=uuuul, dashed, ]
	\ar[from=uuu, "1|\phi^-|1"]
	\ar[rr, "\phi^-|\phi^+|1+\phi^{+}|\phi^-|1"]&\,
&[.9 cm]C_{n+1}|\cU_{\ell+3}|D_{m+1}
\end{tikzcd}
 \]
 Here, the length 2 arrows on the top and bottom faces are the same as in Equation~\eqref{eq:iterated-derived-tensor-product}, and the length 2 arrows along the sides are similar. The length 3 arrow counts tensors of solutions, such as $\phi^+|p^{[-,-,-]}|\phi^-$, where $p^{[-,-,-]}$ counts solutions which relate different orderings of three handles. Concretely, $p^{[-,-,-]}$ counts solutions over a 2-dimensional permutahedron (i.e. a hexagon), whose faces correspond to the terms in the length 3 hypercube relation. See \cite{BloomSpectralSequence} for more on constructing families of metrics indexed by permutohedra. 
 
 By the same logic which transformed Equation~\eqref{eq:iterated-derived-tensor-product} into~\eqref{eq:iterated-derived-tensor-product-Hom}, we see that the right hand side of Equation~\eqref{lem:U-action-A_infty-tensor-product} is homotopy equivalent as a hypercube to a hypercube of the following form:
 \[
 \begin{tikzcd}[column sep={3.2cm,between origins},row sep=1.2cm,labels=description]
 H_*C_n|H_*\cU_\ell|H_*D_m
 	\ar[rr, "\phi^-|\phi^+|1+\phi^{+}|\phi^-|1"]
 	\ar[dr,
 	"1|\phi^+|\phi^-+1|\phi^{-}|\phi^+"
 	]
 	\ar[ddd, "1|\phi^-|1"]
 	\ar[ddddrrr,dotted]
 &[.7 cm]\,
 &H_*C_{n+1}|H_*\cU_{\ell+1}|H_*D_m
 	\ar[ddd, "1|\phi^-|1"]
 	\ar[dr,"1|\phi^+|\phi^-+1|\phi^{-}|\phi^+"]
 &[.7 cm]\,
 \\
 &[.7 cm] 
 H_*C_n|H_*\cU_{\ell+1}|H_*D_{m+1}
 	\ar[rr, crossing over,"\phi^-|\phi^+|1+\phi^{+}|\phi^-|1"]
 &\,
 &[.7 cm] H_*C_n|H_*\cU_{\ell+2}|H_*D_{m+1}
 	\ar[ddd, "1|\phi^-|1"]
 \\
 \\
 H_*C_n|H_*\cU_{\ell+1}|H_*D_{m}
 	\ar[dr,"1|\phi^+|\phi^-+1|\phi^{-}|\phi^+"]	
 	\ar[rr,"\phi^-|\phi^+|1+\phi^{+}|\phi^-_*|1"]
 &[.7 cm]\,
 &
 H_*C_{n+1}|H_*\cU_{\ell+2}|H_*D_{m}
 	\ar[dr, "1|\phi^+|\phi^-+1|\phi^{-}|\phi^+"]
 &[.9 cm]\,\\
 & [.7 cm]H_*C_n|H_*\cU_{\ell+2}|H_*D_{m+2}
 	\ar[from=uuu,crossing over, "1|\phi^-|1"]
 	\ar[rr, "\phi^-|\phi^+|1+\phi^{+}|\phi^-|1"]&\,
 &[.9 cm]H_*C_{n+1}|H_*\cU_{\ell+3}|H_*D_{m+1}
 \end{tikzcd}
 \]
That is, we can eliminate all of the length 2 arrows in the complex (though potentially not the length 3 map). By an identical argument to Lemma~\ref{lem:U-action-A_infty-tensor-product}, we see that the induced map from the top to the bottom of the  cube has the same action on cycles as $1|\phi^-|1$, so these coincide on homology. This completes the proof.

\section{On the bypass exact triangle}
\label{sec:bypass-exact-triangle}

We now consider the bypass exact triangle in instanton theory, see \cite{LiLimits}*{Theorem~1.1}. This takes the form
\begin{equation}
\begin{tikzcd}[column sep=-.5cm] \KHI^-(Y,K)\ar[rr, "U"]&& \KHI^-(Y,K) \ar[dl]\\
& \widehat{\KHI}(Y,K)\ar[ul] 
\end{tikzcd}
\label{eq:bypass-overview}
\end{equation}
The above triangle is  the surgery exact triangle for a surgery in a suitable choice of closure of $Y\setminus \nu(K)$.

The arrow labeled $U$ corresponds to a 2-handle map in the exact triangle, however this 2-handle map is by definition the same as the action of $U$ on the direct limit model of knot Floer homology. The other maps in the above exact triangle do not admit as immediate an algebraic description in terms of the module structure on $\KHI^-(Y,K)$. In this section we describe the other maps in terms of the $\C[U]$-module structure on $\KHI^-(Y,K)$.

To describe the other maps, we first pick a free resolution $C^-(K)$ of $\KHI^-(Y,K)$ over $\C[U]$. There is a short exact sequence
\[
\begin{tikzcd}
0\ar[r] &C^-(K)\ar[r, "U"] & C^-(K)\ar[r, "\pi"]& C^-(K)/U C^-(K)\to 0.
\end{tikzcd}
\]
Taking homology, we obtain an exact triangle of the same shape as Equation~\eqref{eq:bypass-overview}. We call this triangle the \emph{algebraic} exact triangle for $\KHI^-(Y,K).$

\begin{thm}
\label{thm:exact-triangle}The bypass exact triangle in Equation~\eqref{eq:bypass-overview} is isomorphic to the algebraic exact triangle for $\KHI^-(Y,K).$
\end{thm}

Before proving Theorem~\ref{thm:exact-triangle}, it is helpful to give another description of the exact triangle.
We begin by considering the short exact sequence
\[
\begin{tikzcd}
0\ar[r]& \C[U] \ar[r, "U"]& \C[U] \ar[r, "U\mapsto 0"]& \C[U]/U\ar[r]& 0
\end{tikzcd}
\]
This gives us two quasi-isomorphisms of $\C[U]$-modules
\[
\Cone\left(U\colon \C[U]\to \C[U]\right)\to\C[U]/U\quad
\]
\[
\C[U]\to \Cone\left(U\mapsto 0\colon \C[U]\to \C[U]/U\right).
\]
We will write $\pi$ for the map $U\mapsto 0$.
There is, in fact, a third equivalence which is more involved to write down. To state the equivalence, it convenient to work in the setting of $A_\infty$-modules instead of $\C[U]$-chain complexes. This takes the form of an $A_\infty$-module map $f_j\colon \C\to \C[U]$, given by $f_j=0$ if $j\neq 2$ and 
\[
f_2(U^i, 1)=\begin{cases} U^{i-1} & \text{ if } i>0\\
0& \text{ if } i=0.
\end{cases}
\]
The above morphism may be found by inverting the morphism $\Cone(U\colon \C[U]\to \C[U])\to \C[U]/U$ as an $A_\infty$-morphism (e.g. using the homological perturbation lemma). The map $f_*$ is the component mapping from $\C[U]/U$ to the domain part of $\Cone(\C[U]\to \C[U])$.
We observe that there is an isomorphism of $A_\infty$-modules
\[
\C[U]\simeq \Cone(f_*\colon \C[U]/U\to \C[U]).
\]
In fact, the right hand side is a genuine $\C[U]$-module isomorphic to $\C[U].$

\begin{lem} The algebraic exact triangle for $\KHI^-(Y,K)$ is isomorphic to the following exact triangle:
\[
\begin{tikzcd}[column sep=-.5cm] \KHI^-(Y,K)\tildeotimes \C[U]\ar[rr, "\id\tildeotimes U"]&& \KHI^-(Y,K)\tildeotimes \C[U] \ar[dl, "\id\tildeotimes \pi"]\\
& \KHI^-(Y,K)\tildeotimes \C[U]/(U)\ar[ul, "\id\tildeotimes f_*"] 
\end{tikzcd}
\]
\end{lem}
We leave the above lemma to the reader, since it is straightforward to prove directly. 

We observe that each of the modules appearing in the exact triangle admit an Alexander grading where $U$ has grading $-1$, and each of the maps in the sequence is homogeneously graded. By shifting the first copy of $\KHI^-(Y,K)\tildeotimes \C[U]$ down by 1, the maps in the exact triangle become Alexander grading preserving. We have the following basic algebraic lemma:

\begin{lem} Suppose $A$, $A'$, $B$, $B'$, $C$ and $C'$ are finitely generated $\C[U]$ modules which are Alexander graded. Suppose that we have two exact triangles of $\C[U]$-modules
\[
\begin{tikzcd}[column sep=1cm] A\ar[rr]&& B \ar[dl]\\
& C\ar[ul] 
\end{tikzcd}\quad \begin{tikzcd}[column sep=1cm] A'\ar[rr]&& B' \ar[dl]\\
& C'\ar[ul] .
\end{tikzcd}
\]
Suppose further that for each $q\in \Z$, the truncations in gradings $>q$ are isomorphic. Then the two triangles are isomorphic.
\end{lem}
\begin{proof} The assumptions guarantee that the exact triangles split over Alexander gradings. Finite generation implies that the subspaces in grading $q$ are isomorphic to those in grading $q-1$ for $q\ll 0$. Hence, we may merely shift the isomorphism in some grading $q\ll 0$ into grading $q-1$ to extend the isomorphism further.
\end{proof}

For the proof, we consider the sutured manifold for the knot complement of $K$ with $n$-framed longitudinal sutures for some $n\gg 0$. We take the connected sum of this with an $\ell$-framed unknot, and apply the maps for a negative bypass triangle to the unknot.  This gives us an exact triangle
\[
\begin{tikzcd}[column sep=-.5cm]
H_*\Cone( C_n|\cU_\ell\to C_{n+1}|\cU_{\ell+1})\ar[rr]&& H_*\Cone( C_n|\cU_{\ell+1}\to C_{n+1}|\cU_{\ell+2}) \ar[dl]\\
& H_*\Cone( C_n|\cU_\mu\to C_{n+1}|\cU_{\mu})\ar[ul] 
\end{tikzcd}
\]
(The map in each Cone is $\phi^+|\phi^-+\phi^-|\phi^+$). 
 We apply the homological perturbation lemma, as before, to take homology of each $C_n$ and $\cU_\ell$. We get a diagram of the form
\[
\begin{tikzcd}[column sep=-2.5cm] \Cone( H_*C_n|H_*\cU_\ell\to H_*C_{n+1}|H_*\cU_{\ell+1})\ar[rr, "F"]&& \Cone(H_* C_n|H_*\cU_{\ell+1}\to H_*C_{n+1}|H_*\cU_{\ell+2}) \ar[dl, "G"]\\
& \Cone(H_*C_n|H_*\cU_\mu\to H_*C_{n+1}|H_*\cU_{\mu})\ar[ul, "H"] 
\end{tikzcd}
\]
which is exact after taking homology of each vertex.

If $q$ is fixed and $n\gg 0$, then the top two complexes are identified (in Alexander gradings $>q$) with the derived tensor product $\HFK^-(Y,K)\tildeotimes \C[U]$. On the bottom row, we observe that $\phi^+$ vanishes on $H_*\cU_\mu$, so the bottom row may be identified on gradings $>q$ with $\KHI^-(Y,K)\tildeotimes \C[U]/U$.

\begin{lem} Suppose $q$ is fixed and $n,\ell\gg 0$. When restricted to Alexander gradings $>q$, we have the following:
\begin{enumerate}
\item The map $F$ is equal on homology to $\id\tildeotimes U$. 
\item The map $G$ is $\id\tildeotimes \pi$.
\item The map $H$ is $\id\tildeotimes f_*$.
\end{enumerate}
\end{lem}
\begin{proof} The computations for $F$ and $G$ are no different than our proof of the $U$-action on connected sums, so we omit them and focus on the map $H$. In this case, we are considering the diagram
\[
\begin{tikzcd}
[column sep=3.5cm, row sep=2cm] C_n|\cU_\mu
	\ar[r, "\phi^+|\phi^-+\phi^-|\phi^+"]
	\ar[d, "1|F_W"]
	\ar[dr, dashed, "\phi^+|h^{[-,W]}+\phi^-|h^{[+,W]}", labels= description]
&
C_{n+1}|\cU_{\mu}
	\ar[d, "1|F_W"]
\\
C_n|\cU_\ell\ar[r,"\phi^+|\phi^-+\phi^-|\phi^+"]
& C_{n+1}|\cU_{\ell+1}
\end{tikzcd}
\]
We may now apply homological perturbation to the above diagram. The map $F_{W}$ vanishes on homology, since it corresponds to the map $\widehat{\KHI}(u)$ to $\KHI^-(u)$ in the bypass exact triangle for the unknot. Similarly, the maps $h^{[-,W]}$ and $h^{[+,W]}$ are homogeneously Alexander grading. On the other hand $\cU_\mu$ is supported in Alexander grading 0, while $\cU_\ell$ has top Alexander grading $0$. The map $h^{[-,W]}$ has Alexander grading $+1$ while $h^{[+,W]}$ has Alexander grading $0$. In particular, after applying homological perturbation, $h^{[-,W]}=0$, while $h^{[+,W]}$ is either 0, or maps $1$ to $1$. Note also that the computation is universal in the sense that it only depends on $\mu$ and $\ell$, but not $K$ or $n$.

The possibility that $h^{[+,W]}=0$ is impossible, since it would imply that the exact sequence splits for all knots, which is well known not the case (e.g. consider the right handed trefoil). In particular, we must have $h^{[+,W]}(1)=1$.

However, the diagram
\[
\begin{tikzcd}
[column sep=3.5cm, row sep=2cm] 
H_*C_n|H_*\cU_\mu
	\ar[r, "\phi^+|1"]
	\ar[d, "0"]
	\ar[dr, dashed, "\phi^-|1", labels= description]
&
H_*C_{n+1}|H_*\cU_{\mu}
	\ar[d, "0"]
\\
H_*C_n|H_*\cU_\ell\ar[r,"\phi^+|\phi^-+\phi^-|\phi^+"]
& H_*C_{n+1}|H_*\cU_\ell
\end{tikzcd}
\]
is equal (in gradings above $q$) to $\id\tildeotimes f_*$. 
\end{proof}

\section{Free models}
\label{sec:limit-free}

As described in the introduction, our techniques give a model of $\KHI^-(Y,K)$ which is free and finitely generated over $\C[U]$:
\[
\mathit{CKI}^-(K,n):=\Cone\left(\begin{tikzcd}[column sep=2cm] \SHI(Y\setminus \nu(K),\Gamma_n)\otimes_{\C} \C[U]\ar[r, "\phi^-|\id-\phi^+|U"] & \SHI(Y\setminus \nu(K),\Gamma_{n+1})\otimes_{\C} \C[U]
\end{tikzcd}\right).
\]

\begin{rem} Although $\phi^-$ and $\phi^+$ are only well-defined up to multiplication by a unit in $\C,$ changing the maps $\phi^-$ and $\phi^+$ by a unit does not change the $\C[U]$-module isomorphism type of the cone, cf. Lemma~\ref{lem:rescale-module}.
\end{rem}

\begin{thm}\label{thm:isomorphism-free-model}
 There is an isomorphism of $\C[U]$-modules
\[
H_*\mathit{CKI}^-(K,n)\iso \KHI^-(Y,K)
\]
for all $n$. 
\end{thm}
\begin{proof} The idea is to apply our proof of the connected sum formula to the case that one knot is $K$, and the other is an unknot $u$. We let the framing of $K$ stay fixed while we let the framing of $u$ approach $\infty$.

In more detail, fix $n\in \Z$. We apply Lemma~\ref{lem:reorganize-cone} when $K$ is a knot given longitudinal framing $n$, and $K'=u$ is an $\ell$-framed unknot, where $\ell\gg 0$. We obtain that $\SHI(Y\setminus \nu(K),\Gamma_{n+\ell})$ is homotopy equivalent to the diagram
\[
\begin{tikzcd}[column sep=2cm, row sep=2cm]
C_{n-1}\otimes \cU_\ell\oplus C_n\otimes \cU_{\ell-1}
    \ar[r, "\phi^-|1\oplus 1|\phi^-"]
    \ar[d, "\phi^+|1\oplus 1|\phi^+"]
    &
C_n\otimes \cU_\ell\\
C_n\otimes \cU_\ell
\end{tikzcd}
\]
where $C_n=\SHI(Y\setminus \nu(K), \Gamma_n)$ and $\cU_\ell=\SHI(Y\setminus \nu(u), \Gamma_\ell)$. We observe that
\[
\Cone(\phi^+|1\oplus 1|\phi^+\colon C_{n-1}\otimes \cU_\ell \oplus C_n\otimes \cU_\ell)_{\ge q}\simeq (C_{n-1}\otimes \cU_{\ell})_{\ge q}
\]
if $\ell$ is suitably large. Note that this does not depend on the value of $n$ except that $\ell$ must be suitably large relative to $n$ and $q$. This may be seen since the subcomplex 
\[
\Cone(1|\phi^+\colon C_n\otimes \cU_{\ell-1} \to C_n\otimes \cU_{\ell})_{\ge q}
\]
is acyclic as $\phi^+\colon (\cU_{\ell-1})_{\ge q}\to (\cU_{\ell})_{\ge q}$ is an isomorphism.

In particular, if $\ell \gg 0$, then the truncation of the above complex in Alexander gradings at least $q$ is homotopy equivalent to the subcomplex $\mathit{CKI}^-(K,n)_{\ge q}$. Letting $\ell\to \infty,$ we obtain that $\SHI(Y\setminus \nu(K), \Gamma_{n+\ell})_{\ge q}$ is quasi-isomorphic to the  truncation 
\[
\Cone(\phi^+|\phi^-+\phi^-|\phi^+\colon C_{n}\otimes \cU_{\ell}\to C_{n+1}\otimes  \cU_{\ell+1})_{\ge q}.
\]
The above is identical to the truncation $\mathit{CKI}^-(K,n)_{\ge q}$, showing the isomorphism of the statement on the level of graded vector spaces.

The claim that the isomorphism is of $\C[U]$-modules follows from the same logic as our proof of Theorem~\ref{thm:U-action}.
\end{proof}

As an immediate corollary:

\begin{cor}Theorem~\ref{thm:isomorphism-free-model} also holds in the Heegaard Floer setting.
\end{cor}

\subsection{Free models for links}
\label{sec:free-models-for-links}

The above free model of the instanton knot Floer complex extends in an entirely formal manner to the case of links. We now describe the model in the case of a two component link. In this case, write $\Gamma_{n_1,n_2}$ for a framing consisting of an $n_1$ and $n_2$ framed longitude on the two link components. To state the formula, we write $C_{n_1,n_2}=\SHI(Y\setminus \nu(L),\Gamma_{n_1,n_2})$. Our model takes the form:
\[
\mathit{CLI}^-(L,n_1,n_2)=
\begin{tikzcd}[column sep=3.5cm, row sep=3.5cm, labels=description]
C_{n_1,n_2}\otimes \C[U_1,U_2]
    \ar[r, " \phi^+_1|U_1+\phi^-_1|1"] 
    \ar[d, "\phi^+_2|U_2+\phi^-_{2}|1"]
    \ar[dr, "\substack{ \phi^{[+_1,+_2]}|U_1U_2\\
    +\phi^{[+_1, -_2]}|U_1 \\
    +\phi^{[-_1, +_2]}|U_2 \\
    +\phi^{[-_1,-_2]}|1}" ]
&C_{n_1+1,n_2}\otimes \C[U_1,U_2]
    \ar[d, "\phi^+_2|U_2+\phi^-_{2}|1"]
    \\
C_{n_1,n_2+1}\otimes \C[U_1,U_2]
    \ar[r, " \phi^+_1|U_1+\phi^-_1|1"] &
C_{n_1+1,n_2+1}\otimes \C[U_1,U_2]
\end{tikzcd}
\]
We form the above by considering the iterated mapping cone construction on the natural instanton Floer complexes. Then we apply the hypercube eigenspace techniques of Section~\ref{sec:hypercubes-eigenspaces} to obtain a complex involving the sutured instanton complexes. We then take the homology at each vertex using the homological perturabtion lemma for hypercubes (cf. \cite{HHSZDual}*{Lemma~2.10}).

The above description generalizes in a straightforward way to the setting of an $n$-component link in a homology 3-sphere.

For brevity, we will not attempt to prove the following:

\begin{conj} If $L$ is an $n$-component link in a homology 3-sphere and $\ve{n}$ is a choice of framing, then the homotopy type of $\mathit{CLI}^-(L,\ve{n})$ is well-defined and independent of $\ve{n}$. Furthermore, the homologies are isomorphic as $\C[U_1,\dots, U_n]$-modules to the direct limit construction from \cite{GhoshLiDecomposing}*{Section~6.2}. Finally, these link groups are tensorial under connected sum of links, and also satisfy the Skein exact triangle, as stated in Conjecture~\ref{conj:intro}.
\end{conj}

 Note that to take derived tensor products with the modules from \cite{GhoshLiDecomposing}, one would need to equip them with $A_\infty$-module structures since they are not free, whereas the complexes $\mathit{CLI}^-(L,\ve{n})$ defined above are free, and hence suitable for taking tensor products. Assuming the above conjecture, our free models $\mathit{CLI}^-(L,\ve{n})$ would equip the direct limit models from \cite{GhoshLiDecomposing} with an $A_\infty$-module structure via the homological perturbation lemma for $A_\infty$-modules.

\subsection{Free models in Heegaard Floer theory}

\label{sec:free-models-Heegaard}

In this section, we prove that the analog of Theorem~\ref{thm:isomorphism-free-model} holds in Heegaard Floer theory. Our proof is based on reinterpretting the work of Etnyre, Vela-Vick and Zarev.

We begin by recalling components of Etnyre, Vela-Vick and Zarev's proof  \cite{EVZLimit} of the isomorphism between $\HFK^-(Y,K)$ and the direct limit construction of knot Floer homology. Their proof is a local computation using bordered Floer homology \cite{LOTBordered} and bordered sutured Floer homology \cite{ZarevBorderedSutured}. We recall the Lipshitz, Ozsv\'{a}th, Thurston torus algebra, which we denote $\cA$. This algebra has two idempotents, $\iota_1$ and $\iota_2$, and is the path algebra with relations
\[
\begin{tikzcd}[column sep=2cm] \iota_1 \ar[r,bend left, "\rho_1"] \ar[r, bend right, "\rho_3",swap]& \iota_2 \ar[l, "\rho_2"]
\end{tikzcd}/ (\rho_2\rho_1=\rho_3\rho_2=0).
\]

We will let ${}^{\cA} M_n$ denote the bordered sutured type-$D$ module for $(T^2\setminus D^2)\times [0,1]$ considered by Etnyre, Vela-Vick and Zarev. The exact sutures and parametrizations are not important, but essentially $(\bT^2\setminus D^2)\times \{0\}$ is parametrized as in \cite{LOTBordered}, while on $(\bT^2\setminus D^2)\times \{1\}$ has two parallel sutures, of slope $n$ relative to the bordered framing. One of these sutures intersects $\d D^2$ and runs to the side of $(T^2\setminus D^2)\times \{0\}$. See \cite{EVZLimit}*{Figure~27} for a schematic.

The exact sutures and parametrizations are not important for our present formulation. Rather, what is important is that if $\widehat{\mathit{CFA}}(Y\setminus \nu(K)){}_{\cA}$ is the type-$A$ module for a knot complement, then 
\[
\widehat{\mathit{CFA}}(Y\setminus \nu(K)){}_{\cA}\boxtimes {}^{\cA} M_n\simeq \SFH(Y\setminus \nu(K),\Gamma_n).
\]

Etnyre, Vela-Vick and Zarev define positive and negative bypass maps. These take the form of type-$D$ module maps 
\[
\phi_+^1,\phi_-^1 \colon {}^{\cA}M_n\to {}^{\cA} M_{n+1}.
\]

There is a natural doubly pointed Heegaard knot diagram for a fiber $S^1\subset S^1\times D^2$. See \cite{EVZLimit}*{Figure~18}. We will write ${}^{\cA}\scK^{\bF[U]}$ for the corresponding type-$DD$ module. This module has a single generator $\xs$ and has differential given by
\[
\delta^1(\xs)=\rho_{23} \otimes \xs\otimes U.
\]
For our purposes, the  important property of this module is that
\[
\widehat{\mathit{CFA}}(Y\setminus \nu(K)){}_{\cA}\boxtimes {}^{\cA} \scK^{\bF[U]}\simeq \HFK^-(Y,K)^{\bF[U]}.
\]

Etnyre, Vela-Vick and Zarev show that
\[
\injlim {}^{\cA} M_n\iso {}^{\cA} \scK^{\bF[U]}\boxtimes \bF[U]
\]
where the $\bF[U]$ action on the direct limit is induced by the negative bypass maps $\phi_-^1$.  It is helpful also to use extension of scalars to view each $M_n$ as a type-$DD$ module over ${}^{\cA} M_n^{\bF[U]}$ (so that $\delta^1$ has non non-zero $U$-powers).

  We now prove the following analog of  \cite{EVZLimit}*{Theorem~1.1}, which does not involve direct limits:

\begin{prop} For any $n$,  there is a homotopy equivalence
\[
{}^{\cA} \scK^{\bF[U]}\simeq \Cone(\phi^1_+\otimes U+\phi^1_-\otimes 1\colon {}^{\cA} M_n^{\bF[U]}  \to {}^{\cA} M_{n+1}^{\bF[U]})
\]
\end{prop}
\begin{proof} Let us write ${}^{\cA} C^{\bF[U]}$ for the mapping cone complex in the right hand side of the statement. We use the computation of \cite{EVZLimit}*{Section~7.2}. Note first that tensoring with the bimodule for changing the $\cA$ boundary parametrization commutes with the positive and negative bypass maps. Hence, it is sufficient to show the claim for $n=0$.

 The type-$D$ module $M_0$ has a single generator $a$, concentrated in idempotent 2. The structure map $\delta^1$ vanishes. The module $M_1$ has two generators, $a$ and $b$, where $a$ is in idempotent 2 and $b$ is in idempotent 1. The structure map satisfies
 \[
 \delta^1(b)=\rho_2\otimes a\otimes 1
 \]
 See \cite{EVZLimit}*{Section~7.2.1}.
 
 Etnyre, Vela-Vick and Zarev  also compute the positive and negative bypass maps \cite{EVZLimit}*{Section~7.3}. These are given by the following formulas (as maps from $M_0$ to $M_1$):
 \[
 \phi_+^1(a)=\iota_2\otimes a\quad\text{and} \quad \phi_-^1(a)=\rho_2\otimes b.
 \]
 
 Hence, the mapping cone in the statement becomes
 \[
 \begin{tikzcd}[column sep=2cm, row sep=1cm, labels=description]
 &a_0
 	\ar[d, "1"]
 	\ar[dl, "\rho_3 \otimes U"]
 \\
 b_1
 	\ar[r, "\rho_2"]
 & a_1
 \end{tikzcd}
 \]
 (In the above, the subscripts of $a_i$ and $b_i$ denote whether they are from $M_0$ or $M_1$).  A homotopy equivalence
 \[
 \begin{tikzcd}
 {}^{\cA} \scK^{\bF[U]} \ar[r, shift left, "I"]& {}^{\cA} C^{\bF[U]}\ar[l, shift left, "\Pi"] \ar[loop right, "H"] 
 \end{tikzcd}
 \]
  is given as follows:
 \[
 \begin{split}
	 I(b_1)&=1\otimes b_1\otimes 1+\rho_2 \otimes a_0\otimes \\
	 \Pi(b_1)&=1\otimes b_1\otimes 1\\
	 \Pi(a_1)&=\rho_3 \otimes b_1\otimes U\\
	 \Pi(a_0)&=0\\
	 H(a_1)&=1\otimes a_0\otimes 1\\
	 H(a_0)&=H(b_1)=0.
	 \end{split}
 \]
 One computes easily that $\Pi\circ I=\id$ and $I\circ \Pi=\id+\d(H)$. 
\end{proof}

\bibliographystyle{custom}
\def\MR#1{}
\bibliography{biblio}

\end{document}